\documentclass[11pt]{amsart} 
\usepackage{amssymb}
\usepackage{verbatim}
\usepackage{hyperref}
\usepackage{color}

\usepackage{tikz}

\newtheorem{theorem}{Theorem}    
\newtheorem{proposition}[theorem]{Proposition}
\newtheorem{corollary}[theorem]{Corollary}
\newtheorem{lemma}[theorem]{Lemma}

\newtheorem{conjecture}[theorem]{Conjecture}
\newtheorem{definition}{Definition}
\newtheorem{question}[definition]{Question}

\newtheorem{example}[definition]{Example}

 \numberwithin{theorem}{section}
 \numberwithin{definition}{section}
 \numberwithin{equation}{section}

\def\bff{\mathbf f}

\def\bx{\mathbf x}
\def\bh{\mathbf h}
\def\bk{\mathbf k}
\def\bn{\mathbf n}
\def\bg{\mathbf g}

\def\bs{\mathbf s}
\def\bp{\mathbf p}
\def\by{\mathbf y}

\def\bm{\mathbf m}
\def\bxi{{\mathbf \xi}}

\def\by{\mathbf y}
\def\bz{\mathbf z}

\newcommand{\norm}[1]{ \|  #1 \|}

\def\scriptd{{\mathcal D}}
\def\scripts{{\mathcal S}}

\def\scriptf{{\mathcal F}}
\def\scriptm{{\mathcal M}}

\def\scripte{{\mathcal E}}
\def\scripta{{\mathcal A}}

\def\scriptastar{{\mathcal A}^*}

\DeclareFontFamily{U}{mathx}{\hyphenchar\font45}
\DeclareFontShape{U}{mathx}{m}{n}{
	<5> <6> <7> <8> <9> <10>
	<10.95> <12> <14.4> <17.28> <20.74> <24.88>
	mathx10 
}{}
\DeclareSymbolFont{mathx}{U}{mathx}{m}{n}
\DeclareFontSubstitution{U}{mathx}{m}{n}
\DeclareMathAccent{\widecheck}{0}{mathx}{"71}

\def\bk{\mathbf k}
\def\kernel{\operatorname{kernel}}
\def\eps{\varepsilon}

\def\reals{\mathbb R}
\def\naturals{\mathbb N}
\def\integers{\mathbb Z}

\def\one{\mathbf 1}
\def\complex{{\mathbb C}\/}

\def\lt{L^2}
\def\three{\{1,2,3\}}
\def\distance{\operatorname{distance}}

\begin{document}
\setcounter{tocdepth}{1}

\title [Trilinear oscillatory integral inequalities] 
{On Trilinear Oscillatory Integral Inequalities \\ and  Related Topics}

 \author{Michael Christ}

\address{
        Michael Christ\\
        Department of Mathematics\\
        University of California \\
        Berkeley, CA 94720-3840, USA}
\email{mchrist@berkeley.edu}

\date{February 16, 2022.}

\begin{abstract}
Inequalities are established for certain trilinear scalar-valued
functionals. These functionals act on measurable functions of one real variable,
are defined by integration over two-- or three--dimensional spaces,
and are controlled in terms of Lebesgue space norms of the functions,  
and in terms of negative powers 
of large parameters describing a degree of oscillation.
Related sublevel set inequalities are a central element of the analysis. 
\end{abstract}



\dedicatory{In memoriam Elias~M.~Stein} 
\thanks{Research supported in part by NSF grants DMS-13363724 
and DMS-1901413.}

\maketitle
\tableofcontents

\section{Introduction}

Oscillatory integral operators, inequalities governing them,
curvature, and the interrelationships between these topics
are pervasive themes in the work of E.~M.~Stein, as for instance in
\cite{steinwainger},\cite{steinbook},\cite{steinICMvideo}.
In the present paper, we
investigate scalar-valued multilinear oscillatory integral forms
\begin{equation} \label{multioscform1}
T_\lambda^\phi(\bff) = \int_{B} e^{i\lambda \phi(\bx)} 
\prod_{j\in J} f_j(x_j)\,d\bx, 
\end{equation}
along with related forms and inequalities.
Here $B$ is a ball or product of balls in $(\reals^d)^J$,
$J$ is a finite index set of cardinality $|J|\ge 2$,
$\bff = (f_j: j\in J)$ is a tuple of 
rather arbitrary functions $f_j:\reals^d\to\complex$,
$\bx = (x_j: j\in J)\in(\reals^d)^J$,
$\phi:(\reals^d)^J\to\reals$ is a $C^\infty$ function,
and $\lambda\in\reals$ is a large parameter.
More generally, one can form
\begin{equation}  \label{multioscform2}
S_\lambda(\bff) = \int_{B\subset\reals^D} e^{i\lambda \phi(\bx)} 
\prod_{j\in J} (f_j\circ\varphi_j)(\bx)\,d\bx, 
\end{equation}
with $B\subset\reals^D$ a ball,
$\varphi_j:B\to \reals^d$ smooth submersions,
and with the cardinality of $J$ finite but 
$|J|d$ possibly large relative to $D$.
We seek upper bounds, for $|T_\lambda^\phi(\bff)|$ and for $|S_\lambda(\bff)|$,
that are small when $|\lambda|$ is large,
require no smoothness hypothesis on $\bff$,
are uniform over a large class of $\bff$,
reflect cancellation due to oscillation of $e^{i\lambda\phi}$ when $|\lambda|$
is large, and also reflect the influence of geometric and algebraic
effects implicit in  $(\phi, (\varphi_j: j\in J))$.
In this paper we establish such bounds,
and deduce an application to related multilinear
forms in which no oscillatory factors are overtly present.

\subsection{Background}
Inequalities of the form
\begin{equation} \label{Lpversion}
|T_\lambda^\phi(\bff)| \le C|\lambda|^{-\gamma} \prod_{j\in J} \norm{f_j}_{L^{p_j}(\reals^d)},
\end{equation}
with $\gamma>0$ and $C<\infty$ dependent on $\phi$  and on $\eta$,
have been analyzed in various works.
H\"ormander \cite{hormander} 
established the fundamental upper bound $O(|\lambda|^{-d/2}\norm{f_1}_2\norm{f_2}_2)$
for the bilinear case $|J|=2$,
with the mixed Hessian $\frac{\partial^2 \phi(x,y)}{\partial x \, \partial y}$ 
everywhere nonsingular. 
The bilinear case, with $x,y$ in spaces of unequal dimensions,
has been intensively studied in connection with Fourier restriction inequalities.
Likewise, in connection with Fourier restriction,
multilinear forms have been investigated, in which each function $f_j$
is individually acted upon by a linear oscillatory integral operator,
and the product of the resulting functions is integrated.
Such multilinear forms are not studied here.

Forms $|T_\lambda^\phi(\bff)|$ and $|S_\lambda(\bff)|$,
with $|J|\ge 3$, have been studied by
Phong-Stein-Sturm \cite{PSS},
Gilula-Gressman-Xiao \cite{gressmanetal}, 
and others.
An introductory treatment can be found in the book \cite{steinbook} of Stein.
The works \cite{PSS} and \cite{gressmanetal} 
deal with general phase functions $\phi$, 
and seek optimal relationships between $\phi$, decay
exponents $\gamma$, and Lebesgue exponents $p_j$.
\cite{PSS} also emphasizes stability --- whether the optimal
exponent $\gamma$ is a lower semicontinuous function of $\phi$. 

The regime $|J| > D/d$ is singular,
in the sense that the integral extends only
over a positive codimension subvariety
of the Cartesian product of the domains of the functions $f_j$.
Variants of the form \eqref{multioscform2}, in the
singular regime and with all mappings $\varphi_j$ linear,
were investigated 
by Li, Tao, Thiele and the present author \cite{CLTT}. They 
established conditions under which there exists an exponent $\gamma>0$
for which a corresponding inequality holds. 
One of the results of the present paper relaxes the assumption of linearity.
Another treats certain cases with $\varphi_j$ linear that
were not treated in \cite{CLTT}, and provides
an alternative proof of one of the main results of that work.

Results of this type under lower bounds for certain partial derivatives
of the phase function, but with no upper bounds at all,
have been investigated by Carbery and Wright
\cite{carberywright} for $|J|\ge 3$, building on earlier work
\cite{CCW} for the bilinear case $|J|=2$.
This thread is not developed further in the present paper, in which
upper bounds are implicit through smoothness hypotheses on phase functions.

For certain ranges of exponent tuples $\bp = (p_j: j\in J)$,
the works \cite{PSS} and \cite{gressmanetal} establish upper bounds for \eqref{multioscform1}
with optimal exponents $\gamma$ as $|\lambda|\to\infty$,
up to powers of $\log(1+|\lambda|)$, 
where ``optimal'' means largest possible under indicated hypotheses on $\phi$
for given $\bp$.
We will not review the hypotheses of those works precisely,
but their general form is significant for our discussion.
For each $\alpha$ 
(with $\alpha_j\ne 0$ for at least two distinct indices $j$)
there are certain parameters $\bp,\gamma$ for which nonvanishing of 
$\partial^{\alpha} \phi\,/\,\partial \bx^\alpha$ at $\bx_0$
implies validity of
\eqref{Lpversion}, with $B=B(\bx_0,r)$ for sufficiently small $r>0$.
This conclusion is independent of
other coefficients in the Taylor expansion of $\phi$ about $\bx_0$.
Thus any other nonvanishing coefficients imply corresponding inequalities, 
and interpolation of the resulting inequalities yields further inequalities. 
All bounds obtained in these cited works are consequences of bounds obtained in this way, 
together with inclusions among $L^p$ spaces resulting from H\"older's inequality.

\subsection{Four questions}

\begin{question} \label{question1}
Let $\phi$ be a
real analytic, real-valued 
phase function. Let $\bx_0\in (\reals^d)^J$. 
For which $\gamma>0$ do there exist $C<\infty$ and 
a neighborhood $B$ of $\bx_0$ such that
\begin{equation} \label{Linftyversion1}
|T_\lambda^\phi(\bff)| \le C|\lambda|^{-\gamma} \prod_{j\in J} \norm{f_j}_{L^\infty}
\ \text{ for all functions $f_j\in L^\infty$,} 
\end{equation}
for every $\lambda\in\reals$?
\end{question}

For multilinear forms of the more general type \eqref{multioscform2},
a less precise question is at present appropriate.
\begin{question} \label{question1b}
Let $\phi$ be a real analytic, real-valued 
phase function. 
Let $\varphi_j:\reals^D\to\reals^d$ be real analytic submersions. 
Let $\bx_0\in \reals^D$. 
Under what conditions on $\phi$ and on $(\varphi_j: j\in J)$
does there exist
$\gamma>0$ 
such that
\begin{equation} \label{Linftyversion2}
|S_\lambda(\bff)| \le C|\lambda|^{-\gamma} \prod_{j\in J} \norm{f_j}_{L^\infty}
\ \text{ for all functions $f_j\in L^\infty$,} 
\end{equation}
for every $\lambda\in\reals$?
\end{question}
Roughly speaking, the difficulty in establishing inequalities
\eqref{Linftyversion2} increases as the ratio $d/D$ increases.

In the formulation \eqref{Linftyversion1}, 
the main structural hypothesis is that the 
nonoscillatory part of the
integrand is a product of factors $f_j(x_j)$. No smoothness
is required of these factors, and
the strongest possible size restriction, $L^\infty$, is imposed. 
As a refinement, one could ask to what degree the $L^\infty$
norms could be replaced by weaker $L^{p_j}$ norms
without reducing $\gamma$. 
One focus of the present paper is on obtaining a comparatively large
exponent $\gamma$, rather than on weakening the hypotheses
under which it is obtained.
Although we are not able to determine optimal exponents $\gamma$ in Question~\ref{question1},
we improve on the largest exponent previously known 
for generic real analytic phases in the trilinear case with  $d=1$.
We show that interactions between monomial terms
can give rise to upper bounds not obtainable from monomial-based inequalities. 

A second focus, for related functionals,
is on obtaining some decay inequality of power law type,
for situations in which no decay bound was previously known,
without attention to the value of the exponent $\gamma$. 
See for instance Theorem~\ref{thm:3}.

Oscillation can arise implicitly through the presence of
high frequency Fourier components in the factors $f_j$,
instead of explicitly through the presence of overt oscillatory
factors $e^{i\lambda \phi}$. This suggests a third question.

\begin{question} \label{question:implicitoscillation}
Let $\eta$ be smooth and compactly supported. Let $J$ be a finite index set.
Let $\varphi_j:\reals^D\to\reals^d$ be real analytic submersions.
Under what circumstances can the quantity 
$\int \prod_{j\in J} (f_j\circ\varphi_j)(x)\,\eta(x)\,dx$
be majorized by a product of strictly negative
Sobolev norms of the factors $f_j$?
\end{question}

In analyzing these questions about oscillatory integral forms,
we are led to questions about sublevel sets.
Let $\varphi_j:[0,1]^2\to\reals^1$ and $a_j:[0,1]^2\to\reals^1$ be real analytic.
To an ordered triple $\bff$ of Lebesgue measurable $f_j:\reals\to\reals$, and to
 $\eps>0$, associate the sublevel set
\begin{equation} \label{sublset}
S(\bff,\eps) = \{x\in[0,1]^2: \big| \sum_{j=1}^3 a_j(x) (f_j\circ\varphi_j)(x)\big|
<\eps\}.
\end{equation}

\begin{question} \label{question:sublevel}
Under what hypotheses
on $(\varphi_j,a_j: j\in\three)$
and what conditions on $\bff$
do there exist $\gamma>0$ and $C<\infty$
such that for every small $\eps>0$,
\begin{equation} \label{ineq/sublevel}
|S(\bff,\eps)| \le C\eps^\gamma?
\end{equation}
\end{question}

Some condition on $\bff$ is needed to exclude 
trivial solutions with $\bff\equiv 0$
or with every $|f_j|$ small pointwise.
Another necessary condition for an inequality \eqref{ineq/sublevel}
is that any exact smooth solution
$\bff$ of $\sum_j a_j\cdot (f_j\circ\varphi_j)\equiv 0$,
in an arbitrary nonempty open set, should vanish identically. 
Situations in which there is a small family
of such exact solutions, e.g. constant $\bff$
or affine $\bff$, are also of interest. In such
a situation, one asks instead whether
the inequality \eqref{ineq/sublevel}
can fail to hold only for those $\bff$ that are closely approximable by 
elements of the family of exact solutions.

One of the themes of this work is the web of 
interconnections between these three questions and their variants.
For instance, variants of \eqref{ineq/sublevel}, 
in which the coefficients $a_j$ are vector-valued,
arise naturally in our investigation of Questions~\ref{question1}
and \ref{question1b}.

\subsection{Content of paper}
We begin with remarks and examples placing
Question~\ref{question1} better in context. 
We then focus on the trilinear case, with $d=1$.
In all previous results for this case known to this author, 
the exponent $\gamma$ obtained for
\eqref{Linftyversion1} has been no greater than $\tfrac12$,
and we seek to surpass that threshold.
We introduce a condition for $\phi$,
whose negation we call rank one degeneracy on some hypersurface,
or simply rank one degeneracy.
We prove that if $\phi\in C^\omega$ is not rank one degenerate 
and satisfies an auxiliary hypothesis,
then the inequality \eqref{Linftyversion1} 
holds for some $\gamma$ strictly greater than $\tfrac12$.
Conversely, if $\phi\in C^\omega$ is rank one degenerate on some hypersurface,
then there exists no $\gamma>\tfrac12$ for which  \eqref{Linftyversion1} holds.
In a sense clarified in \S\ref{section:hypotheses},
the nondegeneracy hypothesis is satisfied by generic $C^\omega$ phase functions. 
We also explore the connection between multilinear oscillatory
forms \eqref{multioscform2} and the multilinear oscillatory forms studied in \cite{CLTT}.

Two applications to Question~\ref{question:implicitoscillation} are derived. 
The first is to products $\prod_{j=1}^3 (f_j\circ\varphi_j)$,
for functions $f_j:\reals^1\to\complex$ and for systems of mappings 
$\varphi_j:\reals^2\to\reals^1$
satisfying an appropriate curvature condition. 
We show that these are indeed well-defined as distributions
when $f_j$ lie in Sobolev spaces of slightly negative orders.
The second, a consequence of the first, is an alternative
proof of a theorem of Joly, M\'etivier, and Rauch \cite{JMR}
on weak convergence of products
$\prod_{j=1}^3 (f_j\circ\varphi_j)$
when the functions $f_j$ are weakly convergent
and the system of mappings $(\varphi_j: j\in\{1,2,3\})$
satisfies a suitable curvature hypothesis.
We exploit the improvement of the exponent $\gamma$ beyond the threshold $\tfrac12$
in \eqref{Linftyversion1} in these applications.

Our main results concerning \eqref{Linftyversion1} and 
\eqref{Linftyversion2}, respectively, are 
Theorem~\ref{mainthm} and Theorem~\ref{thm:nonosc}.  
Their proofs are based on decomposition 
in phase space, a dichotomy between structure and pseudo-randomness,
a two scale analysis, and a connection with sublevel sets.

In \S\ref{section:moreon}
we use the same method to give an alternative proof of a theorem
of Li, Tao, Thiele, and the author \cite{CLTT},
and to establish an extension. 

The machinery developed here establishes, and utilizes,
upper bounds of the form \eqref{ineq/sublevel} 
for Lebesgue measures of sublevel sets
associated to certain vector-valued functions,
in situations in which having each $f_j$ vector-valued is an advantage.
However, we also study the scalar-valued case.
In \S\ref{section:FEqns} 
we consider sublevel sets of the type \eqref{sublset},
with constant coefficients $a_j$,
and establish upper bounds for their Lebesgue measures
under natural hypotheses on $(\varphi_j: j\in\{1,2,3\})$
and appropriate nonconstancy hypotheses on $\bff$.
In \S\ref{section:variablecoeffs} we develop
a rather different method to study sublevel set inequalities
for nonconstant coefficients $a_j$,
in the special case in which the mappings $\varphi_j$ are all linear.
In \S\ref{section:sublevelexample} we construct an example
demonstrating optimality, 
for one of the simplest possible vector-valued instances of \eqref{ineq/sublevel}, 
of the apparently crude bound that our method yields. 
This example is based on multiprogressions of rank greater than $1$.
Finally, \S\ref{section:sublevelquestion} is devoted to further remarks
and questions concerning sublevel set inequalities.

The case $|J|=2$ of \eqref{multioscform1} is already well understood. 
We focus primarily on the next simplest case, 
in which $|J|=3$ and $d=1$, although these restrictions
are relaxed in some of our results. 
The techniques used here are developed further and applied in
two sequels, with Durcik and Roos \cite{CDR} in work on multilinear singular integral operators,
and with Durcik, Kova\v{c},
and Roos \cite{CDKR} in work on averages associated to $\reals$--actions.
The author plans to treat more singular cases,
such as $|J|\ge 4$ in Theorem~\ref{thm:nonosc}, in future work
via a further extension of the method.

In most of the paper, we assume phase functions and
mappings $\varphi_j$ to be real analytic, rather than merely 
infinitely differentiable. This is done primarily because
hypotheses can be formulated more simply in the 
$C^\omega$ case, with its natural dichotomy between 
functions that vanish identically, and those that vanish to finite order. 
Extensions of two of the main theorems to the $C^\infty$ case are formulated in \S\ref{section:variants}.
Another extension will be developed in the forthcoming dissertation of Zirui Zhou.

There are connections between 
the results and methods in this paper, 
a much earlier work of Bourgain \cite{bourgain},
and recent works of Peluse and Prendiville \cite{peluse},
\cite{peluse+prendiville2019},
\cite{peluse+prendiville2020}
involving quantitative nonlinear analogues
of Roth's theorem, cut norms, and degree reduction. 
See also \cite{taublogcut} for an exposition
of some of these ideas.

The author is indebted to Zirui Zhou for corrections
and useful comments on the exposition,
to Philip Gressman for a useful conversation,
and to Terence Tao for pointing out the connection
with the works of Bourgain, Peluse, and Prendiville. 
He thanks Craig Evans for serendipitously acquainting him with 
the work of Joly, M\'etivier, and Rauch, and for stimulating discussion. 

\section{Examples} \label{section:examples}

In all of the examples of this section, and most of the main results
of this paper, $d=1$. $B$ is often replaced by $[0,1]^J$, so
$T_\lambda^\phi(\bff) 
= \int_{[0,1]^J} e^{i\lambda\phi(\bx)} \prod_{j\in J} f_j(x_j)\,d\bx$.
In these examples, 
excepting Example~\ref{example:higherD}, $T_\lambda^\phi$ is trilinear.  

\begin{example} \label{example:first}
{\rm
For $|J|=3$ and 
$\phi(x_1,x_2,x_3) = x_2(x_1 + x_3)$,
the inequality \eqref{Linftyversion1} holds with $\gamma=1$.
To justify this, write
\[
T_\lambda^\phi(\bff)
= |\lambda|^{-1/2} \int_{[0,1]^2} f_1(x) F(y)e^{-i\lambda xy}\,dx\,dy
\]
where $F(y) = c f_2(y) |\lambda|^{1/2} \widehat{f_3}(\lambda y)$
for a certain harmless constant $c\ne 0$.
This function satisfies $\norm{F}_2 \lesssim \norm{f_2}_\infty \norm{f_3}_2
\le \norm{f_2}_\infty \norm{f_3}_\infty$.
One factor of $|\lambda|^{-1/2}$ has already been gained.
The remaining integral is $O(|\lambda|^{-1/2} \norm{f_1}_2\norm{F}_2)$ by
Plancherel's theorem and a change of variables.

This gives $|T_\lambda^\phi(\bff)| \le C|\lambda|^{-1} \norm{f_1}_2\norm{f_2}_\infty
\norm{f_3}_2$. For $p\in[2,\infty)$,
the optimal bound in terms of $\norm{f_1}_\infty \norm{f_2}_\infty \norm{f_3}_p$
is $O(|\lambda|^{-1}|\lambda|^{1/p})$. 
This can be seen by considering $f_j(x_j)=e^{-i\lambda x_j} \one_{[0,1]}(x_j)$
for $j=1,3$. Integrating with respect to $x_1,x_3$ then leaves a function
of $x_2$ whose real part is bounded below 
on $[0,\tfrac\pi4 |\lambda|^{-1}]$ by a positive constant independent of $\lambda$.
Choose $f_2$ to be the indicator function of $[0,\tfrac\pi4|\lambda|^{-1}]$.

However, for $\phi = x_1x_2+x_2x_3$, the situation changes if $\norm{f_2}_\infty$
is replaced by $\norm{f_2}_2$. The inequality
$|T_\lambda^\phi(\bff)| \le C|\lambda|^{-\gamma} \norm{f_2}_2\norm{f_1}_\infty\norm{f_3}_\infty$
holds for $\gamma = \tfrac12$, but not for any strictly larger exponent.
This is seen by choosing $f_2$ to be the indicator function of $[0,\tfrac\pi4 |\lambda|^{-1}]$
and $f_1,f_3$ each to be the indicator function of $[0,1]$.

This example will illustrate subtle points regarding the necessity of 
hypotheses in some of our main results, below.
} \end{example}

\begin{example}
{\rm
More generally, consider 
$\int_{[0,1]^n} 
e^{i\lambda \phi(\bx)}
\prod_{j=1}^n f_j(x_j)
\,d\bx$.
If \eqref{Linftyversion1} holds then
$\gamma \le \tfrac12 (n-1)$. 
Indeed, by a change of variables, one can replace $[0,1]$ by $[-1,1]$.
Define $a_j = \partial \phi/\partial x_j(0)$.
For each $j\le n-1$ define
\[ f_j(x_j) = e^{i\tau\lambda x_j^2}e^{-i\lambda a_j x_j}\eta(x_j)\]
where $\eta$ is a $C^\infty$ function supported in a small neighborhood
of $0$, satisfying $\eta(0)\ne 0$.
If $\tau$ is a sufficiently large constant, depending on $\phi$ but not on $\lambda$,
and if $\eta$ is supported in a sufficiently small neighborhood of $0$,
then by the method of stationary phase,
for a certain constant $c\ne 0$,
as $\lambda\to +\infty$,
\[\Big|
\int_{[0,1]^{n-1}} 
e^{i\lambda \phi(\bx)}
\prod_{j=1}^{n-1} f_j(x_j)
\,dx_1\,dx_2\cdots\,dx_{n-1}\Big|
= c\lambda^{-(n-1)/2} + O(\lambda^{-(n-3)/2})
\]
uniformly for all $x_n$ in some neighborhood $V$ of $0$ independent of $\lambda$.
Choose $f_n(x_n)$ to vanish outside of $V$
and to be equal to  $e^{ih(x_n)}$ with $h$ real-valued
so that \[ e^{ih(x_n)} 
\int_{[0,1]^{n-1}} 
e^{i\lambda \phi(\bx)}
\prod_{j=1}^{n-1} f_j(x_j)
\,dx_1\,dx_2\cdots\,dx_{n-1}
\ge 0\]
for each $x_n\in V$.
Thus $|T_\lambda^\phi(\bff)| = c'\lambda^{-(n-1)/2} + O(\lambda^{-(n-3)/2})$
with $c'\ne 0$.
} \end{example}

\begin{example} \label{example:higherD} {\rm 
For any $n\ge 3$,
the exponent $\gamma = (n-1)/2$ is realized for 
\[ \phi(x_1,\dots,x_n) = x_1x_2 + x_2x_3 + \cdots + x_{n-1}x_n.\]
This follows from the same reasoning as for $n=3$.
} \end{example}

\begin{example} {\rm 
For
$\phi(x_1,x_2,x_3) = x_1 x_2 + x_2x_3 + x_3x_1$, 
the optimal exponent is $\gamma=\tfrac12$.
Choosing $f_j(x) = e^{i\lambda x_j^2/2}$,
the integrand becomes $e^{i\lambda\psi}$
with $\psi(\bx) = (x_1+x_2+x_3)^2$.
This net phase function $\psi$ factors through a submersion from $\reals^3$ to $\reals^1$,
and has a critical point. 

This example, contrasted with $\phi = x_1x_2+x_2x_3$, for which the optimal
exponent is $1$, demonstrates that enlarging
the set of monomials that occur with nonzero coefficients can cause the
optimal exponent $\gamma$ to {\em decrease},
in contrast to the theory for a restricted range of parameters
developed in \cite{PSS} and \cite{gressmanetal}.
}\end{example}

\begin{example} \label{example:instability} {\rm 
More generally, for any parameter $r>0$, the optimal exponent is $\tfrac12$ for
\[\phi(x_1,x_2,x_3) = x_1 x_2 + x_2x_3 + r x_3x_1.\] 
Thus the optimal exponent is not lower semicontinuous with respect to $\phi$.
This also suggests that the exponent $\gamma = (n-1)/2$ is rarely attained.
}\end{example}

\begin{example}\label{example:mellin}{\rm
For $\phi(\bx) = x_1x_2x_3$, the exponent $\gamma=\tfrac12$ is again optimal.
Choosing $f_j(x) = e^{-i\lambda \ln(x)}\one_{[\tfrac12,1]}(x)$,
the net oscillatory factor becomes $e^{i\lambda\psi}$
with $\psi(\bx) = x_1x_2x_3 - \ln(x_1x_2x_3)$.
The gradient of $\psi$ vanishes identically on the hypersurface
$x_1x_2x_3=1$, and the integral is no better than $O(|\lambda|^{-1/2})$.
}\end{example}

\begin{example} \label{example:muddies} {\rm 
$\phi(\bx) = (x_1+x_2)x_3$.
This is merely Example~\ref{example:first}, with the indices $1,2,3$ permuted.
Thus we have already observed that
the inequality \eqref{Linftyversion1} holds with $\gamma=1$. 
Here we reexamine that example from an alternative perspective.
Integrating with respect to $x_3$ leads to
\[ \lambda^{-1/2} \int_{[0,1]^2} f_1(x)f_2(y) F_3(x+y) \,dx\,dy \]
where $F_3$ depends on $\lambda$
and satisfies $\norm{F_3}_{\lt} = O(\norm{f_3}_{\lt})$, but no stronger inequality for any $L^p$
norm of $F_3$ is available.
No oscillatory factor remains, yet we have already shown
in Example~\ref{example:first}
that an upper bound with another factor of $\lambda^{-1/2}$ does hold.
}\end{example}

\begin{example}\label{example:CLTTtype} {\rm
Consider $\phi(\bx) = x_3\varphi(x_1,x_2)+\psi(x_1,x_2)$
where $\psi$ is a polynomial and $(x,y)\mapsto \varphi(x,y)$ is a linear function
that is a scalar multiple neither of $x_1$ nor of $x_2$.
The same analysis as above leads to $\lambda^{-1/2}$ multiplied by
\begin{equation} \label{useCLTT} 
\int_{[0,1]^2} f_1(x)f_2(y) F(\varphi(x,y))e^{i\lambda \psi(x,y)}\,dx\,dy
\end{equation}
with $\norm{F}_2 = O(\norm{f_3}_2)$.
Suppose that $(\varphi,\psi)$ is nondegenerate in the sense that there do not
exist polynomials $h_j$ satisfying
\[ \psi(x,y) = h_1(x)+h_2(y)+h_3(\varphi(x,y)) \qquad \forall\,(x,y).\]
Then according to an inequality\footnote{An alternative proof of this inequality
is developed in \S\ref{section:moreon}.} of 
Li-Tao-Thiele and the present author \cite{CLTT}, the integral \eqref{useCLTT}
satisfies an upper bound of the form
$O(\lambda^{-\delta}\norm{f_1}_2\norm{f_2}_2\norm{F}_2)$
for some\footnote{The imprecise methods of \cite{CLTT} very rarely yield optimal exponents.}
$\delta(\varphi,\psi)>0$.
Thus $|T_\lambda^\phi(\bff)| \le C|\lambda|^{-\tfrac12-\delta} \prod_{j=1}^3 \norm{f_j}_\infty$.

Moreover, the analysis of \cite{CLTT}
implicitly proves that this bound holds uniformly for all sufficiently 
small perturbations of $\varphi,\psi$.
}\end{example}

\begin{example}\label{example:x3k} {\rm
Consider $\phi(\bx) = x_1x_2 + x_2 x_3^k$, with $\naturals\owns k\ge 2$.
For $k=2$,
\eqref{Linftyversion1} holds for every $\gamma$ strictly less than $1$. 
For $k\ge 3$,
it holds for every $\gamma \le \tfrac12 + \tfrac1k$,
and this exponent is optimal.
This can be shown by substituting
$x_3^k = \tilde x_3$
and using 
\[ \big| \int_{[0,1]^3} e^{i\lambda(x_1x_2+x_2x_3}\prod_{j=1}^3 g_j(x_j)\,d\bx\big|
\le C|\lambda|^{-1/2} \norm{g_1}_\infty \norm{g_2}_\infty \norm{g_3}_2\]
with $g_3(y) = f_3(y^{1/k}) y^{\tfrac1k-1}$.
That this exponent cannot be improved when $k\ge 3$
can be shown by considering $f_3$ equal to the indicator function of
$[0,\tfrac\pi4 \lambda^{-1/k}]$. 
} \end{example}

\begin{example} {\rm
Let $\phi(x,y) = x^2y-xy^2$,
or more generally, any homogeneous cubic polynomial that is
not a linear combination of $x^3,y^3,(x+y)^3$.
Then
\[ \big|\iint_{[0,1]^2} e^{i\lambda \phi(x,y)} f_1(x)f_2(y)f_3(x+y)\,dx\,dy\big|
\le C|\lambda|^{-\gamma} \prod_j \norm{f_j}_\infty\]
holds for $\gamma=\tfrac14$ \cite{GressmanXiao}.
Even for this simplest trilinear case of \eqref{multioscform2},
the optimal exponent remains unknown.
}\end{example}

\section{Nondegeneracy and curvature} 

For convenience we integrate over $[0,1]^3$, rather than over a ball,
though this makes no effective difference.
The two formulations are equivalent, by simple and well known arguments involving partitions
of unity and expansion of cutoff functions in Fourier series, resulting
in unimodular factors that can be absorbed into the functions $f_j$.

Thus we study functionals
\begin{equation} \label{3Dform}
T^\phi_\lambda(f_1,f_2,f_3) = \int_{[0,1]^3} e^{i\lambda \phi(\bx)} 
\prod_{j=1}^3 f_j(x_j)\,d\bx
\end{equation}
with $\bx=(x_1,x_2,x_3)\in\reals^3$ 
and $f_j:[0,1]\to\complex$,
and associated inequalities
\begin{equation} \label{Linftyversion}
|T_\lambda^\phi(\bff)| \le C|\lambda|^{-\gamma} \prod_{j=1}^3 \norm{f_j}_{L^{\infty}}.
\end{equation}
We assume throughout the discussion that $\lambda$ is positive
(as may be achieved by complex conjugation if $\lambda$ is initially negative) 
and that $\lambda\ge 1$.
For this situation, none of the results in \cite{PSS} and \cite{gressmanetal}
yield any exponent $\gamma$ strictly greater than $\tfrac12$,
and we focus on exceeding this benchmark exponent $\tfrac12$.

In results of this type, $\phi$ should be regarded as an equivalence class
of functions. 
If
$\tilde\phi$ takes the form
$\tilde\phi(\bx)  = \phi(\bx) - \sum_{j=1}^3 h_j(x_j)$
with all functions $h_j$ real-valued and Lebesgue measurable, then
\[ \sup_{\norm{f_j}_{\infty}\le 1} |T_\lambda^\phi(\bff)|
= \sup_{\norm{f_j}_{\infty}\le 1} |T_\lambda^{\tilde \phi}(\bff)|\]
since
each function $f_j$ can be replaced by $f_je^{-i\lambda h_j}$.
Thus $\tilde\phi$ is equivalent to $\phi$, so far as the 
inequality \eqref{Linftyversion} is concerned.
On the other hand, it is natural to require that the functions $h_j$
possess the same degree of regularity as is required of $\phi$.
The next definition is formulated in terms of maximally regular $h_j$,
but the minimally regular situation inevitably arises in the analysis.

The examples in \S\ref{section:examples} suggest a notion of degeneracy for phases $\phi$,
or equivalently, for such equivalence classes.
Write $\pi_j(x_1,x_2,x_3)=x_j$.

\begin{definition}
Let $U\subset\reals^3$ be open and nonempty.
Let $\phi:U\to\reals$ be $C^\omega$.
Let $H\subset U$ be a $C^\omega$ hypersurface.
$\phi$ is rank one degenerate on $H$ if
there exist $C^\omega$ functions
$h_j$ defined in $\pi_j(U)$
such that the associated net phase function  
$\tilde \phi = \phi -\sum_{j=1}^3 (h_j\circ \pi_j)$
satisfies 
\begin{equation} \label{rank1degen}
(\nabla\tilde\phi)\big|_H\equiv 0.
\end{equation}
\end{definition}

In this definition, $H$ may be defined merely in some small subset of $U$.

$\phi:U\to\reals$ is said to be rank one degenerate on some hypersurface,
or simply rank one degenerate,
if there exist $H\subset U$ and functions $h_j$ such that \eqref{rank1degen} holds.
$\phi:[0,1]^3\to\reals$ is said to be rank one degenerate on some hypersurface
if this holds for the restriction of $\phi$ to $(0,1)^3$.

If \eqref{rank1degen} holds, then the Hessian matrix of $\tilde\phi$
has rank less than or equal to $1$ at each point of $H$, whence the term ``rank one''. 
It is the restriction to $H$ of the full gradient that is assumed
to vanish in \eqref{rank1degen}, rather than the gradient of the restriction.

\begin{example}{\rm
$\phi(\bx) = x_1x_2+x_2x_3+x_3x_1$ is rank one degenerate
on the hypersurface $H$ defined by $x_1+x_2+x_3=0$.
Choosing $h_j(x_j) = -x_j^2/2$
gives $\tilde\phi(\bx) = (x_1+x_2+x_3)^2/2$,
whose gradient vanishes on $H$.

More generally, for $r\ne 0$, the rank one degenerate phases
$\phi_r(\bx) = x_1x_2+x_2x_3+ rx_3x_1$ are
equivalent to phases $\tilde\phi$ whose gradients vanish
along hyperplanes $H_r$ defined by $x_2 = -r(x_1+x_3)$.
} \end{example}

\begin{example}{\rm
Let $r\in\reals$.
$\phi(\bx) = x_3(x_1+x_2) + rx_1x_2x_3$
is not rank one degenerate.
$\phi(\bx) = (x_1+x_2+x_3)^2 + rx_1x_2x_3$
is rank one degenerate if and only if $r=0$.
} \end{example}

\begin{proposition}
If $\phi\in C^\omega$ is rank one degenerate on a hypersurface $H$, then
the inequality \eqref{Linftyversion1} cannot hold for any $\gamma$
strictly greater than $\tfrac12$ on any ball $B$ containing $H$.
\end{proposition}

\begin{proof}
If the Hessian of $\phi$ does not vanish identically on $H$
then there exists a relatively open subset $\tilde H$ of $H$
on which this Hessian has rank $1$.
Choose $f_j(x_j) = e^{i\lambda h_j(x_j)}$,
multiplied by cutoff functions
that localize $\prod_j f_j(x_j)$ to a neighborhood of $\tilde H$,
and invoke asymptotics provided by the method of stationary phase.
The same reasoning applies so long as $\phi$ is not an affine function
on $[0,1]^3$, by fibering a neighborhood of a point of $H$
by line segments transverse to $H$, evaluating the asymptotic
contribution of each line segment as $|\lambda|\to\infty$,
and integrating with respect to a transverse parameter.
\end{proof}

\begin{example}{\rm
For any multi-index $\alpha\in \naturals^3$,
the phase function $\phi(\bx) = \bx^\alpha = \prod_ {j=1}^3 x_j^{\alpha_j}$
is rank one degenerate on every open subset of $(\reals\setminus\{0\})^3$.
Therefore $\phi$ does not satisfy \eqref{Linftyversion1}
with $\gamma>\tfrac12$ on any domain $B$.
} \end{example}

We will also study integrals of the form
\begin{equation}
\int_{\reals^2} e^{i\lambda\psi(\bx)} \prod_{j=1}^3 (f_j\circ\varphi_j)(\bx)\,\eta(\bx)\,d\bx
\end{equation}
with $\psi,\varphi_j:\reals^2\to\reals^1$ real analytic, $\lambda\in\reals$,
and $f_j$ in Lebesgue spaces or Sobolev spaces of negative order.
$\eta\in C^\infty(\reals^2)$ will be a compactly supported smooth cutoff function.
Both the situations in which $\lambda$ 
is a large parameter, and that in which $\lambda=0$, are of interest.

The concepts of a $3$-web, and its curvature, are relevant here.
A {\it $3$-web} 
in $\reals^2$ is by definition a $3$--tuple of pairwise transverse
smooth foliations of a connected open
subset of $\reals^2$ 
\cite{BB},\cite{JMR}.
The leaves of each foliation are one-dimensional. 
If $\varphi_i:\reals^2\to\reals^1$ are smooth functions,
and if $\nabla\varphi_j(\bx)$ and $\nabla\varphi_k(\bx)$ are linearly independent
at every point $\bx$ for each pair of distinct indices $j,k$, then
the datum $(\varphi_j: j\in\{1,2,3\})$ defines a  $3$-web, 
whose leaves are level sets of these functions.
Conversely, any $3$-web is locally defined by such a tuple of functions.
If $\varphi$ and $\tilde\varphi$ define the same foliation,
then any $f\circ \varphi$ can be written as $\tilde f\circ\tilde\varphi$,
where $\tilde f$ has $L^p$ and $W^{s,p}$ Sobolev norms comparable to those of $f$.
Thus the inequalities that we will study will depend on the underlying
web, rather than on the tuple $(\varphi_j)$ used to describe it.

Associated to a $3$-web on an open set $U$ is its curvature, 
a real-valued function with domain $U$ defined by Blaschke, and
discussed in \cite{JMR} and references cited there.
This curvature vanishes at
a point $\bx_0$ if and only if there exist
smooth functions $f_j:\reals\to\reals$ satisfying $f'_j(\varphi_j(\bx_0))\ne 0$
for at least one index $j$,
such that the associated function $F = \sum_{j=1}^3 f_j\circ\varphi_j$ 
satisfies\footnote{Assuming pairwise transversality of the foliations, there 
always exist $f_j$ such that $F(\bx)-F)(\bx_0) = O(|\bx-\bx_0|^3)$.}
$F(\bx)-F(\bx_0) = O(|\bx-\bx_0|^4)$ as $\bx\to\bx_0$.
This condition depends only on the underlying $3$-web,
not otherwise on associated functions $\varphi_j$.
It is invariant under local diffeomorphism of the ambient space $\reals^2$.
The equivalence of this condition with vanishing curvature can be shown via 
a short calculation in local coordinates chosen so that $\varphi_j(x_1,x_2)\equiv x_j$
for $j=1,2$.

If $\varphi_j(\bx)=x_j$ for $j=1,2$,
then a web defined by $(\varphi_j: j\in\{1,2,3\})$
has curvature identically zero in an open set
if and only if the ratio $\frac{\partial\varphi_3/\partial x_1}{\partial\varphi_3/\partial x_2}$
factors locally as the product of a function of $x_1$ alone
with a function of $x_2$ alone \cite{JMR}.

If there exist $f_j$ such that $F$ vanishes identically in a neighborhood
of $\bx_0$ then necessarily $f'_j(\varphi_j(\bx_0))\ne 0$
and the change of variables $\bx\mapsto (f_1\circ\varphi_1(\bx),f_2\circ\varphi_2(\bx))$
and the substitution $\varphi_3\mapsto f_3\circ\varphi_3$
transform all three functions $\varphi_i$ into affine functions.
If $\varphi_j(x_j)\equiv x_j$ for $j=1,2$,
and if 
$\frac{\partial^2 \varphi_3}{\partial x_1\partial x_2}$
vanishes identically in a neighborhood of $\bx_0$,
then $\varphi_3$ is a sum of functions of the individual
coordinates. Therefore the 
curvature vanishes identically in a neighborhood of $\bx_0$.

We say that
$(\varphi_j: j\in\{1,2,3\})$ is equivalent to a linear system
if there exist 
$C^\omega$ real-valued functions $H_j$,
each with derivatives that do not vanish identically
in any neighborhood of $\varphi_j([0,1]^2)$, satisfying
$\sum_{j=1}^3 H_j\circ\varphi_j\equiv 0$.
In this situation, $(\tilde\varphi_j=H_j\circ\varphi_j: j\in\{1,2,3\})$ defines
the same $3$-web as does $(\varphi_j)$. Taking $\tilde\varphi_j$
as coordinates for $j=1,2$,
all three functions $\tilde\varphi_j$ become linear.

If $\nabla\varphi_j,\nabla\varphi_k$ are linearly
independent at $\bx_0$ for each pair of distinct indices $j\ne k$,
then the curvature of the $3$-web defined by $(\varphi_j)$
vanishes identically in a neighborhood of $\bx_0$
if and only if
$(\varphi_j: j\in\{1,2,3\})$ is equivalent to a linear system
in a neighborhood of $\bx_0$.

\medskip
The following lemma connects two notions of curvature/nondegeneracy,
and will be used in the proof of Theorem~\ref{thm:nonosc}.

\begin{lemma} \label{lemma:curv-nondegen}
Suppose that in some nonempty open subset $U\subset \reals^2$,
$\varphi\in C^\infty$, $\partial\varphi/\partial x_i$ vanishes
nowhere for $i=1,2$, and the $3$-web associated to
$(x_1,x_2,\varphi(x_1,x_2))$ has nowhere vanishing curvature.
Then the phase function $\phi(x_1,x_2,x_3)=x_3\varphi(x_1,x_2)$
is not rank one degenerate in any open subset of $U\times (\reals\setminus\{0\})$.
\end{lemma}

\begin{proof}
Write $\varphi_i = \partial\varphi/\partial x_i$ for $i=1,2$.
Suppose that $\tilde\phi=x_3\varphi(x_1,x_2) - \sum_{j=1}^3 h_j(x_j)$
has gradient identically vanishing on a smooth hypersurface $H$.
If $H$ can be expressed in some nonempty open set in the form $x_3=F(x_1,x_2)$,
then 
$x_3 \varphi_i(x_1,x_2)\equiv h'_i(x_i)$ for $i=1,2$.
It is given that $\varphi_i$ does not vanish. Therefore we may
form the ratio of partial derivatives $\varphi_1/\varphi_2$
and conclude that it can be expressed, in some nonempty open
subset of $U$, as a product of a function of $x_1$ with a function of $x_2$.
This contradicts the hypothesis of nonvanishing curvature,
as shown in \cite{JMR}.

If $\tilde\phi$ has gradient identically vanishing on some smooth
hypersurface $H$ that cannot be expressed in the above form
on any nonempty open set, then $H$ must take the form
$\Gamma\times I$ for some nonconstant curve $\Gamma\subset\reals$
and some interval $I\subset\reals$ of positive length. 
The equations $x_3\varphi_i(x_1,x_2) \equiv h'_i(x_i)$
force $\varphi_i(x_1,x_2)\equiv 0$ on $\Gamma$, 
contradicting the assumption that $\varphi_i = \partial\varphi/\partial x_i$
vanishes nowhere.
\end{proof} 

\section{Formulations of some results} \label{section:mainresults}

The first main result of this paper is concerned with
multilinear expressions
\begin{equation} \label{multioscform3}
T_\lambda^\phi(\bff) = \int_{[0,1]^3} e^{i\lambda \phi(\bx)} 
\prod_{j=1}^3 f_j(x_j)\,d\bx,
\end{equation}
restricting attention to three functions $f_j:[0,1]\to\complex$,
and integrating over $[0,1]^3$ rather than over a ball.

\begin{theorem} \label{mainthm}
Let  $\phi$ be a real analytic, real-valued function in a neighborhood $U$ of $[0,1]^3$.
Suppose that $\phi$ is not rank one degenerate on any hypersurface in $U$.
Suppose that for each pair of distinct indices $j\ne k\in\{1,2,3\}$,
$\frac{\partial^2 \phi} {\partial x_j\,\partial x_k}$
vanishes nowhere on $[0,1]^3$.
Then there exist $\gamma>\tfrac12$ and $C<\infty$ 
such that the operators defined in \eqref{multioscform3} satisfy
\begin{equation} \label{ineq:main}
\big| T_\lambda^\phi(\bff) \big| \le C|\lambda|^{-\gamma} \prod_j \norm{f_j}_2
\end{equation}
uniformly for all functions $f_j\in L^2(\reals^1)$ and all $\lambda\in\reals$.
\end{theorem}

The condition that a single partial derivative
$\frac{\partial^2 \phi} {\partial x_1\,\partial x_2}$
vanishes nowhere suffices, for $C^\infty$ phases $\phi$ without other hypotheses, to ensure that 
\[ \int_{[0,1]^2} e^{i\lambda\phi(x_1,x_2,x_3)} \prod_{j=1}^2 f_j(x_j)\,dx_1\,dx_2
= O\big(|\lambda|^{-1/2}\norm{f_1}_2\norm{f_2}_2\big)\]
uniformly in $x_3$ \cite{hormander}. 
Consequently
\[  T_\lambda^\phi(\bff)   = O\big(|\lambda|^{-1/2}  
\norm{f_1}_2 \norm{f_2}_2 \norm{f_3}_1\big).\]
The content of Theorem~\ref{mainthm} is the improvement, with appropriate
norms on the right-hand side, of the exponent beyond $\tfrac12$.

The set of all $\phi$ that satisfy the hypotheses of  Theorem~\ref{mainthm}
is nonempty, and is open with respect to the $C^3$ topology. 
The set of all $3$--jets for $\phi$ at $\bx_0$
that guarantee validity of the hypotheses in some small neighborhood
of $\bx_0$ is open and dense. Moreover, its complement is contained in a $C^\omega$ variety
of positive codimension in the space of jets.
This is shown in \S\ref{section:hypotheses}.

The theorem is not valid for $C^\infty$ phases $\phi$ as stated.
If $\phi$ were merely $C^\infty$, 
then $\phi$ could vanish to infinite order at a single point,
without any equivalent phase $\tilde\phi$ satisfying $\nabla\tilde\phi|_H\equiv 0$
for any hypersurface $H$.
Infinite order degeneracy at a point
implies that \eqref{ineq:main} does not hold for any $\gamma>0$,
even with $L^\infty$ norms on the right-hand side of the inequality.
Corresponding remarks apply to other results formulated in this paper.

The norms appearing on the right-hand side
of \eqref{ineq:main} are $L^2$ norms, rather than $L^\infty$.
Thus phases that satisfy the hypotheses of the theorem
enjoy stronger bounds on $L^2\times L^2\times L^2$
than does the example $\phi(\bx) = x_3(x_1+x_2)$, which attains
the largest possible exponent, $\gamma = 1$, on $L^\infty\times L^\infty\times L^\infty$,
but only $\gamma=\tfrac12$ on $L^2\times L^2\times L^2$.
This phase satisfies the main hypothesis of rank one nondegeneracy,
but fails to satisfy the auxiliary hypothesis of three nonvanishing
mixed second partial derivatives.

We believe that under the rank one nondegeneracy hypothesis,
the conclusion holds if one of the three mixed second partial derivatives
vanishes nowhere, but the other two are merely assumed
not to vanish identically.
Theorem~\ref{thm2}, below, supports this belief.

Functions associated to $\phi$ by solutions of certain implicit equations
arise naturally in our analysis, so it is not natural to restrict
attention to polynomial phases in the formulation of the theorems,
as is sometimes done in works on this topic.
Example~\ref{example:mellin} also demonstrates that for polynomial phases,
it is not always natural to restrict to polynomial functions $h_j$
in formulating the equivalence relation between phases
or the notion of rank one degeneracy.

\smallskip
Oscillatory factors do not appear explicitly in the formulation of
our second main result, Theorem~\ref{thm:nonosc},
which is concerned with conditions under which the integral of
$\prod_{j\in J} (f_j\circ\varphi_j)$ is well-defined.
If $\eta\in C^0$ has compact support in $\reals^2$, and if
$\nabla\varphi_j$ and $\nabla\varphi_k$ are linearly independent
at each point in the support of $\eta$ for every pair
of distinct indices $j,k\in\{1,2,3\}$,
and if each $f_j\in L^{3/2}(\reals^1)$,
then the product $\eta(\bx)\prod_{j=1}^3 f_j\circ\varphi_j$
belongs to $L^1(\reals^2)$.
This is a simple consequence of complex interpolation, since the
product belongs to $L^1$ whenever two of the three functions
belong to $L^1(\reals^1)$ and the third belongs to $L^\infty$.
The exponent $\tfrac32$ is optimal in this respect.
This leaves open the possibility that the
integral might be well-defined when the $f_j$
belong to certain Sobolev spaces of negative orders.

For $p\in(1,\infty)$ and $s\in\reals$,
denote by $W^{s,p}$ the Sobolev space of all distributions
having $s$ derivatives in $L^p$.

\begin{question}
Let $J$ be a finite index set. Let $U\subset\reals^2$ nonempty and open.
For $j\in J$, let $\varphi_j:U\to\reals$
be $C^\omega$ with nowhere vanishing gradient. Suppose that for any $j\ne k\in J$,
$\nabla\varphi_j$ and $\nabla\varphi_k$ are linearly independent at almost
every point in $U$. Let $\eta\in C^\infty_0(U)$. Do there exist $s<0$,
$p<\infty$,  and $C<\infty$ such that
\begin{equation} \label{multidesire}
\big| \int_{\reals^2} \eta\cdot \prod_{j\in J} (f_j\circ\varphi_j) \big|
\le C\prod_{j\in J} \norm{f_j}_{W^{s,p}}\end{equation}
for all functions $f_j\in C^1(\varphi_j(U))$?
\end{question}

The answer is negative without further hypotheses. In particular, it is negative
whenever all $\varphi_j$ are linear. But inequalities \eqref{multidesire} do hold
under suitable conditions.

\begin{theorem} \label{thm:nonosc}
Let $\varphi_j\in C^\omega$ for each $j\in\{1,2,3\}$.
Suppose that for every pair of distinct indices $j\ne k\in\{1,2,3\}$,
$\nabla\varphi_j$ and $\nabla\varphi_k$ are linearly independent at $\bx_0$.
Suppose that the curvature of the web defined by $(\varphi_1,\varphi_2,\varphi_3)$
does not vanish at $\bx_0$.
Then there exist $\eta\in C^\infty_0$ satisfying $\eta(\bx_0)\ne 0$
such that for any exponent $p>\tfrac32$,
there exist $C<\infty$ and $s<0$ such that
	\begin{equation} \label{nonosc:conclusion}
\big| \int_{\reals^2} \prod_{j=1}^3 (f_j\circ\varphi_j)\,\eta\big|
\le C\prod_j \norm{f_j}_{W^{s,p}}
	\ \text{ for all $\bff\in (L^{3/2}(\reals^1))^3$.}
\end{equation}
\end{theorem}

The assumption that $\bff\in L^{3/2}$ guarantees absolute convergence of the integral.
The particular instance of Theorem~\ref{thm:nonosc}
with the ordered triple $(x_1,x_2)\mapsto (x_1,x_1+x_2,x_1+x_2^2)$ of mappings was treated
by Bourgain \cite{bourgain} in 1988.

The proof will implicitly establish a formally stronger inequality. 
Let $\gamma\in(0,1)$. 
Let $\lambda\in(0,\infty)$ be large.
Suppose that the Fourier transform of at least one of the functions $f_j$
is supported in the region in which the Fourier variable satisfies $|\xi|\ge\lambda$.
Partition a sufficiently small neighborhood of the support of $\eta$
into cubes $Q_n$, each of sidelength $\lambda^{-\gamma}$.
Then
\begin{equation*} 
\sum_n \big| \int_{Q_n } \prod_{j=1}^3 (f_j\circ\varphi_j)\big|
\le C\lambda^{s} \prod_j \norm{f_j}_{L^p}.
\end{equation*}

Theorem~\ref{thm:nonosc} is a simple consequence of Theorem~\ref{mainthm},
with the validity of the inequality \eqref{ineq:main} for some exponent strictly greater than $\tfrac12$
being crucial in the analysis.
It is worth noting that the deduction relies on the appearance of
$\lt$ norms, rather than merely $L^\infty$ norms, on the right-hand side of \eqref{ineq:main}.
The tuple $(\varphi_1,\varphi_2,\varphi_3) = (x_1,x_2,x_1+x_2)$
illustrates this relatively delicate distinction. This example
does not satisfy the inequality \eqref{nonosc:conclusion}.
When the analysis used below to reduce Theorem~\ref{thm:nonosc}
to \eqref{ineq:main} is applied to it,
the phase that arises is $\phi(x_1,x_2,x_3)=x_3(x_1+x_2)$.
This is Example~\ref{example:first},
for which the $L^\infty$ inequality holds with $\gamma=1$,
but the $\lt$ inequality \eqref{ineq:main} holds only
for $\gamma=\tfrac12$, not for any larger exponent.

Theorem~\ref{thm:nonosc} has the following immediate consequence
for the weak convergence of products of weakly convergent factors.

\begin{corollary} \label{cor:weakconverge}
Let $\eta,\varphi_j$ satisfy
the hypotheses of Theorem~\ref{thm:nonosc}.
Let $p>\tfrac32$.
For $\nu\in\naturals$ let
$f_j^{\nu}\in L^p$ have uniformly bounded $L^p$ norms.
If $f_j^{\nu}$ converges weakly to $f_j$ 
as $\nu\to\infty$ for $j=1,2,3$ then
\begin{equation}
\prod_{j=1}^3 (f_j^{\nu}\circ\varphi) 
\text{ converges weakly to } 
\prod_{j=1}^3 (f_j\circ\varphi)
\ \text{ as $\nu\to\infty$}
\end{equation}
in a neighborhood of $\bx_0$.
\end{corollary}

That is,
\begin{equation}
\int \eta\, \prod_{j=1}^3 (f_j^{\nu}\circ\varphi)
\to \int \eta\, \prod_{j=1}^3 (f_j\circ\varphi)
\ \text{ as $\nu\to\infty$}
\end{equation}
for every function $\eta\in C^\infty$ supported in a sufficiently
small neighborhood of $\bx_0$.

Corollary~\ref{cor:weakconverge}
is a slight variant
of a result established by Joly-M\'etivier-Rauch \cite{JMR} 
using semiclassical defect measures.\footnote{In \cite{JMR}
the functions $\varphi_j$ are $C^\infty$ rather than $C^\omega$.
In Theorem~2.2.1 of \cite{JMR}
it is assumed that 
the curvature is nonzero at $\bx_0$, while in Theorem~2.2.3
the curvature is allowed to vanish on any set of Lebesgue measure zero,
but a stronger hypothesis is imposed on $f_j$.}
In \cite{JMR}, each $f_j\circ\varphi_j$ is replaced by a
function that possesses some quantitative smoothness along the 
level curves of $\varphi_j$ but need not be constant.
Such an extension is a simple consequence of Theorem~\ref{thm:nonosc},
and is formulated and proved below as Theorem~\ref{thm:nonosc2}
and Corollary~\ref{cor:weakconverge2}.

\smallskip
Consider functionals of the form
\begin{equation}
S_\lambda(\bff) =
\int_{[0,1]^2} 
e^{i\lambda \psi(\bx)} 
\prod_{j=1}^3 f_j(\varphi_j(\bx))
\,d\bx.
\end{equation}

\begin{theorem} \label{thm:3}
Let $\varphi_j:[0,1]^2\to\reals$ and $\psi:[0,1]^2 \to\reals$ be real analytic.
Suppose that for any two indices $j,k\in\{1,2,3\}$,
the Jacobian determinant of
the mapping $[0,1]^2\owns \bx \mapsto (\varphi_j(\bx),\varphi_k(\bx))\in\reals^2$
does not vanish identically.
Suppose that there exist no nonempty open subset
$U\subset(0,1)^2$ and $C^\omega$ functions $h_j:\varphi_j(U)\to\reals$ satisfying
\begin{equation} \label{psiexpression}
\psi(\bx) = \sum_{j=1}^3 h_j(\varphi_j(\bx)) \ \text{ for all $\bx\in U$.}
\end{equation}
Then there exist $\delta>0$ and $C<\infty$ satisfying
\begin{equation} \label{thm3conclusion}
|S_\lambda(\bff)| \le C|\lambda|^{-\delta}\prod_{j=1}^3 \norm{f_j}_{L^2}
\ \text{ for all $\bff$ and all $\lambda\in\reals$.}
\end{equation}
\end{theorem}

For linear mappings $\varphi_j$, 
two different generalizations of this inequality were proved in \cite{CLTT}. 
For this linear case, and for any tuple $(\varphi_j: j\in\{1,2,3\})$
reducible to a linear tuple by a change of variables,
Theorem~\ref{thm:3} is a special case of results obtained in that work.
While the method of analysis in \cite{CLTT} exploited linearity
of $\varphi_j$, in \S\ref{section:moreon} we sketch an alternative proof 
of one of the two main results of \cite{CLTT} by the method developed here
that allows an extension to the nonlinear case.

If the Jacobian determinant of $\bx\mapsto (\varphi_j(\bx),\varphi_k(\bx))$ 
vanishes nowhere for each pair of distinct indices $j,k$, then
$|S_\lambda(\bff)| \le C\norm{f_i}_1\norm{f_j}_1\norm{f_k}_\infty$
for any permutation $(i,j,k)$ of $(1,2,3)$.
Thus by interpolation, it suffices to prove \eqref{thm3conclusion}
with the $L^2$ norms replaced by $L^\infty$ norms on the right-hand side.

\begin{example}
For $(\varphi_1,\varphi_2,\varphi_3) = (x_1,x_2,x_1+x_2)$
and $\psi(\bx) = x_1^2 x_2$, and with the $\lt$ norms on the
right-hand side replaced by $L^\infty$ norms,
the inequality for $S_\lambda(\bff)$ holds with $\delta = \tfrac14$,
and fails for $\delta > \tfrac13$ \cite{GressmanXiao}. The optimal exponent,
for $L^\infty$ norms, is unknown for even this (simplest) example.
\end{example}

\begin{conjecture} 
Let $J$ be a finite set of indices.
Let $D\ge 2$, and let $d_j\ge 1$ for $j\in J$.
Let $B\subset\reals^D$ be a ball of finite radius.
For each $j\in J$, let $\varphi_j\in C^\omega(B,\reals^{d_j})$
be nonconstant.
Likewise, Let $\psi\in C^\omega(B,\reals)$.
Suppose that $\psi$ cannot be expressed as $\psi = \sum_{j\in J} h_j\circ\varphi_j$
in any open subset of $B$, with $h_j\in C^\omega$.
Then there exists 
$\gamma>0$ such that
for all $\lambda\in\reals$ and all continuous functions $f_j$,
\[ \big| \int_{B} e^{i\lambda\psi} \prod_{j\in J} (f_j\circ\varphi_j) \big|
\le 
C|\lambda|^{-\delta} 
\, \prod_{j\in J} \norm{f_j}_{L^\infty}.\]
\end{conjecture}

For the case in which $D=2$, $d_j=1$, all $\varphi_j$ are linear, and $\psi$ is a polynomial,
this is proved in \cite{CLTT}. 

\medskip
This paper is organized so that Theorem~\ref{thm:nonosc} is proved along with related results,
including Theorems~\ref{mainthm} and \ref{thm2}.
A more direct and somewhat simpler
roof of Theorem~\ref{thm:nonosc} can be extracted from the discussion. 

\section{Variants and extensions} \label{section:variants}

We next formulate a result for the special case 
in which $\phi$ is an affine function of $x_3$; thus
$\phi(\bx) = x_3\varphi(x_1,x_2)+\psi(x_1,x_2)$.
The proof  developed below for this special case
is a simplification of the proof of Theorem~\ref{mainthm},
and relies on Theorem~\ref{thm:3}, 
thus bringing to light connections between these results.

\begin{theorem} \label{thm2}
Let $J=\{1,2,3\}$ and $d=1$.
Let 
\begin{equation}
\phi(x_1,x_2,x_3) = x_3\varphi(x_1,x_2) + \psi(x_1,x_2)
\end{equation}
where $\varphi,\psi$ are real-valued 
real analytic functions defined in a neighborhood of $[0,1]^2$.
Suppose that 
$\partial\varphi/\partial x_1$
and $\partial\varphi/\partial x_2$
vanish nowhere on $[0,1]^2$.
Suppose that there exists no open subset of $[0,1]^2$
in which $\psi$ can be expressed in the form
\begin{equation} \label{pairdegeneracy} 
\psi(x_1,x_2) = Q_1(x_1)+Q_2(x_2)+(Q_3\circ\varphi)(x_1,x_2)\end{equation}
for $C^\omega$ functions $Q_1,Q_2,Q_3$.
Then there exist $\gamma>\tfrac12$ and $C<\infty$ satisfying
\begin{equation} \label{ineq:main2}
\big| T_\lambda^\phi(\bff) \big| \le C|\lambda|^{-\gamma} \prod_{j=1}^3 \norm{f_j}_2
\end{equation}
uniformly for all functions $f_j\in L^2(\reals^1)$ and all $\lambda\in\reals$.
\end{theorem}

Theorem~\ref{thm2} is not quite a special case of 
Theorem~\ref{mainthm}, 
because it is not assumed here that $\frac{\partial^2\psi}{\partial x_1\partial x_2}$
is nonzero, and therefore 
$\frac{\partial^2\phi}{\partial x_1\partial x_2}(0)$
could vanish.

The hypothesis 
that $\psi$ cannot be expressed in the form
\eqref{pairdegeneracy}
is not necessary for the conclusion to hold, as shown by the example
$\phi(\bx) = x_3(x_1+x_2)$, for which $\psi\equiv 0$. 
In this respect, Theorem~\ref{mainthm} is more satisfactory.
A more typical example excluded by this hypothesis is $(\varphi,\psi) = (x_1+x_2,x_1x_2)$,
for which the conclusion \eqref{ineq:main2} does indeed fail.

Theorem~\ref{thm:3} directly implies Theorem~\ref{thm2}. 
Indeed, set $\varphi_j(x_1,x_2)=x_j$ for $j=1,2$, and $\varphi_3=\varphi$.  Then
\begin{equation}
T_\lambda^\phi(\bg) = |\lambda|^{-1/2} S_\lambda(\bff)
\end{equation}
where $f_j=g_j$ for $j=1,2$, and $f_3(t) = |\lambda|^{1/2}\widehat{g_3}(\lambda t)$.
Then $f_3$ satisfies $\norm{f_3}_2 = O(\norm{g_3}_2)$.

The next result combines oscillation with negative order Sobolev norms
in the context of Theorem~\ref{thm:nonosc}.

\begin{theorem} \label{thm:4}
Consider $S_\lambda$ with $J=\{1,2,3\}$, $d=1$, and $D=2$.
Let $\varphi_j\in C^\omega$ for each $j\in\{1,2,3\}$.
Suppose that for any two indices $j\ne k$,
$\nabla\varphi_j$ and $\nabla\varphi_k$ are linearly independent
at every point.
Suppose that there exist no nonempty open subset
$U\subset(0,1)^2$ and $C^\omega$ functions $h_j:\varphi_j(U)\to\reals$ satisfying
\begin{equation} 
\psi(\bx) = \sum_{j=1}^3 h_j(\varphi_j(\bx)) \ \text{ for $\bx\in U$.}
\end{equation}
Suppose also that $(\varphi_1,\varphi_2,\varphi_3)$
is not equivalent to a linear system in any nonempty open set.
Then for each $p>\tfrac32$
there exist $C<\infty$, $\delta>0$, and $s<0$ such that
\begin{equation} 
|S_\lambda(\bff)| \le C(1+|\lambda|)^{-\delta}\prod_{j=1}^3 \norm{f_j}_{W^{s,p}}
\ \text{ for all $\bff\in (L^p)^3$ and all $\lambda\in\reals$.}
\end{equation}
\end{theorem}

Theorem~\ref{thm:nonosc},
Corollary~\ref{cor:weakconverge},
and
Theorem~\ref{thm:3}
imply straightforward generalizations
to integrals over $[0,1]^d$ with products of $d+1$ functions,\footnote{Theorem~\ref{mainthm} 
generalizes in the same way, but the threshold exponent $\gamma = \tfrac12$
is less natural for $d>3$.}
for arbitrary $d\ge 2$. 
Such generalizations are obtained
by changing variables and regarding the domain
of integration as the union of a $d-2$--dimensional
family of two-dimensional slices, in such a way
that $d-2$ of the factors are constant along slices.
Under appropriate hypotheses, the results of this paper
can be applied to each slice.

Here is such an analogue of Theorem~\ref{thm:nonosc}.
Let $d\ge 3$, and let $J$ be an index set of cardinality $|J|=d+1$.
Let $\varphi_j\in C^\omega$ for each $j\in J$.
Suppose that $(\nabla\varphi_j: j\in J)$ are transverse,
in the sense that for any subset $\tilde J\subset J$
of cardinality $d$, $\{\nabla\varphi_j: j\in\tilde J\}$
are linearly independent at $\bx_0$.
For each subset $J'\subset J$ of cardinality $d-2= |J|-3$
consider the foliation of a neighborhood of $\bx_0$ in $\reals^d$
with $2$--dimensional leaves $L_t=\{x: \varphi_j(x)=t_j \ \forall\,j\in J'\}$,
where $t\in\reals^{J'}$.
For each such $t$, restriction of
the family of three functions $\{\varphi_i: i\in J\setminus J'\}$ to $L_t$
defines a web on $L_t$ for each $t$.

\begin{theorem} \label{thm:nonosc_d>2}
Let $d\ge 3$ and $|J|=d+1$.
Let $\bx_0\in\reals^d$, and let $V$ be a neighborhood of $\bx_0$. 
Suppose that at $\bx_0$, $(\varphi_j: j\in J)$
satisfies the transversality hypothesis introduced above.
Suppose that for each $i\in J$
there exists
a subset $J'\subset J$ satisfying $|J'|=d-2$,
with $i\notin J'$,
such that for $t\in \reals^{J'}$ defined by $\varphi_j(\bx_0)=t_j$
for each $j\in J'$, 
the curvature of the web defined above on $L_t$ does not vanish $\bx_0$.
Then there exist $\eta\in C^\infty_0$ satisfying $\eta(\bx_0)\ne 0$,
$C<\infty$, and $s<0$ satisfying
\begin{equation*} 
\big| \int_{\reals^d} \prod_{j=1}^3 (f_j\circ\varphi_j)\,\eta\big|
\le C\prod_{j\in J} \norm{f_j}_{W^{s,2}}
\ \text{ for all $\bff\in (L^2(\reals^d))^J$.}
\end{equation*}
\end{theorem}

This hypothesis, application of Theorem~\ref{thm:nonosc}
to integrals over two-dimensional slices defined by
$L_j(\bx)=t_j$ for $j\in J'$, and integration with respect to $t$
over a bounded subset of $\reals^{J'}$
yield an upper bound of the form
$C\norm{f_i}_{W^{s,2}} \prod_{j\ne i} \norm{f_j}_{L^2}$
for some $s<0$, for each $i\in J$. Interpolation of these bounds
then produces an upper bound of the desired form
$C\prod_{j\in J} \norm{f_j}_{W^{s,2}}$,
with $s$ replaced by $s/|J|$.

Theorem~\ref{thm:nonosc_d>2} in turn implies a corresponding 
extension of Corollary~\ref{cor:weakconverge}.

All of our results have extensions to the case of $C^\infty$ phase functions,
but the hypothesis of rank one nondegeneracy must be reformulated.
For Theorem~\ref{mainthm}, such an extension can be phrased as follows.

\begin{theorem} \label{Cinftythm}
Let $J=\{1,2,3\}$, and $d=1$. 
Let  $\phi\in C^\infty$ be real-valued and defined in a neighborhood $U$ of $[0,1]^3$.
Suppose that there do not exist a point $\bz\in U$,
a germ $\scriptm$ of $C^\infty$ manifold $\scriptm$ of dimension $2$ at $\bz$,
and $C^\infty$ functions $h_j$
such that the restriction to $\scriptm$
of the gradient of $\tilde\phi(\bx) = \phi(\bx)-\sum_{j=1}^3 h_j(x_j)$
vanishes to infinite order at $\bz$.

Suppose that for each pair of distinct indices $j\ne k\in\{1,2,3\}$,
$\frac{\partial^2 \phi} {\partial x_j\,\partial x_k}$
vanishes nowhere on $[0,1]^3$.
Then there exist $\gamma>\tfrac12$ and $C<\infty$ satisfying
\begin{equation*} 
\big| T_\lambda^\phi(\bff) \big| \le C|\lambda|^{-\gamma} \prod_j \norm{f_j}_2
\end{equation*}
uniformly for all functions $f_j\in L^2(\reals^1)$ and all $\lambda\in\reals$.
\end{theorem}

A corresponding modification of the hypotheses of Theorem~\ref{thm:3}
is needed for a $C^\infty$ analogue.
Consider functionals of the form
$ S_\lambda(\bff) = \int_{[0,1]^2} e^{i\lambda \psi(\bx)} 
\prod_{j=1}^3 f_j(\varphi_j(\bx)) \,d\bx $
as in Theorem~\ref{thm:3}, where $\varphi_j:U\to\reals$ and $\psi:U\to\reals$ 
are $C^\infty$ functions defined in some neighborhood $U$ of $[0,1]^2$.

\begin{theorem} \label{thm:3Cinfty}
Suppose that for any two indices $j,k\in\{1,2,3\}$,
the Jacobian determinant of
the mapping $[0,1]^2\owns \bx \mapsto (\varphi_j(\bx),\varphi_k(\bx))\in\reals^2$
does not vanish identically.
Suppose that there do not exist 
$C^\infty$ functions $h_j$ defined in neighborhoods
of the closure of $\varphi_j(U)$
and a point $\bz\in U$ such that
$\psi(\bx) - \sum_{j=1}^3 h_j(\varphi_j(\bx))$
vanishes to infinite order at $\bz$.
Then there exist $\delta>0$ and $C<\infty$ satisfying
\begin{equation*} 
|S_\lambda(\bff)| \le C|\lambda|^{-\delta}\prod_{j=1}^3 \norm{f_j}_{L^2}
\ \text{ for all $\bff$ and all $\lambda\in\reals$.}
\end{equation*}
\end{theorem}

The proofs of 
Theorems~\ref{Cinftythm} and \ref{thm:3Cinfty} are the same as
those of the corresponding results for the $C^\omega$ case,
with small modifications in the concluding sublevel set analysis.
The details of these modifications are omitted.

\medskip
The next part of the paper is organized as follows.
We begin the proofs with
Theorem~\ref{thm2}, reducing it in \S\ref{section:reductions} to
a special case of Theorem~\ref{thm:3}, which we then prove
in \S\S\ref{section:decompose}, \ref{section:localbound},
\ref{section:reducetosublevel}, and \ref{section:sublevel1}.

Theorem~\ref{mainthm} is proved in \S\ref{section:proofofmainthm}
and \S\ref{section:sublevel2} by elaborating on that analysis.
We establish Theorem~\ref{thm:3} in its full generality,
and derive Theorems~\ref{thm:nonosc} and \ref{thm:4} from
these methods and results, in \S\ref{section:completeproofs}.
In \S\ref{section:yetanother} we enunciate and prove
extensions to the Joly-M\'etivier-Rauch
framework, in which the condition that the factors $f_j$
be constant along leaves of foliations is replaced by smoothness
along those leaves.
\S\ref{section:hypotheses}
contains remarks concerning the hypotheses, demonstrating that
these are satisfied generically, in an appropriate sense.

\section{Reductions} \label{section:reductions}

We begin by showing how 
Theorem~\ref{thm2} follows from a bandlimited case of Theorem~\ref{thm:3}.
Let $(\varphi,\psi)$ satisfy its hypotheses. 
There are two cases, depending on whether or not $\varphi$ can be expressed 
in the form 
\begin{equation} \label{varphiform} 
\varphi(x,y) \equiv H(h_1(x) + h_2(y)) \text{ on $[0,1]^2$}
\end{equation}
with $H,h_1,h_2\in C^\omega$.
If $\varphi$ does take the form \eqref{varphiform} then 
a $C^\omega$ change of variables with respect to $x$ and to $y$,
together with replacement of $\varphi$ by $\tilde H\circ\varphi$ for
appropriate $\tilde H$,
reduces matters to the case in which $h_1,h_2$ are linear. In these new coordinates,
$\psi$ remains $C^\omega$, and \eqref{pairdegeneracy} continues to hold.
This places us in the setting of Example~\ref{example:CLTTtype}, 
which was treated above as a consequence of the results of \cite{CLTT}.
We restrict attention henceforth to the second case, in which $\varphi$ cannot be expressed 
in the form \eqref{varphiform}.

Integrate with respect to $x_3$ to reexpress
\begin{multline*} \int_{[0,1]^3} e^{i\lambda x_3\varphi(x_1,x_2) }
e^{i\lambda\psi(x_1,x_2)}
\prod_{j=1}^3 f_j(x_j)\,dx_1\,dx_2\,dx_3
\\
=  |\lambda|^{-1/2} \int_{[0,1]^2} e^{i\lambda\psi(x_1,x_2)} 
f_1(x_1)f_2(x_2)F_3(\varphi(x_1,x_2))\,dx_1\,dx_2\end{multline*}
with $F_3(t) = |\lambda|^{1/2}\widehat{f_3}(\lambda t)$
satisfying $\norm{F_3}_2 = c\norm{f_3}_2$.
Thus
\[ T_\lambda^\phi(\bff) = c|\lambda|^{-1/2}S_\lambda(f_1,f_2,F_3)\]
with $S_\lambda$ defined in terms of the phase function $\psi$,
and with the ordered triple of mappings
\[(\varphi_1,\varphi_2,\varphi_3)(x_1,x_2) = (x_1,x_2,\varphi(x_1,x_2)).\]
The hypotheses of Theorem~\ref{thm:3} are satisfied
by $(\psi,\varphi_1,\varphi_2,\varphi_3)$.
Therefore the conclusion of Theorem~\ref{thm2} for $T_\lambda^\phi(\bff)$
is a consequence of the conclusion of Theorem~\ref{thm:3},
which yields a factor of $|\lambda|^{-\delta}$ with $\delta>0$,
supplementing the factor $|\lambda|^{-1/2}$ that is already present.
\qed

The function $F_3$ is $|\lambda|$--bandlimited, that is, its Fourier
transform was supported in $[-|\lambda|,|\lambda|]$.
Thus this proof relies  only on this bandlimited case of Theorem~\ref{thm:3}.

In the following sections we will establish the conclusion
of Theorem~\ref{thm:3} in the $O(|\lambda|)$--bandlimited case,
thus completing the proof of Theorem~\ref{thm2}.
The general case of Theorem~\ref{thm:3} will be treated later, in \S\ref{section:completeproofs}. 
Theorem~\ref{thm2} will be used in the proof for the general case.
Our treatment of the bandlimited case of Theorem~\ref{thm:3}
will not rely on Theorem~\ref{thm2}, so the reasoning is not circular.

We begin the proof of Theorem~\ref{thm:3},
for general $(\psi,\varphi_1,\varphi_2,\varphi_3)$
satisfying its hypotheses,
without any bandlimitedness hypothesis for the present.
Thus it is given that for each $j\ne k\in\{1,2,3\}$,
$\nabla\varphi_j,\nabla\varphi_k$ are linearly independent 
on the complement of a analytic variety of positive codimension.
If $(\varphi_j: j\in\{1,2,3\})$ is equivalent to a linear system,
then the conclusion \eqref{thm3conclusion} holds.
Indeed, suppose that $\sum_j H_j\circ\varphi_j\equiv 0$.
Supposing initially that the derivatives of $H_j$ vanish nowhere,
the change of variables $\bx\mapsto(H_1\circ\varphi_1(\bx),H_2\circ\varphi_2(\bx))$
reduces matters to the case in which $\varphi_j(\bx)\equiv x_j$ for $j=1,2$.
Replace $\varphi_3$ by $\tilde\varphi_3 = -H_3\circ \varphi_3$.
In these new coordinates,
$\tilde\varphi_3(\bx) = x_1+x_2$, 
and the nondegeneracy hypothesis for $\psi$ continues to hold.
For this situation,
the conclusion \eqref{thm3conclusion} was established in \cite{CLTT}.

In the more general case in which derivatives  $H'_j$
are permitted to vanish at isolated points,
and gradients $\nabla\varphi_j$ are permitted to be pairwise linearly
dependent on analytic varieties of positive codimensions,
the same conclusion is reached by partitioning $[0,1]^2$
into finitely many good rectangles, on each of which each
derivative has absolute value bounded below by $|\lambda|^{-\delta}$,
together with a bad set of Lebesgue measure $O(|\lambda|^{-\delta'})$
for small exponents $\delta,\delta'>0$.
The reasoning of the preceding paragraph 
gives the desired bound for the contribution of each good rectangle,
while the contribution of the remaining bad set is majorized
by a constant multiple of its Lebesgue measure.

We claim further that in order to prove
Theorem~\ref{thm:3}, it suffices to treat 
the special case in which $\varphi_j(x_1,x_2)\equiv x_j$ for $j=1,2$, 
neither partial derivative $\frac{\partial \varphi_3}{\partial x_j}$
with $j=1,2$ vanishes at any point of $[0,1]^2$,
and $(\varphi_j: j\in\{1,2,3\})$ is not equivalent to a linear system.
To justify this claim,
let $\eps>0$ be a small auxiliary parameter, and
partition $[0,1]^2$ into subcubes 
of sidelengths comparable to $\lambda^{-\eps}$.
Discard every subcube on which any one of the three Jacobian determinants
fails to have magnitude greater than $\lambda^{-\eps}$.
The sum of the measures of these discarded subcubes is $O(\lambda^{-\delta})$
for some $\delta=\delta(\eps)>0$.
Treat each of the remaining subcubes by reducing it to $[0,1]^2$
via an affine change of variables. This replaces $\lambda$ by a positive
power of $\lambda$, and likewise modifies $\varphi_j,\psi$.

Next, make the change of variables
$\bx = (x_1,x_2)\mapsto \phi(\bx) = (\varphi_1(\bx),\varphi_2(\bx))$,
which is a local diffeomorphism because of the nonvanishing Jacobian condition.
Replace $\varphi_3$ by $\varphi_3\circ \phi^{-1}$, replace $\varphi_j(\bx)$ by $x_j$
for $j=1,2$, and replace $\psi$ by $\psi\circ \phi^{-1}$.
The hypotheses of Theorem~\ref{thm:3} continue to hold for this new system of data.
$\phi([0,1]^2)$ is no longer equal to $[0,1]^2$,
but is contained in a finite union of rectangles, in each of which the
hypotheses of the theorem hold after affine changes of variables.

This change of variables introduces a Jacobian factor, which is a function
of $\bx$ rather than of individual coordinates. This Jacobian can
be expanded into a Fourier series, expressing it as an absolutely convergent 
linear combination of products
of unimodular functions of the individual coordinates. These factors can be absorbed
into the functions $f_j$.
The case in which $(\varphi_j: j\in\{1,2,3\})$ is equivalent to a linear system
has already been treated.

Write $D  = \frac{d}{dx}$.
\begin{definition}
Let $\lambda\in(0,\infty)$ and $N\in\naturals$.
$\norm{\cdot}_{N,\lambda}$ is the norm on the Banach
space of $N$ times continuously differentiable functions on $[0,1]$ given by
\begin{equation}
\norm{f}_{N,\lambda} = \sum_{k=0}^N \lambda^{-k}\norm{D^k f}_{L^\infty([0,1])}.
\end{equation}
\end{definition}

\S\S\ref{section:decompose}, \ref{section:localbound},
\ref{section:reducetosublevel}, and \ref{section:sublevel1}
are devoted to the proof of the following lemma.

\begin{lemma} \label{lemma:CLTTbandlimited}
Suppose that $\varphi_j(x_j)\equiv x_j$ for $j=1,2$,
that $\varphi_3$ is not expressible in the form $h_1(x_1)+h_2(x_2)$,
and that $\psi$ is not expressible in the form \eqref{psiexpression}.
Then there exist $N,C,\delta$ such that for all $\bff$ and every $\lambda\ge 1$,
\begin{equation} \label{desired}
\big|S_\lambda(\bff)\big| \le C \lambda^{-\delta} \norm{f_1}_\infty \norm{f_2}_\infty
\norm{f_3}_{N,\lambda}.
\end{equation}
\end{lemma}

We have observed that
\begin{equation} \label{TlambdaSlambda}
 T_\lambda^\phi(\bg) = 
 \lambda^{-1/2} 
S_\lambda(\bff) 
\end{equation}
with $f_j=g_j$ for $j=1,2$, $\norm{f_3}_2 \le C \norm{g_3}_2$,
and $f_3$ is $|\lambda|$--bandlimited.
Therefore 
in order to complete the proof of Theorem~\ref{thm2},
it suffices to prove that $|S_\lambda(\bff)|\le C|\lambda|^{-1/2} \norm{f_1}_\infty \norm{f_2}_\infty
\norm{f_3}_2$ under this bandlimitedness assumption on $f_3$.

Theorem~\ref{thm2} follows from this lemma. Indeed,
we may assume without loss of generality that $\lambda>0$,
by replacing $\psi$ by $-\psi$ if $\lambda$ is initially negative.
Since $f_3$ is $\lambda$--bandlimited, we may express $f_3=P_\lambda(f_3)$,
where $P_\lambda$ are linear smoothing operators that satisfy
\begin{equation} \norm{\nabla^k P_\lambda f}_q\le C_{q,k} \lambda^k \norm{f}_q
\ \text{ for all $f\in L^q$} \end{equation}
uniformly for all $q\in[1,\infty]$ 
and $\lambda>0$, for each $k\in\{0,1,2,\dots\}$. 
Thus
\begin{equation}
\norm{P_\lambda f}_{N,\lambda} \le C_N \norm{f}_\infty
\end{equation}
uniformly for all $\lambda>0$ and $f\in L^\infty$.

The hypothesis that $\partial\varphi/\partial x_1$ does not vanish leads immediately 
to an upper bound
\[ |S_\lambda(\bff)| \le C\norm{f_1}_\infty \norm{f_2}_1\norm{f_3}_1,\]
and interchanging the roles of the coordinates
gives a bound $C\norm{f_i}_\infty \norm{f_j}_1\norm{f_k}_1$ for
any permutation $(i,j,k)$ of $(1,2,3)$.
Therefore by interpolation, since $P_\lambda$ is bounded on $L^q$ for all $q$
uniformly in $\lambda$,
\eqref{desired} implies that
\begin{equation} 
\big|S_\lambda(f_1,f_2,P_\lambda(f_3))\big| \le C \lambda^{-\delta} \norm{f_1}_\infty \norm{f_2}_\infty
\norm{f_3}_2.
\end{equation}
By \eqref{TlambdaSlambda},
this completes the proof of Theorem~\ref{thm2}.
\qed

\section{Microlocal decomposition} \label{section:decompose}

We decompose each $f_j$ in phase space into
summands that are essentially supported in rectangles
of dimensions $(\lambda^{-1/2},\lambda^{1/2})$ in $[0,1]_x\times \reals_\xi$.
To do this,
partition $[0,1]$ 
into $\asymp \lambda^{1/2}$
intervals $I_m$ of lengths $|I_m|=\lambda^{-1/2}$.
Let $\eta_m$ be $C^\infty$ functions with each $\eta_m$
supported on the interval $I_m^*$ of length $2\lambda^{-1/2}$ 
concentric with $I_m$, with $\sum_m \eta_m^2\equiv 1$ on $[0,1]$,
and with $d^k\eta_m/dx^k = O(\lambda^{k/2})$ for each $k\ge 0$.

For $\nu = (m_1,m_2)$ let $Q_\nu = I_{m_1}\times I_{m_2}\subset[0,1]\times[0,1]$.
Let $z_\nu$ be the center of $Q_\nu$.
To each $\nu$ are associated those intervals $I_{m_3}$
for which there exists at least one point $\bx=(x_1,x_2)\in Q_\nu$
such that $\varphi(\bx)\in I_{m_3}^*$. 
Because each partial derivative $\partial\varphi/\partial x_j$ vanishes nowhere,
the number of such indices $m_3$ is majorized by a constant independent of $\lambda,\nu$.

Let $\sigma\in(0,1]$ be a small quantity to be chosen at the very end of the analysis.
For each interval $I_m$, decompose
$f_j \eta_m^2$ as
\begin{equation} \label{fjdecomp1}
f_j \eta_m^2 = g_{j,m}+h_{j,m}
\end{equation}
with $g_{j,m},h_{j,m}$ identically zero outside of $I_m^*$, 
\begin{equation} \label{fjdecomp2}
\left\{ \begin{gathered}
g_{j,m}(x) = \eta_m(x) \sum_{k=1}^N a_{j,m,k}e^{i\xi_{j,m,k} x} 
\\
|a_{j,m,k}|=O(\norm{f_j}_\infty), 
\\
\xi_{j,m,k}\in \pi \lambda^{1/2}\integers,
\\
N  = \lceil \lambda^{2\sigma} \rceil.
\end{gathered} \right. \end{equation}
while 
\begin{equation} \label{fjdecomp3}
\left\{ 
\begin{gathered} 
h_{j,m}(x) = \eta_m(x) \sum_{n\in \integers} b_{j,m,n}e^{i\pi \lambda^{1/2} n x} 
\\
(\sum_n |b_{j,m,n}|^2)^{1/2} = O(\norm{f_j}_\infty)
\\
|b_{j,m,n}| = O(\lambda^{-\sigma} \norm{f_j}_\infty).
\end{gathered}
\right. \end{equation}
Decompositions of this type were used by the author and J.~Holmer,
in unpublished work circa 2009, 
to prove upper bounds for certain generalizations of twisted convolution
inequalities.

This is achieved by expanding $f_j\eta_m$ into Fourier series 
\[f_j(x)\eta_m(x) =  \one_{I_m^*}(x)
\sum_{n\in \integers} c_n  e^{i\pi\lambda^{1/2} n x}, \]
with coefficients $c_n$ that depend also on the indices $j,m$.
Define $g_{j,m}$ to be 
the sum of all terms with $|c_n| > \lambda^{-\sigma}\norm{f_j}_\infty$,
multiplied by $\eta_m$. 
Define $h_{j,m} = f_j\eta_m-g_{j,m}$.
By Parseval's identity, there are at most $ \lceil \lambda^{2\sigma} \rceil$
values of $n$ for which 
$|c_n| > \lambda^{-\sigma}$.
Define the frequencies $\xi_{j,m,k}$
and associated coefficients $a_{j,m,k}$ to be those frequencies $\pi\lambda^{1/2}n$
and associated coefficients $c_n$ 
that satisfy $|c_n|>\lambda^{-\sigma} \norm{f_1}_\infty$,
with some arbitrary ordering.
If there are fewer than $N$ indices $n$ for which $|c_n|>\lambda^{-\sigma}\norm{f_1}_\infty$,
then augment this list by introducing extra
indices $k$ so that there are exactly $N$ terms, and
set some $a_{j,m,k}=0$ for each of these extra indices. 
This is done purely for convenience of notation.

Define
\[ g_j = \sum_m g_{j,m} \ \text{ and } \  h_j = \sum_m h_{j,m}
\ \text{ for $j\in\{1,2\}$.} \] 

For $j=3$, this construction is modified in order to exploit
the bandlimited character of $f_3$.
Let $\rho>0$ be another small parameter.\footnote{One may
think of $\rho$ as being arbitrarily small, but of $\sigma$ as moderate in size.
Thus factors such as $\lambda^{-\sigma+C\rho}$ will be small for large $\lambda$,
so long as $C$ remains constant.}
It follows from $N$-fold integration by parts that
\[ |\widehat{f_3\eta_m}(\xi)| \le C_N 
\lambda^{N} |\xi|^{-N} \norm{f_3}_{N,\lambda}
\ \forall\,\xi.  \]
If $N$ is chosen to satisfy $N \ge \rho^{-1}$, it follows that
\[ |\widehat{f_3\eta_m}(\xi)| \le C_N 
\lambda^{-1} \ \text{ whenever $|\xi| \ge \lambda^{1+\rho}$.}\]
Therefore the frequencies $\xi_{3,m,k}$ defined above satisfy 
\begin{equation}\label{xi3band}
|\xi_{3,m,k}|\le  \lambda^{1+\rho}.
\end{equation}

Moreover, if $N$ is chosen sufficiently large as a function of $\rho$,
then the contribution made to $h_3$ by all terms 
$b_{3,m,k} e^{ikx}$ with $|k| \ge \lambda^{1+\rho}$
has $\lt$ norm $O(\lambda^{-1})$. 
Define $F_3$ to be the sum of all of these terms. Then $f_3$ is decomposed as
\begin{equation} \label{F3decomp} f_3 = g_3+h_3+F_3,\end{equation}
with 
\begin{equation} \norm{F_3}_\infty = O(\lambda^{-1}),\end{equation} 
with $g_3,h_3$ enjoying all of the properties indicated above
for $j=1,2$,
and with the supplementary bandlimitedness property 
\[ |n|\le\lambda^{1+\rho}\] for all
frequencies $n$ appearing in terms $\eta_m(x)b_{3,m,n}e^{inx}$,
as well as for all frequencies $\xi_{3,m,k}$.

\section{Local bound} \label{section:localbound}

Recall that
\[ S_\lambda(F_1,F_2,F_3)
= \int_{\reals^2} 
F_1(x_1)F_2(x_2)F_3(\varphi(x_1,x_2))\,e^{i\lambda\psi(x_1,x_2)}
\,dx_1\,dx_2.  \]

Let $\bm = (m_1,m_2,m_3)\in\integers^3$.
Write $\norm{a_{r,\cdot}}_{\ell^p} = (\sum_{n} |a_{r,n}|^p)^{1/p}$,
with the usual limiting interpretation for $p=\infty$.

\begin{lemma} \label{lemma:localbound}
Let $\rho>0$ be a small auxiliary parameter.
Let $f_j$ be functions of the form
\[ f_j(x) = 
\sum_{n\in\integers} a_{j,n}e^{i\pi \lambda^{1/2} nx}\]
with $a_{j,n}\in\complex$ and $|a_{3,n}|=0$ for all $|n|>\lambda^{1+\rho}$. 
Then for each $\bm$
and any permutation $(j,k,l)$ of $(1,2,3)$,
\begin{equation}
\big|S_\lambda(f_1\eta_{m_1},f_2\eta_{m_2},f_3\eta_{m_3}) \big|
\le C
\lambda^{-1+2\rho}
\norm{a_{j,\cdot}}_{\ell^2}
\norm{a_{k,\cdot}}_{\ell^2}
\norm{a_{l,\cdot}}_{\ell^\infty}
\end{equation} 
\end{lemma}

\begin{proof}
For $i=1,2$ we write
$\varphi_i,\psi_i$ as shorthand for $\partial\varphi/\partial x_i$,
$\partial\psi/\partial x_i$, respectively.
Write $\nu = (m_1,m_2)$, and recall that $z_\nu$ denotes
the center of $Q_\nu = I_{m_1}\times I_{m_2}$.
Write $\bx = (x_1,x_2)$.

Let $\bxi = (\xi_j: j\in\{1,2,3\})\in\reals^3$, and suppose that 
\begin{equation} \label{xioneplusbound} 
\max_j|\xi_j|\le  \lambda^{1+\rho}.\end{equation}
Consider
\begin{equation} I(\xi) =  \int_{\reals^2} e^{i\xi_{1}x_1}e^{i\xi_2 x_2} e^{i\xi_3\varphi(\bx)}
e^{i\lambda\psi(\bx)}
\eta_{m_1}(x_1)\eta_{m_2}(x_2)\eta_{m_3}(\varphi(\bx))\,d\bx.\end{equation}
The net phase function in this integral is 
\begin{equation} \label{netPhi}
\Phi(\bx) = \xi_1x_1 + \xi_2 x_2 + \xi_3\varphi(\bx)+\lambda\psi(\bx),\end{equation}
whose gradient is
\begin{equation*} 
\nabla\Phi(\bx) = 
\begin{pmatrix}
\xi_1 + \xi_3\varphi_1(\bx) + \lambda\psi_1(\bx)
\\
\xi_2 + \xi_3\varphi_{2}(\bx) + \lambda\psi_{2}(\bx)
\end{pmatrix}\,.  \end{equation*}

If
\begin{equation} \label{bigatcenter}
\big|\nabla\Phi(z_\nu)|
\ge \lambda^{2\rho} \lambda^{1/2}
\end{equation}
then
\begin{equation} \label{Ibound}
|I| \le C_{\rho,K} \lambda^{-K}
\ \text{for every $K<\infty$.}
\end{equation}
Indeed,
if \eqref{bigatcenter} holds then 
$\big|\frac{\partial\Phi}{\partial x_i}(z_\nu)\big| \ge \tfrac12 \lambda^{2\rho}\lambda^{1/2}$ 
for at least one index $i\in\{1,2\}$. Suppose without loss of generality that this holds for $i=1$.
Then
\begin{equation} \label{bigatcenter_1}
\Big|\frac{\partial \Phi}{\partial x_1}(u_1,u_2) \Big| \ge \lambda^{\rho} \lambda^{1/2}
\ \text{ for every point $(u_1,u_2)\in Q_\nu^* = I_{m_1}^*\times I_{m_2}^*$.}
\end{equation}
This holds because
the function $\lambda\psi_{1}$ varies by at most $O(\lambda\cdot\lambda^{-1/2})$ over $Q_{\nu}^*$,
while the assumption \eqref{xioneplusbound} guarantees that $\xi_3 \varphi_{1}$ varies by at most 
$O(\lambda^{1+\rho}\lambda^{-1/2}) = O(\lambda^{\tfrac12+\rho})$.
Integrating by parts $CK \rho^{-1}$ times 
with respect to the $x_1$ coordinate 
and invoking \eqref{bigatcenter_1}
then yields \eqref{Ibound}.

Now writing $\bn = (n_1,n_2,n_3)\in\integers^3$, 
\[ S_\lambda(f_1\eta_{m_1},f_2\eta_{m_2},f_3\eta_{m_3}) 
= \sum_{n_1,n_2,n_3} 
a_{1,n_1}a_{2,n_2}a_{3,n_3}
I(\pi \lambda^{1/2} n_1,\pi \lambda^{1/2} n_2,\pi \lambda^{1/2} n_3).\]
For any $\bn = (n_1,n_2,n_3)$ there is the trivial bound
$|I(\bn)| = O(|Q_\nu^*|) = O(\lambda^{-1})$.
On the other hand,
by \eqref{Ibound}, 
the $\bn$-th term is $O(\lambda^{-K})$
if the associated phase function $\Phi$ defined by \eqref{netPhi}
with $\bxi =  \pi\lambda^{1/2}\bn$ satisfies
$|\nabla\Phi(z_\nu)| \ge \lambda^{\tfrac12+2\rho}$. 

For each $n_1$, there are $O(\lambda^{2\rho})$ pairs $(n_2,n_3)$
for which $\bxi = \pi\lambda^{1/2}\bn$
fails to satisfy \eqref{bigatcenter}.
This follows from the form of $\nabla\Phi$
and the assumption that both partial derivatives $\varphi_1,\varphi_2$
are nowhere vanishing.
The same holds with the roles of $n_1,n_2,n_3$ permuted in an arbitrary way.
Since the total number of all tuples $(\bm,\bn)$
is $O(\lambda^{3+2\rho})$,
the conclusion of the lemma follows directly from these facts 
by invoking \eqref{Ibound} with $K$ sufficiently large.
\end{proof}

\section{Reduction to sublevel set bound} \label{section:reducetosublevel} 
Let $f_1,f_2$ be decomposed as $f_j = g_j+h_j$ as 
in \eqref{fjdecomp1}, \eqref{fjdecomp2}, 
and let $f_3$ have the modified form $f_3 = g_3+h_3+F_3$ of \eqref{F3decomp}, 
with the restriction \eqref{xi3band}.
Then
$S_\lambda(f_1,f_2,F_3) = O(\lambda^{-1}\prod_{j=1}^3 \norm{f_j}_\infty)$,
so the contribution of $F_3$ can be disregarded
and $f_3$ may be replaced by $\tilde f_3 = g_3+h_3$. 
By summing over all cubes $Q_\nu$
we conclude from Lemma~\ref{lemma:localbound} that
\begin{equation}
|S_\lambda(h_1,f_2,\tilde f_3)| \le C\lambda^{-\sigma} \lambda^{2\rho} \prod_{j=1}^3\norm{f_j}_\infty.
\end{equation}
In the same way,
\begin{equation}
|S_\lambda(g_1,h_2,\tilde f_3)| 
\ + \  |S_\lambda(g_1,g_2,h_3)|  \le C\lambda^{-\sigma} \lambda^{2\rho} \prod_{j=1}^3\norm{f_j}_\infty
\end{equation}
so that 
\begin{equation}
|S_\lambda(f_1,f_2,f_3)| \le 
|S_\lambda(g_1,g_2,g_3)| 
+ C\lambda^{-\sigma} \lambda^{2\rho} \prod_{j=1}^3\norm{f_j}_\infty.
\end{equation}
Thus matters are reduced to the analysis of $S_\lambda(g_1,g_2,g_3)$.

To complete the proof, we analyze functions of the special form
\begin{equation} \label{fjform} 
G_j(x) = \sum_{m} \eta_m(x)  a_{j,m} e^{ix\cdot \xi_{j,m}} 
\end{equation}
with each $a_{j,m}\in\complex$ satisfying $|a_{j,m}|\le 1$, and each $\xi_{j,m}\in\reals$.
For $j=3$, we also assume
\begin{equation} \label{fj2} 
|\xi_{3,m}|\le \lambda^{1+\rho}. 
\end{equation}

For each index $j$, $g_{j}$
is expressed as a sum over $k_j\in\{1,2,\dots,N\}$
of functions $G_j$ of the form \eqref{fjform}, multiplied by $O(\norm{f_j}_\infty)$.
Moreover, each summand $G_3$ is bandlimited in the sense \eqref{fj2}. 
\begin{equation} \label{gtoGsum}
S_\lambda(g_1,g_2,g_3) = 
O\big(\prod_{j=1}^3 \norm{f_j}_\infty) \cdot
\sum_{(k_1,k_2,k_3)\in\{1,2,\dots,N\}^3} |S_\lambda(G_{1,k_1},G_{2,k_2},G_{3,k_3}|\big)
\end{equation}
with $N^3$ terms in the sum.

We will prove:
\begin{lemma} \label{lemma:finalspecialform}
There exist  $\tau_0>0$ and $C<\infty$ such that
for all functions of the form \eqref{fjform}
satisfying also \eqref{fj2},
\begin{equation} |S_\lambda(G_1,G_2,G_3)| \le C\lambda^{-\tau_0} \end{equation}
uniformly for all real $\lambda\ge 1$. 
\end{lemma}

Taking Lemma~\ref{lemma:finalspecialform} for granted for the present, 
we can now complete the proof of Theorem~\ref{thm:3} in the $O(|\lambda|)$--bandlimited
case, and hence the proof of Theorem~\ref{thm2}.
Applying Lemma~\ref{lemma:finalspecialform} to each of the $N^3$ summands in \eqref{gtoGsum} yields
\begin{equation}
|S_\lambda(g_1,g_2,g_3)| \le CN^3 \lambda^{-\tau_0} \prod_{j=1}^3 \norm{f_j}_\infty.
\end{equation}
In all,
\begin{align}
|S_\lambda(f_1,f_2,f_3)| 
&\le \big(CN^3 \lambda^{-\tau_0} + C\lambda^{-\sigma}\lambda^{2\rho} \big)
\prod_{j=1}^3 \norm{f_j}_\infty
\notag
\\
&\le \big(C\lambda^{6\sigma} \lambda^{-\tau_0} + C\lambda^{-\sigma}\lambda^{2\rho} \big)
\prod_{j=1}^3 \norm{f_j}_\infty,
\label{in_all}
\end{align}
where $C<\infty$ depends only on $\varphi,\psi$ and the auxiliary parameters $\sigma,\rho>0$.
The exponent $\sigma$ remains at our disposal, while $\rho$ may be taken to be arbitrarily small.
Choosing $\sigma = \tau_0/7$ gives
\begin{equation} \label{gives}
|S_\lambda(f_1,f_2,f_3)| 
\le 
C\lambda^{-\tau}
\prod_{j=1}^3 \norm{f_j}_\infty
\end{equation}
for every $\tau< \tau_0/7$.
\qed

\medskip
We next reduce 
Lemma~\ref{lemma:finalspecialform} to a sublevel set bound.
Let $G_j$ have the above form for $j\in\{1,2,3\}$.
By decomposing $G_3$ as a sum of $O(1)$ subsums,
we may assume that for each $\nu = (m_1,m_2)$
there exists at most one index $m_3=m_3(\nu)$ 
for which the product
$\eta_{m_1}(x_1) \eta_{m_2}(x_2) \eta_{m_3}(\varphi(x_1,x_2))$
does not vanish identically.

For $\nu = (m_1,m_2)$ and for $\bm = (m_1,m_2,m_3(\nu))$,
for each $(x_1,x_2)\in Q_\nu$ define 
\[\Phi_\nu (x_1,x_2) = \xi_{1,m_1}x_1 + \xi_{2,m_2}x_2 + \xi_{3,m_3}\varphi(x_1,x_2)
+ \lambda \psi(x_1,x_2).\]

Decompose $S_\lambda(G_1,G_2,G_3)$ as 
\begin{equation} \label{SlG}
\sum_{\nu}
a_{1,m_1} a_{2,m_2}a_{3,m_3} 
\int 
e^{i\Phi_\bm(x_1,x_2)}
\eta_{m_1}(x_1) \eta_{m_2}(x_2) \eta_{m_3}(\varphi(x_1,x_2))
\,d\bx
\end{equation}
with $\nu,\bm = (m_1,m_2,m_3)$ related as above. 
This sum is effectively taken over either a single index $m_3=m_3(\nu)$,
or over an empty set of indices $m_3$. Indices $\nu$ of the latter type may be dropped.

For each remaining $\nu$, the integral in \eqref{SlG} 
is $O(\lambda^{-K})$ for every $K$ unless
$|\nabla \Phi_{\nu} (z_\nu)| \le \lambda^{2\rho}\lambda^{1/2}$.

\begin{definition} \label{defn:sublevelset}
The sublevel set 
$\scripte_\sharp$ is the union of all $Q_\nu$
for which
$|\nabla \Phi_\nu (z_\nu)| \le \lambda^{2\rho}\lambda^{1/2}$.
\end{definition}

The contribution of each such $Q_\nu$ to $S_\lambda(G_1,G_2,G_3)$
is $O(|Q_\nu| \prod_j \norm{f_j}_\infty)$.
Therefore
\begin{equation}
|S_\lambda(G_1,G_2,G_3)|
= O\big(\lambda^{-K} + |\scripte_\sharp| \big) \prod_j \norm{f_j}_\infty.
\end{equation}
To complete the proof of Lemma~\ref{lemma:finalspecialform}
and hence the proofs of Theorems~\ref{thm:3} and \ref{thm2}, it suffices to show that
there exists $\tau_0>0$ such that
\begin{equation} \label{Esharpbound} 
|\scripte_\sharp| = O(\lambda^{-\tau_0}) \end{equation}
uniformly in all possible choices of functions $m_j\mapsto \xi_{j,m_j}$.

\section{Interlude}

A connection between oscillatory integral bounds of the form
\begin{equation} \label{interlude:osc}
\big| \int_B e^{i\lambda\psi} \prod_{j\in J} (f_j\circ\varphi_j)\big|
\le\Theta(\lambda)\prod_{j\in J}\norm{f_j}_\infty,
\end{equation}
where $\Theta(\lambda)\to 0$ as $|\lambda|\to\infty$, 
and bounds for Lebesgue measures of sublevel sets
\begin{equation} \label{interlude:subsetdefn}
\scripte=\{x\in B: \big| \psi(x) - \sum_j (g_j\circ\varphi_j)(x)\big|<\eps\},
\end{equation}
of the form
\begin{equation} \label{interlude:sub} |\scripte|\le \theta(\eps) \end{equation}
where $\theta(\eps)\to 0$ as $\eps\to 0^+$
with $\theta(\eps)$ independent of $(g_j)$,
is well known. The former implies the latter:
Fix an auxiliary compactly supported $C^\infty$ function 
$\zeta:\reals\to [0,\infty)$
satisfying $\zeta(t)=1$ for $|t|\le 1$.
Then
\[ |\scripte| 
\le \int_B \zeta( \eps^{-1} \big( \psi-\sum_j (g_j\circ\varphi_j))
= \int_\reals \widehat{\zeta}(t)
\big(\int_B e^{2\pi i (t/\eps)\psi(x)} 
\prod_{j\in J} (f_{j,t}\circ\varphi_j)(x)\,dx\big)\,dt
\]
with $f_{j,t} = e^{-2\pi i (t/\eps) g_j}$.
Rewriting this as
\[
\int_\reals \eps \widehat{\zeta}(\eps\lambda)
\Big(\int_B e^{2\pi i \lambda \psi(x)} 
\prod_{j\in J} (f_{j,\eps\lambda}\circ\varphi_j)(x)\,dx\Big)\,d\lambda
\]
and invoking  \eqref{interlude:osc} gives \eqref{interlude:sub}.
 
The analysis in this paper proceeds primarily in the opposite sense, 
using sublevel set bounds to deduce bounds for oscillatory integrals.
However, the sublevel sets that arise here are variants of those defined
by \eqref{interlude:subsetdefn}, 
in which $\nabla\psi$ appears, rather than $\psi$ itself. 
The reasoning in the preceding paragraph is elaborated
in \ref{section:FEqns} to establish an inverse theorem,
roughly characterizing tuples $(g_j:\reals^2\to\reals)$ for which
associated sublevel sets $\scripte = \{x\in B: |\sum_{j=1}^3 (g_j\circ\varphi_j(x)|<\eps\}$
are relatively large.

Sublevel set bounds of the type \eqref{interlude:sub},
with $\scripte$ defined by \eqref{interlude:subsetdefn}
have been established in certain cases \cite{christ_sublevel},
with $\varphi:_j:\reals^D\to\reals^1$ and $|J|$ arbitrarily large
relative to $D$,
as consequences of an extension of Szemer\'edi's theorem
due to Furstenberg and Katznelson.

\section{Proof of a sublevel set bound}\label{section:sublevel1}

Continue to denote by $\varphi_j,\psi_j$ the partial derivatives of $\varphi,\psi$
with respect to $x_j$ for $j=1,2$, respectively.
The following lemma is essentially a restatement of the desired bound
$|\scripte_\sharp| = O(\lambda^{-\tau_0})$, with the substitutions 
\begin{equation} h_j = \lambda^{-1} \sum_m \xi_{j,m}\one_{I_m} \end{equation}
and $\eps = \lambda^{\rho-\delta_0}$.

\begin{lemma} \label{lemma:sublevel}
Let $(\varphi,\psi)$ satisfy the hypotheses of Theorem~\ref{thm:3}.
Suppose that the ordered triple of mappings $(x_1,x_2)\mapsto (x_1,x_2,\varphi(x_1,x_2))$ is not 
equivalent to a linear system.
Then there exist $C<\infty$ and $\varrho>0$ with the following property.
Let $h_j$ be real-valued Lebesgue measurable functions, and let $\eps\in(0,1]$.
Let $\scripte$ be the set of all $(x,y)\in[0,1]^2$ that satisfy
\begin{equation} \label{sublevel}
\left\{
\begin{aligned}
&\big| h_1(x) + \varphi_1(x,y) h_3(\varphi(x,y)) + \psi_1(x,y)\big| \le\eps
\\
&\big|h_2(y) + \varphi_2(x,y) h_3(\varphi(x,y))  + \psi_2(x,y)\big| \le\eps.
\end{aligned}
\right. \end{equation}
Then
\begin{equation} \label{ineq:sublevel1} |\scripte| \le C\eps^\varrho. \end{equation}
\end{lemma}

The upper bound \eqref{Esharpbound} for the measure
of the set $\scripte_\sharp$ of Definition~{defn:sublevelset}
is an immediate consequence of \eqref{ineq:sublevel1},
so Lemma~\ref{lemma:sublevel} suffices to complete the proof of Theorem~\ref{thm:3}. 

The proof of Lemma~\ref{lemma:sublevel} 
relies on the next lemma, which should be regarded as
being well known, though it is more often formulated
only for the special case of families of polynomials of bounded degree.
We write $F_\omega(\bx) = F(\bx,\omega)$.

\begin{lemma} \label{lemma:scriptf}
Let $\Omega$ be a compact topological space, let $K\subset\reals^D$
be a compact convex set with nonempty interior, and let $V\subset\reals^D$
be an open set containing $K$.
Assume that $F_\omega\in C^\omega(V)$ for each $\omega\in\Omega$,
and that the mappings $(\bx,\omega)\mapsto \partial^\alpha_\bx F(\bx,\omega)$
are continuous for every multi-index $\alpha$.
Suppose further that none of the functions $F_\omega$ vanish identically on $K$.
Then there exist $\tau>0$ and $C<\infty$
such that for every $\eps>0$ and every $\omega\in\Omega$,
\begin{equation}
\big| \{\bx\in K: |F_\omega(\bx)|<\eps \} \big| \le C\eps^\tau.
\end{equation}
\end{lemma}

\begin{proof}
A simple compactness and slicing argument reduces matters to
the case in which $D=1$ and $K$ has a single element. 
A proof for that case is implicit in proofs of van der Corput's lemma
concerning one-dimensional oscillatory integrals, for
instance in \cite{steinbook} and \cite{zygmund}. 
For a derivation as a corollary of bounds for
oscillatory integrals, see \cite{carberywright}, page 14.
\end{proof}

The following simple result will be used repeatedly.

\begin{lemma} \label{lemma:squaremeasure}
Let $(X,\mu)$ and $(Y,\nu)$ be probability spaces.
Let $\lambda= \mu\times\nu$.
Let $E\subset X\times Y$ satisfy $\lambda(E)>0$.
Define \[\tilde E = \{x\in X: \nu(\{y: (x,y)\in E\})\ge \tfrac12 \lambda(E)\}.\]
There exists $y_0\in Y$ such that
\[ \lambda( \{ (x,y)\in E: \text{$x\in\tilde E$ and $(x,y_0)\in E$} \})
\ge \tfrac18 \lambda(E)^2.\]
\end{lemma}

\begin{proof}
Since \[\lambda(E\setminus (E\cap (\tilde E\times Y)) 
\le \int_Y \tfrac12 \lambda(E)\,d\nu = \tfrac12 \lambda(E),\]
one has $\lambda(E\cap (\tilde E\times Y))\ge \tfrac12\lambda(E)$
and therefore $\mu(\tilde E)\ge \tfrac12 \lambda(E)$.

Consider \[E^* = \{(x,y,y'):
\text{$x\in\tilde E$, $(x,y)\in E$, and $(x,y')\in E$,} \}\]
which satisfies $(\mu\times\nu\times\nu)(E^*) \ge \tfrac14 \lambda(E)^2$.
Indeed, by the Cauchy-Schwarz inequality,
\begin{align*}
\tfrac14 \lambda(E)^2
& \le \lambda(E \cap (\tilde E\times Y))^2
\\& 
= \Big(\int_{\tilde E} \int_Y \one_{E}(x,y)\,d\nu(y)\,d\mu(x)\Big)^2
\\& 
\le \int_{\tilde E} \int_Y \one_{E}(x,y)\,d\nu(y)\,\int_Y \one_{E}(x,y')\,d\nu(y')
\\& 
= (\mu\times\nu\times\nu)(E^*).
\end{align*}
The stated conclusion now follows from Fubini's theorem.
\end{proof}

\begin{proof}[Proof of Lemma~\ref{lemma:sublevel}]
There is a $C^\omega$ function $\kappa_1(x,t)$ satisfying
\[ \varphi(x,\kappa_1(x,t))\equiv t.\]
The hypothesis that $\partial\varphi/\partial x_2$ vanishes nowhere
implies that uniformly for all Lebesgue measurable sets $A$,
$|\{(x,t): (x,\kappa_1(x,t))\in A\}|$ is comparable to $|A|$.
Likewise, there exists $\kappa_2$ satisfying
\[ \varphi(\kappa_2(t,y),y)\equiv t\]
with
$|\{(y,t): (\kappa_2(t,y),y)\in A\}|$ comparable to $|A|$
for all measurable $A$.

Define 
\[ E_0 = \{ (x,y)\in \scripte: 
|h_3(\varphi(x,y))|\le 1 \}.  \]
For $\naturals\owns k>0$ let
\[ E_k = \{ (x,y)\in \scripte: 
2^{k-1}<|h_3(\varphi(x,y))|\le 2^k\}.  \]
It follows immediately from \eqref{sublevel} that
$|h_1(x)|$ and $|h_2(y)|$ are $O(2^k)$ whenever $(x,y)\in E_k$.
We will show that $|E_k| = O(2^{-k\varrho}\eps^\varrho)$.
Summation with respect to $k$ then yields  \eqref{ineq:sublevel1}.

Consider first $E_0$.
Define
\begin{equation*}
E'_0  = \big\{x\in[0,1]:
|\{t: (x,\kappa_1(x,t)))\in E_0\}| \ge c_0 |E_0| \big\}. 
\end{equation*}
By Lemma~\ref{lemma:squaremeasure},
there exists $t_0$ such that the set
\begin{equation} \label{Edoubleprime}
E_0'' = \{(x,t): 
\text{ $x\in E'_0$ 
and $(x,\kappa_1(x,t))\in E_0$
and $(x,\kappa_1(x,t_0))\in E_0$} \} 
\end{equation}
satisfies $|E_0''| \ge c|E_0|^2$,
where $c>0$ is a constant that depends on the function $\kappa_1$,
but not on $|E_0|$.

Define
$\alpha = h_3(t_0)$. By definition of $E_0$,
$\alpha \in[-1,1]$.
For every $(x,t)\in E_0$,
\begin{equation} \label{one_relation} 
\big| h_1(x) + \alpha\varphi_1(x,\kappa_1(x,t_0))  + \psi_1(x,\kappa_1(x,t_0))\big| \le\eps.
\end{equation}
Define
\begin{equation} \tilde h_1(x) = 
- \alpha \varphi_1(x,\kappa_1(x,t_0))  - \psi_1(x,\kappa_1(x,t_0)).\end{equation}
For any $(x,t)\in E_0''$, 
\begin{equation}
\big| \tilde h_1(x) + \varphi_1(x,\kappa_1(x,t))h_3(t)  + \psi_1(x,\kappa_1(x,t))\big| 
\le 2\eps  
\end{equation}
by \eqref{one_relation},
the inequality
\[ \big| h_1(x) + \varphi_1(x,\kappa_1(x,t))h_3(t)  + \psi_1(x,\kappa_1(x,t))\big| \le \eps
 \text{ whenever $(x,\kappa_1(x,t))\in E_0$}, \]  
and the triangle inequality.

The function $\tilde h_1$ belongs to a compact
family of $C^\omega$ functions  of $x\in[0,1]$, parametrized by $\alpha,t_0$.
This family is defined solely in terms of $\varphi,\psi$.
Defining 
\begin{equation} E_0^{(1)} = \{(x,\kappa_1(x,t)): (x,t)\in E_0''\},\end{equation}
one has
$|E_0^{(1)}| \ge c|E_0|^2$
and
\begin{equation} \big| \tilde h_1(x) + \varphi_1(x,y) h_3(\varphi(x,y))  + \psi_1(x,y)\big| \le 2\eps 
\ \text{ for all $(x,y)\in E_0^{(1)}$}. \end{equation}

Repeating this reasoning with the roles of the two coordinates $x,y$ interchanged
and with $E_0$ replaced by $E_0^{(1)}$,
we conclude that 
there exist a subset $E_0^{(2)}\subset E_0^{(1)} \subset [0,1]^2$ 
satisfying $|E_0^{(2)}| \gtrsim |E_0|^4$,
and a function $\tilde h_2$ belonging to a compact family of $C^\omega$ functions
defined solely in terms of $\varphi,\psi$,
that satisfy
\begin{equation*}
\big|\tilde h_2(y) + \varphi_2(x,y) h_3(\varphi(x,y))  + \psi_2(x,y)\big| \le 2 \eps
\ \text{ for all $(x,y)\in E_0^{(2)}$}. \end{equation*}

The condition that $(x,y)\in E_0$ 
directly provides an upper bound $|h_3(\varphi(x,y))|\le 1$.
It also implies upper bounds for $|h_j(x,y)|\le C<\infty$
for $j=1,2$ via the inequalities \eqref{sublevel} and the assumption that $\eps\le 1$.

A third iteration of this reasoning yields
a set $E_0^{(3)}\subset E_0^{(2)}$
and a function $\tilde h_3$ 
of the special form
\[ \tilde h_3(x) =   -[\varphi_1(\kappa_2(x,s),s)]^{-1} 
\big( -\alpha - \psi_1(\kappa_2(x,s),s)\big)\]
for some parameters $s\in[0,1]$ and $\alpha\in\reals$,
satisfying
\begin{equation} 
\left\{
\begin{aligned}
&\big| \tilde h_1(x) + \varphi_1(x,y) \tilde h_3(\varphi(x,y)) + \psi_1(x,y)\big| \le C\eps
\\
&\big|\tilde h_2(y) + \varphi_2(x,y) \tilde h_3(\varphi(x,y))  + \psi_2(x,y)\big| \le C\eps
\end{aligned}
\right. \end{equation}
for all $(x,y)\in E_0^{(3)}$, with $|E_0^{(3)}| \ge c|E_0|^8$.
Again, $\tilde h_3$ belongs 
to a compact family of $C^\omega$ functions
that is defined in terms of $\varphi,\psi$ alone.

Define
\begin{equation} \left\{
\begin{aligned} 
h_1^{s,\alpha}(x) &= - \alpha\varphi_1(x,\kappa_1(x,s)) h_3(s) - \psi_1(x,\kappa_1(x,s))
\\
 h_2^{s,\alpha}(y) &= - \alpha\varphi_2(\kappa_2(y,s),y) h_3(s)  + \psi_2(\kappa_2(y,s),y)
\\
 h_3^{s,\alpha}(u) &= -[\varphi_1(\kappa_2(u,s),s)]^{-1} 
\big( -\alpha - \psi_1(\kappa_2(u,s),s)\big).
\end{aligned} \right. \end{equation}
Let $\scriptf$ be the family of $\reals^2$--valued
$C^\omega$ functions $F_{(\bs,\alpha)}$ of $(x,y)\in[0,1]^2$,
parametrized by $(\bs,\alpha) = (s_1,s_2,s_3,\alpha_1,\alpha_2,\alpha_3)$ with each $s_j\in[0,1]$,
$\alpha_1,\alpha_2\in[-1,1]$, and $\alpha_3\in[-C,C]$ for some appropriate $C<\infty$, 
defined by
\begin{equation}
F_{(\bs,\alpha)}(x,y)
= 
\begin{pmatrix}
h_1^{s_1,\alpha_1}(x) + \varphi_1(x,y) h_3^{s_3,\alpha_2}(\varphi(x,y)) + \psi_1(x,y)
\\
h_2^{s_2,\alpha_2}(y) + \varphi_2(x,y) h_3^{ s_3,\alpha_3 }(\varphi(x,y))  + \psi_2(x,y)
\end{pmatrix}.
\end{equation}

There exist no real-valued functions $h_j^\sharp$ in $C^1$ that satisfy
\begin{equation} 
\left\{
\begin{aligned}
& h_1^\sharp(x) + \varphi_1(x,y) h_3^\sharp(\varphi(x,y)) + \psi_1(x,y)\equiv 0
\\
& h_2^\sharp(y) + \varphi_2(x,y) h_3^\sharp(\varphi(x,y))  + \psi_2(x,y)\equiv 0
\end{aligned}
\right. \end{equation}
on $[0,1]^2$.
For if there were, then defining $H_j$  to be an antiderivative of $h_j^\sharp$,
one would have
\[ \nabla_{x,y} \big( \psi(x,y) - H_1(x)-H_2(y)-H_3(\varphi(x,y))\big)\equiv 0,\]
contradicting the nondegeneracy hypothesis on $(\varphi,\psi)$.
Therefore for any $(\bs,\alpha)$, the function $F_{(\bs,\alpha)}$ does not vanish
identically as a function of $(x,y)\in [0,1]^2$.
Lemma~\ref{lemma:scriptf} can now be applied to conclude that
$|E_0^{(3)}|\le C\eps^\tau$,
with $C<\infty$ and $\tau>0$ depending only on $\varphi,\psi$.
Threfore
\begin{equation} |E_0| \le C'\eps^{\tau/8} \end{equation}
for another constant $C'<\infty$.
This completes the analysis of $E_0$.

\smallskip
The same analysis yields an upper bound of the form
$|E_k| \le C2^{-k\varrho} \eps^\varrho$, uniformly for all $k>0$. 
Indeed, define $\tilde h_j = 2^{-k}h_j$ for $j\in\{1,2,3\}$, 
and set $\tilde\eps = 2^{-k}\eps$, to obtain
\begin{equation} 
\left\{
\begin{aligned}
&\big| \tilde h_1(x) + \varphi_1(x,y) \tilde h_3(\varphi(x,y)) 
+ 2^{-k} \psi_1(x,y)\big| \le \tilde \eps
\\
&\big|\tilde h_2(y) + \varphi_2(x,y) \tilde h_3(\varphi(x,y))  + 2^{-k} \psi_2(x,y)\big| \le\tilde\eps
\end{aligned}
\right. \end{equation}
for all $(x,y)\in E_k$.

Compactify by considering the system of inequalities
\begin{equation} \label{sublevelk}
\left\{
\begin{aligned}
&\big| h_1(x) + \varphi_1(x,y) h_3(\varphi(x,y)) + r \psi_1(x,y)\big| \le \eps'
\\
&\big|h_2(y) + \varphi_2(x,y) h_3(\varphi(x,y))  + r \psi_2(x,y)\big| \le \eps'
\end{aligned}
\right. \end{equation}
for arbitrary $r\in [0,1]$
and $\eps'\in[0,\eps_0]$. We may assume that $\eps_0$ is as small as desired.

The situation differs from the analysis of $E_0$ in one respect: 
For $(x,y)\in E_k$, 
\begin{equation} \label{h3lower} \tfrac12 \le |h_3(\varphi(x,y))| \le 1.  \end{equation}
The lower bound, of which we had no analogue
in the analysis of $E_0$, will be crucial below.

By repeating the above reasoning, we find that if 
$h_j$ satisfy \eqref{sublevelk} and \eqref{h3lower}
on some set $\scripte'$ then there exist functions $\tilde h_j$ drawn
from a compact family of $C^\omega$ functions associated to $\varphi,\psi$,
that satisfy
\begin{equation} 
\left\{
\begin{aligned}
&\big| \tilde h_1(x) + \varphi_1(x,y) \tilde h_3(\varphi(x,y)) + r \psi_1(x,y)\big| \le C\eps'
\\
&\big|\tilde h_2(y) + \varphi_2(x,y) \tilde h_3(\varphi(x,y))  + r \psi_2(x,y)\big| \le C\eps'
\end{aligned}
\right. \end{equation}
for all $(x,y)\in\tilde\scripte'$, with $|\tilde\scripte'|\ge c|\scripte'|^8$.
Moreover, the lower bound \eqref{h3lower} implies that
$\norm{h_3}_{C^0} \ge \tfrac14$, provided that $\eps_0$ is sufficiently small.

There exists no solution $(\tilde h_j: j\in\{1,2,3\})$
of the system of equations
\begin{equation*} 
\left\{
\begin{aligned}
& \tilde h_1(x) + \varphi_1(x,y) \tilde h_3(\varphi(x,y)) + r \psi_1(x,y)=0
\\
&\tilde h_2(y) + \varphi_2(x,y) \tilde h_3(\varphi(x,y))  + r \psi_2(x,y)=0
\end{aligned}
\right. \text{ on $[0,1]^2$.} \end{equation*}
For $r\ne 0$, this follows from the same reasoning as given above for $r=1$
in the analysis of $E_0$.
For $r=0$,  the simplified system 
\begin{equation} 
\left\{
\begin{aligned}
&  h_1(x) + \varphi_1(x,y) h_3(\varphi(x,y)) \equiv 0
\\
& h_2(y) + \varphi_2(x,y) h_3(\varphi(x,y)) \equiv 0
\end{aligned}
\right. \end{equation}
admits no solutions with $h_3$ vanishing nowhere.
For if there were such a solution, 
defining $H_j$ to be an antiderivative of $\tilde h_j$
and adjusting $H_1$ by an appropriate additive constant,
 \begin{equation}
H_3(\varphi(x,y))  + H_1(x)+H_2(y) \equiv 0.
\end{equation}
If $H'_3=h_3$ vanishes nowhere,
this contradicts the hypothesis that $(x_1,x_2,\varphi(x_1,x_2))$
is not equivalent to a linear system.
Thus $h_3$ must vanish, contradicting
the lower bound \eqref{h3lower}. 

By the same reasoning as in the case $k=0$, it follows that
$|\scripte'|\le C(\eps')^\varrho$ for a certain exponent $\varrho>0$. 
Applying this with $\scripte' = E_k$
and $\eps' = 2^{-k}\eps$ gives
$|E_k| \le C 2^{-k\varrho} \eps^\varrho$.
Summing over all $k\ge 0$  completes the proof of the lemma.
\end{proof}

\section{Proof of Theorem~\ref{mainthm}} \label{section:proofofmainthm}

In the deduction of Theorem~\ref{thm2} from Theorem~\ref{thm:3},
we were able to immediately gain a factor of $|\lambda|^{-1/2}$
upon integration with respect to $x_3$,
reducing matters to a self-contained situation in which a supplementary
factor of $|\lambda|^{-\delta}$ was to be gained.
In the framework of Theorem~\ref{mainthm}, the analysis
does not split cleanly into two separate steps.

Let $\tilde\eta$ be a $C_0^\infty$ cutoff function
supported in a small neighborhood of $[0,1]^3$
and identically equal to $1$ on $[0,1]^3$,
such that $\phi$ is real analytic and continues
to satisfy the linear independence hypotheses of the theorem
in a neighborhood of the support of $\tilde\eta$.
Modify the definition of $T_\lambda^\phi$ to
\[ T_\lambda^\phi(\bff) = \int_{\reals^3} e^{i\lambda\phi(\bx)} \prod_{j=1}^3 f_j(x_j)
\tilde\eta(\bx)\,d\bx.\]
We will show that this modified form satisfies the indicated upper bound
as $\lambda\to+\infty$.

It suffices to prove the conclusion \eqref{ineq:main}
with $\prod_j \norm{f_j}_2$
replaced by $\prod_j \norm{f_j}_\infty$
on the right-hand side.
Indeed, the assumption that $\frac{\partial^2\phi}{\partial x_1\partial x_2}$
vanishes nowhere
implies that
\[ \big| \int_{[0,1]^2} e^{i\lambda\phi(x_1,x_2,x_3)} f_1(x_1)f_2(x_2)\,dx_1\,dx_2\big|
\le C|\lambda|^{-1/2}\norm{f_1}_2\norm{f_2}_2\]
uniformly for all $x_3$. Therefore
\[ |T_\lambda^\phi(\bff)| \le C|\lambda|^{-1/2} \norm{f_1}_2\norm{f_2}_2\norm{f_3}_1.\]
Therefore by interpolation, it suffices to establish the conclusion
with $\norm{f_1}_2\norm{f_2}_2\norm{f_3}_\infty$ on the right-hand side
and some exponent $\gamma > \tfrac12$.
By repeating this reduction with the roles of $f_2,f_3$ interchanged,
interpolating between bounds in terms of
$\norm{f_1}_2\norm{f_2}_1\norm{f_3}_\infty$
and
$\norm{f_1}_2\norm{f_2}_\infty\norm{f_3}_\infty$
to conclude a bound in terms of
$\norm{f_1}_2\norm{f_2}_2\norm{f_3}_\infty$,
we infer that it suffices to establish the conclusion 
in terms of $\norm{f_1}_2\norm{f_2}_\infty\norm{f_3}_\infty$.
Repeating this step once more reduces matters to a bound in terms of
the product of $L^\infty$ norms. 
Note that this reasoning requires nonvanishing of all three
mixed second partial derivatives
$\frac{\partial^2\phi}{\partial x_j\partial x_k}$,
hence does not apply to $\phi = x_1x_2 + x_2x_3$.

Write $e_\xi(x) = e^{i\xi x}$.
There exists a constant $A$ depending only on $\phi$ and on the
choice of $\tilde\eta$ such that
\begin{equation}
|T_\lambda^\phi(e_\xi,f_2,f_3)|
\le C_N |\xi|^{-N}\norm{f_2}_1 \norm{f_3}_1
\ \text{ for every $|\xi|\ge A\lambda$}
\end{equation}
for every $N<\infty$ and every $\lambda\ge 1$.
This is proved by writing
\[ e^{i\xi x_1 + i\lambda\phi(\bx)}
= \Big([i\xi + i\lambda \frac{\partial\phi}{\partial x_1}(\bx)]^{-1} 
\frac{\partial}{\partial x_1}\Big)^N e^{i\xi x_1 + i\lambda\phi(\bx)}\]
and integrating by parts $N$ times with respect to $x_1$
while holding $x_2,x_3$ fixed. 
The same holds with the role of $x_1$ taken by $x_2$ or $x_3$.
As a consequence, it suffices to analyze
$T_\lambda^\phi(\bff)$ under the bandlimitedness assumption that for
each $j\in\{1,2,3\}$, $\widehat{f_j}(\xi)=0$ whenever $|\xi|\ge A\lambda$.
We assume this for the remainder of the proof of Theorem~\ref{mainthm}.

Suppose that each function $f_j$ satisfies $\norm{f_j}_\infty \le 1$.
Expand each $f_j$ in the form 
\begin{equation} \label{fjdecompagain}
f_j(x) = \sum_m \eta_m(x) \sum_{k\in\integers} a_{j,m,k} e^{i\pi\lambda^{1/2}kx}
\end{equation}
with
\begin{equation}
\sum_k |a_{j,m,k}|^2 \le C<\infty \ \text{ uniformly in $j,m,\lambda$.}
\end{equation}
Decompose $f_j = g_j+h_j+F_j$
where $F_j$ is the sum of those terms with $|k|>\lambda^{1/2}\lambda^\rho$,
$h_j$ is the sum of those terms with $|k|\le\lambda^{1/2}\lambda^\rho$
and $|a_{j,m,k}| \le \lambda^{-\sigma}$,
and $g_j$ is the sum of all remaining terms.
From the $O(\lambda)$--bandlimitedness condition of the preceding paragraph,
it follows that 
\begin{equation}
\norm{F_j}_2 = O(\lambda^{-N})\ \text{ for every $N<\infty$.}
\end{equation}
$T_\lambda^\phi(\bff)$ equals
$T_\lambda^\phi(g_1+h_1,g_2+h_2,g_3+h_3)$
plus terms involving one or more of the functions $F_j$.
Each of the latter terms is $O(\lambda^{-N})$
for every $N<\infty$, and may consequently be disregarded henceforth.
Thus henceforth, $f_j = g_j+h_j$
and $|k|\le \lambda^{1/2}\lambda^\rho$ in \eqref{fjdecompagain}.

Expand
\begin{equation} \label{msum.mainthm}
T_\lambda^\phi(\bff) = 
\sum_{\bm} 
\sum_{\bk} 
\prod_{j=1}^3 a_{j,m_j,k_j} 
\int e^{i\Phi_\bk(\bx)} \eta_{\bm}(\bx)  \,d\bx
\end{equation}
with $\eta_{\bm}(\bx) =  \prod_{l=1}^3 \eta_l (x_l)$ 
and with the net phase function 
\begin{equation}
\Phi_\bk(\bx) =  \pi\lambda^{1/2} \bk\cdot\bx + \lambda\phi(\bx),
\end{equation}
whose partial derivatives satisfy
\begin{equation} \label{partials.mainthm}
\lambda^{-1/2} \frac{\partial\Phi_\bk}{\partial x_j} = 
\pi k_j + \lambda^{1/2} \frac{\partial \phi}{\partial x_j}
\ \text{ for each $j\in\{1,2,3\}$.}
\end{equation}
$\frac{\partial\Phi_\bk}{\partial x_j}(\bx)$  
depends only on the single component $k_j$ of $\bk = (k_1,k_2,k_3)$;
this will be exploited.
We will establish an upper bound for the sum of absolute values
\begin{equation} \label{sumabsolutevalues} \sum_{\bm} \sum_{\bk} 
\prod_{j=1}^3 |a_{j,m_j,k_j}|
\cdot \big|\int e^{i\Phi_\bk(\bx)} \eta_{\bm}(\bx)  \,d\bx\big|.  \end{equation}

For any $(\bm,\bk)$,
\begin{equation} \int e^{i\Phi_\bk(\bx)} \eta_{\bm}(\bx)  \,d\bx = O(\lambda^{-3/2}).
\end{equation}
A tuple of indices $(\bm,\bk)$ is said to be nonstationary if 
\begin{equation}  \label{nonstationary} 
|\nabla\Phi_\bk(z_\bm)| \ge \lambda^\rho\lambda^{1/2},
\end{equation}
and otherwise is said to be stationary.
For any nonstationary $(\bm,\bk)$, repeated integration by parts gives 
\begin{equation} \int e^{i\Phi_\bk(\bx)} \eta_{\bm}(\bx)  \,d\bx = O(\lambda^{-N}) 
\ \text{ for every $N<\infty$.} 
\end{equation}
The total number of ordered pairs $(\bm,\bk)$ is $O( (\lambda^{1/2})^6)=O(\lambda^3)$.
Therefore the total contribution made to \eqref{sumabsolutevalues}
by all nonstationary $(\bm,\bk)$ is $O(\lambda^{-M})$ for all $M<\infty$.

For each $(m_1,m_2,k_1)$
there are at most $O(\lambda^\rho)$ values of $m_3$ that satisfy
\begin{equation} \label{stationary_1}
|\frac{\partial\Phi_\bk}{\partial x_1}(z_\bm)| \le \lambda^{1/2}\lambda^\rho,
\end{equation}
with the standing notation $\bm = (m_1,m_2,m_3)$.
The condition \eqref{stationary_1} is independent of $k_2,k_3$, since
$\frac{\partial\Phi_\bk}{\partial x_1}(z_\bm)$ 
does not depend on these quantities.
The derivative
$\frac{\partial}{\partial x_3}\frac{\partial \Phi_\bk}{\partial x_1}$ 
vanishes nowhere and has absolute value $\ge c\lambda$. Therefore
for each $(x_1,x_2,k_1)$,
$\big|\frac{\partial\Phi_\bk}{\partial x_1}(x_1,x_2,x_3)\big|\le \lambda^{1/2}\lambda^\rho$
only on a single interval whose length is $O(\lambda^{-1/2}\lambda^\rho)$.
Such an interval intersects the support of $\eta_{m_3}$
for at most $O(\lambda^\rho)$ values of $m_3$.
Thus for each $(m_1,m_2,k_1)$,
for every $m_3$ with at most $O(\lambda^\rho)$ exceptions,
$(\bm,\bk)$ is nonstationary for every choice of $k_2,k_3$.

Likewise, for any $(m_1,m_2,k_1,m_3)$,
there are most $O(\lambda^\rho)$ values of $k_2$
for which 
$\big|\frac{\partial\Phi_\bk}{\partial x_2}(z_\bm)\big|\le \lambda^{1/2}\lambda^\rho$,
and 
most $O(\lambda^\rho)$ values of $k_3$
for which 
$\big|\frac{\partial\Phi_\bk}{\partial x_3}(z_\bm)\big|\le \lambda^{1/2}\lambda^\rho$.
Thus for each $(m_1,m_2,k_1,m_3)$
there are at most $O(\lambda^\rho)$ values of $k_2$ for which there exists $k_3$
such that $(\bm,\bk)$ is stationary; and for any such $k_2$,
there are at most $O(\lambda^\rho)$ such $k_3$.
Therefore for each $(m_1,m_2,k_1)$,
there are at most $O(\lambda^{C\rho})$ values of $k_2$
for which there exists $(m_3,k_3)$ such that $(\bm,\bk)$
is stationary; and for any such $k_2$, there are at most $O(\lambda^{C\rho})$
such pairs $(m_3,k_3)$.

Decompose $T_\lambda^\phi(\bff) = T_\lambda^\phi(f_1,f_2,g_3) + T_\lambda^\phi(f_1,f_2,h_3)$
and consider the second summand. All coefficients arising in the expansion
of $h_3$ satisfy $|a_{3,m_3,k_3}| \le \lambda^{-\sigma}$.
Therefore 
\begin{align*} 
|T_\lambda^\phi(f_1,f_2,h_3)|
&\le O(\lambda^{-N}) +
C\lambda^{-3/2} \sum_{m_1,m_2} \sum_{k_1} \sum_{m_3,k_2,k_3} 
|a_{1,m_1,k_1}a_{2,m_2,k_2}a_{3,m_3,k_3}|
\\
&\le O(\lambda^{-N}) + C\lambda^{-3/2}\lambda^{-\sigma} \sum_{m_1,m_2} \sum_{k_1} \sum_{m_3,k_2,k_3} 
|a_{1,m_1,k_1}a_{2,m_2,k_2}|
\end{align*}
for every $N<\infty$,
with the inner sums over $m_3,k_2,k_3$ extending only over those indices such that
$(\bm,\bk)$ is stationary.
Thus
\begin{equation} 
|T_\lambda^\phi(f_1,f_2,h_3)|
\le O(\lambda^{-N})
+ O(\lambda^{-3/2}\lambda^{C\rho -\sigma}) \sum_{m_1,m_2} \sum_{k_1,k_2} 
|a_{1,m_1,k_1}a_{2,m_2,k_2}|,\end{equation}
with the inner sum taken only over those $(k_1,k_2)$ for which there exist $m_3,k_3$
such that $(\bm,\bk)$ is stationary.

For each $(m_1,m_2,k_1)$, at most $O(\lambda^{2\rho})$ indices $k_2$
appear in this sum. Likewise, for each $(m_1,m_2,k_2)$, 
at most $O(\lambda^{2\rho})$ indices $k_1$ appear.
For each $j,m_j$, the sequence $a_{j,m_j,k_j}$ belongs to $\ell^2$ with respect
to $k_j$, with norm $O(1)$. Therefore an application of Cauchy-Schwarz to the
inner sum gives an upper bound
\[ O(\lambda^{-3/2}\lambda^{C\rho -\sigma}) \sum_{m_1,m_2} O(1) + O(\lambda^{-N}),\]
which is
$O(\lambda^{-1/2}\lambda^{C\rho -\sigma}) + O(\lambda^{-N})$ 
since there are $O(\lambda^{2/2})$ ordered pairs $(m_1,m_2)$.
The conclusion is that
\begin{equation}
|T_\lambda^\phi(f_1,f_2,f_3)| \le |T_\lambda^\phi(f_1,f_2,g_3)
+ O(\lambda^{-1/2}\lambda^{C\rho -\sigma}).
\end{equation}

Repeating this reasoning with indices permuted gives
\[ |T_\lambda^\phi(f_1,f_2,g_3)| \le |T_\lambda^\phi(f_1,g_2,g_3)|
+ O(\lambda^{-1/2}\lambda^{C\rho -\sigma}),\]
and after one more repetition, 
\begin{equation} |T_\lambda^\phi(\bff)| \le |T_\lambda^\phi(\bg)|
+ O(\lambda^{-1/2}\lambda^{C\rho -\sigma}),\end{equation}
where each component of $\bg = (g_1,g_2,g_3)$ satisfies \eqref{fjdecompagain}
with at most $O(\lambda^{2\sigma})$ nonzero coefficients
$a_{j,m_j,k_j}$ for each $m_j$.

It remains to treat $T_\lambda^\phi(\bg)$. 
By the same reasoning as in the proof of Theorem~\ref{thm:3}, in order to complete the
proof of Theorem~\ref{mainthm} it now suffices to prove an appropriate
upper bound for measures of associated sublevel sets,
formulated below as Lemma~\ref{lemma:sublevel2}.

\section{Sublevel set analysis for Theorem~\ref{mainthm}} \label{section:sublevel2}

Write $\nabla_j =  \frac{\partial}{\partial x_j}$
and $\nabla^2_{j,k} 
=  \frac{\partial^2}{\partial x_j \partial x_k}$.

\begin{lemma} \label{lemma:sublevel2}
Suppose that for every distinct pair of indices $j\ne k\in\{1,2,3\}$, 
the mixed partial derivative
$\frac{\partial^2\phi}{\partial x_j \partial x_k}$
vanishes nowhere on the support of $\tilde\eta$.
Then there exist $\delta>0$ and $C<\infty$ such that
for any $\eps\in(0,1]$ and
any Lebesgue measurable real-valued functions $h_1,h_2,h_3$, 
the sublevel set
\begin{equation}
\scripte = \big\{\bx: |\nabla_j\phi(\bx)-h_j(x_j)|\le  \eps
\ \text{ for each $j\in\{1,2,3\}$}\big\}
\end{equation}
satisfies
\begin{equation} |\scripte| \le C\eps^{1+\delta}.  \end{equation}
\end{lemma}

Here we seek a bound with an exponent strictly greater than $1$,
whereas in Lemma~\ref{lemma:sublevel} above, 
we merely sought an exponent greater than $0$.
Invoking Lemma~\ref{lemma:sublevel2} with $\eps = \lambda^{-1/2}\lambda^{C\rho}$,
for $\rho$ sufficiently small relative to $\delta$,
completes the proof of Theorem~\ref{mainthm}.

We may assume that $h_j(x_j)$ belongs to the range of $\nabla_j\phi$
for each index $j$.
By the implicit function theorem
together with the hypothesis $\frac{\partial^2\phi}{\partial x_1\partial x_3}\ne 0$,
there exists a $C^\omega$ function $\kappa_0$ satisfying
\begin{equation}
\nabla_1\phi(x_1,x_2,\kappa_0(x_1,x_2,x_3)) = x_3.
\end{equation}
Differentiating this equation with respect to $x_2$ gives
\begin{equation*}
\frac{\partial \kappa_0(\bx)}{\partial x_2}
=  \frac{-\nabla^2_{1,2} \phi}{\nabla^2_{1,3}\phi}(x_1,x_2,\kappa_0(\bx)).
\end{equation*}
Therefore since $\nabla^2_{1,2}\phi$ never vanishes,
the mapping $\bx\mapsto (x_1,\kappa_0(\bx),x_3)$ is locally invertible.

Define
\begin{equation} 
\kappa(x_1,x_2) = \kappa_0(x_1,x_2,t)\ \text{with $t=h_1(x_1)$.}
\end{equation}
Thus for each $x_1$, $x_2\mapsto \kappa(x_1,x_2)$ is a $C^\omega$ function that satisfies
\begin{equation} \label{12.5} 
\nabla_1\phi(x_1,x_2,\kappa(x_1,x_2)) = h_1(x_1).\end{equation}
This function of $x_2$ is drawn from a compact family
of $C^\omega$ functions that is specified in terms of $\phi$ and
is parametrized by $(x_1,t)$ with $x_1\in[0,1]$ and
$|t|\le \norm{\nabla_1\phi}_{C^0([0,1])} +1$.

Write $\by = (y_1,y_2)\in\reals^2$.
By the nonvanishing of $\frac{\partial}{\partial x_3} \nabla_1 \phi
= \nabla^2_{1,3}\phi$, the relation
$\nabla_1\phi(y_1,y_2,x_3) = h_1(y_1) + O(\eps)$
implies that $|x_3-\kappa(\by)| = O(\eps)$.
Thus $|\scripte| = O(\eps)$. 

Define
\begin{equation} \label{scripteprime}
\scripte'_1 = \big\{\by\in[0,1]^2: 
|(\nabla_j\phi)(\by,\kappa(\by))- h_j(y_j)|\le C_0\eps
\ \text{ for each $j\in\{2,3\}$} \big\}
\end{equation}
with the convention $y_3=\kappa(\by)$ and with $C_0$ a sufficiently
large constant. Then
\[ \scripte\subset\{(x_1,x_2,x_3): 
(x_1,x_2)\in\scripte'_1 \text{ and } 
|x_3-\kappa(x_1,x_2)|\le C_0\eps\}.\]

Define $\scripte'_2$ in the same way that $\scripte'_1$
was defined, but with the roles of the coordinates $x_1$ and $x_2$ interchanged,
relying on the assumption that $\nabla^2_{2,3}\phi$ never vanishes
and replacing $\kappa$ in the construction by the corresponding function
$\kappa_2$ defined by
\begin{equation} 
\kappa_2(x_1,x_2) = \kappa_0(x_1,x_2,t)\ \text{with $t=h_2(x_2)$.}
\end{equation}
Then
\[ \scripte\subset\{(x_1,x_2,x_3): 
(x_1,x_2)\in\scripte'_2 \text{ and } 
|x_3-\kappa_2(x_1,x_2)|\le C_0\eps\}.\]

Define
\begin{equation} \scripte' = \scripte'_1\cap\scripte'_2.\end{equation}
Then
\begin{equation} \scripte \subset\{(x_1,x_2,x_3): 
(x_1,x_2)\in\scripte' \text{ and } 
|x_3-\kappa(x_1,x_2)| + |x_3-\kappa_2(x_1,x_2)|
\le 2C_0\eps\},\end{equation}
whence
\begin{equation} |\scripte| \le C\eps |\scripte'|.  \end{equation}
In order to complete the proof of Lemma~\ref{lemma:sublevel2},
it remains only to show that $|\scripte'|$ is suitably
small, as asserted in the next lemma.

\begin{lemma} \label{lemma:last2Dsublevel}
Suppose that $\phi\in C^\omega$ is not rank one degenerate.
Suppose that for every pair of distinct indices $j\ne k$, 
$\nabla^2_{j,k}\phi$ vanishes nowhere in a neighborhood of $[0,1]^3$.
Then there exists $\delta$ such that
for any measurable functions $h_j$ and any $\eps>0$,
the set $\scripte'$ introduced in \eqref{scripteprime} satisfies
$|\scripte'| = O(\eps^\delta)$.
\end{lemma}

\begin{proof}
The first step is to replace $h_2$ by a $C^\omega$ function,
drawn from a compact family specified in terms of $\phi$ alone.
There exists a set $E\subset\reals^1$ satisfying $|E|\gtrsim |\scripte'|$
such that for each $y_2\in E$, 
\[|\{y_1: (y_1,y_2)\in\scripte'\}| \gtrsim |\scripte'|.\]
Therefore the subset $\scripte''\subset\scripte'$
defined by $\scripte''=\{(y_1,y_2)\in\scripte': y_2\in E\}$
satisfies $|\scripte''| \gtrsim |\scripte'|^2$.

By Fubini's theorem, there exists $\bar y_1$ such that
\[|\{y_2\in E: (\bar y_1,y_2)\in\scripte'\}| \gtrsim |E| \gtrsim |\scripte'|.\]
Consider the relation
$\nabla_2\phi(\bar y_1,y_2,\kappa(\bar y_1,y_2)) =  h_2(y_2) + O(\eps)$ 
for those $y_2\in E$ satisfying $(\bar y_1,y_2)\in \scripte'$.
Since $\phi\in C^\omega$
and $\kappa(\bar y_1,y_2)$ is a $C^\omega$ function of $y_2$,
drawn from a compact family specified in terms of $\phi$ alone,
this relation expresses $h_2(y_2)$ as $\tilde h_2(y_2) + O(\eps)$
for these values of $y_2$,
with $\tilde h_2$ drawn from another compact family of $C^\omega$ functions.
Therefore $h_2$ can be replaced by $\tilde h_2$
in the definition of $\scripte'$, at the cost of replacing
$\scripte'$ by its subset $\scripte''$ and modifying the constant $C_0$
in that definition.

In the preceding two paragraphs, 
the roles of the variables $y_1$ and $y_2$ can be interchanged,
since the definition of $\scripte' = \scripte'_1\cap\scripte'_2$
is invariant under this interchange.
Therefore by replacing $\scripte''$ by an appropriate subset $\scripte'''$,
satisfying $|\scripte'''| \gtrsim |\scripte''|^2 \gtrsim |\scripte|^4$,
we can reduce matters to the case in which $h_1$ is also 
drawn from a compact set of $C^\omega$ functions specified solely in terms of $\phi$.

Return to the equation
$\nabla_1\phi(x_1,x_2,\kappa(x_1,x_2)) = h_1(x_1)$,
restricted now to $(x_1,x_2)\in\scripte'''$.
Since the right-hand side differs from a $C^\omega$ function by $O(\eps)$
on $\scripte'''$,
and since $\nabla_3(\nabla_1\phi)$ never vanishes,
the implicit function theorem can now
be applied 
to conclude that $\kappa$ differs on $\scripte'''$ by $O(\eps)$
from a $C^\omega$ function, drawn from an appropriate compact family. 
Therefore by \eqref{12.5}, 
$\kappa$ can in turn be replaced by a $C^\omega$ function of $\by\in[0,1]^2$,
at the price of replacing $C_0$ by a yet larger constant.

$\kappa$ was defined by the relation
$\nabla_1\phi(x_1,x_2,\kappa(x_1,x_2)) -h_1(x_1)=0$.
Differentiating this equation with respect to $x_2$ gives
\[ \nabla^2_{1,2}\phi(x_1,x_2,\kappa(x_1,x_2)) 
+ \nabla^2_{1,3}\phi(x_1,x_2,\kappa(x_1,x_2))
\, \frac{\partial\kappa(x_1,x_2)}{\partial x_2}=0.\]
Since 
$\nabla^2_{1,2}\phi$ vanishes nowhere by hypothesis,
it follows that
$\frac{\partial\kappa(x_1,x_2)}{\partial x_2}$ vanishes nowhere.
Therefore the relation
\[ x_3=\kappa(x_1,x_2) \Leftrightarrow x_2 = \tilde\kappa(x_1,x_3)\]
defines a $C^\omega$ function $\tilde\kappa$.

The relation 
\[\nabla_3\phi(x_1,x_2,\kappa(x_1,x_2) ) 
=h_3(\kappa(x_1,x_2))+O(\eps)
\ \text{ for $(x_1,x_2)\in\scripte'''$} \]
can be rewritten with the aid of  $\tilde\kappa$ as 
\begin{equation}
\nabla_3\phi(x_1,\tilde\kappa(x_1,x_3),x_3))=h_3(x_3)+O(\eps)
\ \text{ when $(x_1,\tilde\kappa(x_3))\in\scripte'''$.} 
\end{equation}
Therefore $h_3$ can likewise be replaced by
a $C^\omega$ function drawn from an appropriate compact set.

We have thus shown that under the hypotheses of Lemma~\ref{lemma:sublevel2},
there exist $\scripte'''\subset\reals^2$
satisfying $|\scripte| \le C\eps|\scripte'''|^{1/4}$
and 
$C^\omega$ functions $\tilde h_j,\kappa$ belonging to appropriate compact families
such that 
with $x_3=\kappa(x_1,x_2)$,
$|\nabla_j\phi(\bx)-\tilde h_j(x_j)| = O(\eps)$ for all $\bx\in \scripte'''$ for each $j\in\{1,2,3\}$.

With this analyticity in hand,
Lemma~\ref{lemma:scriptf} gives $|\scripte'''|\lesssim \eps^\delta$
unless there exists a choice of $C^\omega$ functions $\tilde h_j,\kappa$ in the indicated families 
satisfying the exact equations
\begin{equation} \label{neartheend}
\nabla_j\phi(\bx) \equiv h_j(x_j)
\text{ for } j\in\{1,2,3\},
\end{equation}
with $x_3=\kappa(x_1,x_2)$, identically in $[0,1]^2$.
If such $h_j,\kappa$ do exist, then for each index $j$, 
define $H_j$ to be an antiderivative of $h_j$. 
Define $\tilde\phi:[0,1]^3\to\reals$ by
\[ \tilde\phi(\bx) = \phi(\bx) - H_1(x_1)-H_2(x_2)-H_3(x_3).\]
The equations \eqref{neartheend} imply that $\nabla\tilde\phi\equiv 0$
on the graph $x_3=\kappa(x_1,x_2)$. 
Thus $\phi$ is rank one degenerate, contradicting a hypothesis of Theorem~\ref{mainthm}.

Therefore $|\scripte'''|\lesssim\eps^\delta$
and consequently $|\scripte'| \lesssim \eps^{\delta/4}$,
completing the proof of Lemma~\ref{lemma:last2Dsublevel}.
Therefore Lemma~\ref{lemma:sublevel2} is proved, as well.
\end{proof}

\section{Completion of proofs of Theorems~\ref{thm:3}, \ref{thm:4}, and \ref{thm:nonosc}}
\label{section:completeproofs}

\begin{proof}[Conclusion of proof of Theorem~\ref{thm:3}]
This theorem has been reduced to the special case
in which $(\varphi_1,\varphi_2,\varphi_3)(x_1,x_2)
= (x_1,x_2,\varphi(x_1,x_2))$
and in which $(\varphi_1,\varphi_2,\varphi_3)$
is not equivalent to a linear system.
We write $S_\lambda^{(\varphi,\psi)}$.
That subcase has been proved under a supplementary bandlimitedness hypothesis 
on $f_3$.

Therefore by choosing $\tau$ to be a positive integral power of
$2$ and summing, it suffices to analyze
$S_\lambda^{(\varphi,\psi)}(f_1,f_2,g)$
with $\widehat{g}$ supported in $[\tau,2\tau]$ with $\tau \ge \lambda^{1+\rho/2}$.
Assume that $\frac{\partial^2\varphi}{\partial x_1\partial x_2}$ does not vanish identically.
We will prove that
\begin{equation} \label{thm:3lastbound}
|S_\lambda^{(\varphi,\psi)}(f_1,f_2,g)| \le C\tau^{-\delta}\norm{g}_2 \prod_{j=1}^2 \norm{f_j}_2
\end{equation}
under this hypothesis, completing
the proof of Theorem~\ref{thm:3}.

By Plancherel's theorem and an affine change of variables, we may express
\[ g(t) = \tau^{1/2}e^{i\tau} \int_{[0,1]} f_3(x_3) e^{it\tau x_3}\,dx_3\]
with $\norm{f_3}_2 = c\norm{g}_2$.
Thus
\[ S_\lambda^{(\varphi,\psi)}(f_1,f_2,g)
= ce^{i\tau} \tau^{1/2}
\int_{[0,1]^3} e^{i\lambda\psi(x_1,x_2)}e^{i\tau x_3\varphi(x_1,x_2)}
\prod_{j=1}^3 f_j(x_j)\,d\bx.\]
Thus
\[ |S_\lambda^{(\varphi,\psi)}(f_1,f_2,g)| 
= c\tau^{1/2} T_\tau^{\Psi_{\tau,\lambda}}(\bff)\]
with
\[ \Psi_{\tau,\lambda}(\bx) = x_3\varphi(x_1,x_2) + \tau^{-1}\lambda\psi(x_1,x_2).\]
The factor $\tau^{-1}\lambda$ is $\le \lambda^{-\rho/2}\ll 1$ for large $\lambda$.
Provided that $\lambda$ is large, $\Psi_{\tau,\lambda}$ is well approximated
by $x_3\varphi(x_1,x_2)$.

If $\psi$ is any $C^\omega$ function
and the partial derivatives $\frac{\partial\varphi}{\partial x_j}$
for $j=1,2$ and $\frac{\partial^2\varphi}{\partial x_1\partial x_2}$ vanish nowhere
on $[0,1]^2$, then $\Psi_{\tau,\lambda}$ satisfies all hypotheses
of Theorem~\ref{thm2}, uniformly for all sufficiently large $\lambda$
and all $\tau \ge \lambda^{1+\rho/2}$. 
The proof of Theorem~\ref{thm2} relied only on a special 
bandlimited case of Theorem~\ref{thm:3} that has already proved in full, 
so we may invoke Theorem~\ref{thm2} here without circularity in the reasoning.
We conclude  that for $|\lambda|$ sufficiently large, 
\[ |T_\tau^{\Psi_{\tau,\lambda}}(\bff)| \le C\tau^{-\gamma}\prod_j\norm{f_j}_2\]
with $C<\infty$ and $\gamma>\tfrac12$ independent of $\lambda,\tau$.
This establishes \eqref{thm:3lastbound}
with $\delta = \gamma-\tfrac12>0$,
completing the proof of Theorem~\ref{thm:3}
under the supplemental hypothesis that the mixed second derivative
$\frac{\partial^2\varphi}{\partial x_1\partial x_2}$ 
vanishes nowhere on $[0,1]^2$.

This nonvanishing hypothesis can be weakened; it suffices to assume that the 
partial derivative does not vanish identically on any open set. 
Ineed, we have already implicitly proved a more quantitative result,
namely an upper bound of the form
\[ C(1+|\lambda|)^{-\delta} 
\left(\min_{j\ne k}\left| \frac{\partial^2\varphi}{\partial x_j\partial x_k}\right|\right)^{-N}
\]
for some $N,C<\infty$
provided that $\varphi,\psi$ lie in some compact (with respect to the $C^3$
norm) family of $C^\omega$ functions. 

Let $\eps_0$ be a sufficiently small positive number, depending
only on $\varphi, \psi$.
Partition a neighborhood of the support of the cutoff function $\eta$
into squares of sidelengths $|\lambda|^{-\eps_0}$.
The union of those squares on which some mixed second partial
derivative of $\varphi$ has magnitude $<|\lambda|^{-\eps_0}$
has Lebesgue measure $O(|\lambda|^{-\eps})$ for some $\eps>0$
that depends only on $\varphi,\psi$ and the choice of $\eps_0$.
The number of remaining squares is $O(|\lambda|^{2\eps_0})$.
The contribution of each such square can be analyzed
by making an affine change of variables that converts it to $[0,1]^2$.
Invoking the more quantitative result
produces a bound of the form
$C|\lambda|^{\delta-C\eps_0}$ 
for each. If $\eps_0$ is sufficiently small, the result follows.
\end{proof}

\begin{proof}[Conclusion of proof of Theorem~\ref{thm:4}]
The roles of the indices $1,2,3$ in Theorem~\ref{thm:3} can be freely permuted
by making changes of coordinates $(x_1,x_2)\mapsto (x_1,\varphi(x_1,x_2))$
and $\mapsto (x_2,\varphi(x_1,x_2))$.
Therefore the roles of the three functions can be freely interchanged in
\eqref{thm:3lastbound}.
Theorem~\ref{thm:4} is an immediate consequence for $p=2$.
For $p\in(\tfrac32,2)$ it is obtained by interpolating between
this result for $p=2$
and the elementary result for $(p,s) = (\tfrac32,0)$.
\end{proof}

\begin{proof}[Proof of Theorem~\ref{thm:nonosc}]
It suffices to analyze the case in which two functions are in $\lt$
and one is in a negative order Sobolev space,
that is, to prove that
\begin{equation}
\int \eta\cdot \prod_{j=1}^3 (f_j\circ\varphi_j)
= O\big(\norm{f_1}_2\norm{f_2}_2\norm{f_3}_{H^s}\big).
\end{equation}
A simple interpolation then completes the proof.

By introducing a partition of unity and making local changes
of coordinates, we may reduce matters to the case in
which $\varphi(\bx)=x_i$ for $i=1,2$,
and $\varphi = \varphi_3$ has a mixed second partial derivative
$\frac{\partial^2\varphi}{\partial x_1\partial x_2}$
that vanishes nowhere on the support of $\eta$.

Express 
\[ f_3(\varphi(x_1,x_2)) = c_0 \int_\reals e^{i\tau \varphi(x_1,x_2)} \widehat{f_3}(\tau)\,d\tau.\]
It suffices to show that for large positive $\lambda$,
the contribution of the interval $\tau\in [\lambda,2\lambda]$ is $O(\lambda^{-\delta})$
for some $\delta>0$.

Substituting $\tau = \lambda x_3$, with $x_3\in[1,2]$, expresses this contribution
as a constant multiple of
\[ \lambda^{1/2} \int_{\reals^2\times[1,2]} 
e^{i\lambda \psi(\bx)} \prod_{j=1}^3 g_j(x_j) \eta(x_1,x_2)\,d\bx\]
with \[\psi(\bx) = x_3\varphi(x_1,x_2),\]
$g_i=f_i$ for $i=1,2$,
and $g_3(t) = \lambda^{-1/2}\widehat{f_3}(\lambda^{-1}t)$.
The function $g_3$ satisfies
\[\norm{g_3}_2 \le C\lambda^{-s} \norm{f_3}_{H^s}.\]

According to Lemma~\ref{lemma:curv-nondegen},
$\psi$ is not rank one degenerate on
the product of the support of $\eta$ with $[1,2]$.
Moreover, for any pair of distinct indices $j\ne k\in\{1,2,3\}$, 
$\frac{\partial^2\psi}{\partial x_j\partial x_k}$
vanishes nowhere on the domain of integration.
For 
$\frac{\partial^2\psi}{\partial x_j\partial x_3}$,
this is equivalent to nonvanishing of $\frac{\partial\varphi}{\partial x_j}$,
which is a hypothesis.
For
$\frac{\partial^2\psi}{\partial x_1\partial x_2}$,
it follows from the nonvanishing of
$\frac{\partial^2\varphi}{\partial x_1\partial x_2}$
and of $x_3$.
Thus $\psi$ satisfies all hypotheses of Theorem~\ref{mainthm}.
Therefore
\begin{align*} \big| \int_{\reals^2\times[1,2]} 
e^{i\lambda \psi(\bx)} \prod_{j=1}^3 g_j(x_j) \eta(x_1,x_2)\,d\bx\big|
&\le C\lambda^{-\gamma} \prod_{j=1}^3 \norm{g_j}_2
\\&
\le C\lambda^{-\gamma + \tfrac12-s} \norm{f_1}_2\norm{f_2}_2 \norm{f_3}_{H^s}
\end{align*}
for some $\gamma > \tfrac12$.
If $s<0$ is sufficiently close to $0$,
then $\gamma > \tfrac12-s$, and the proof is complete.
\end{proof}

\section{Yet another variant} \label{section:yetanother}

Let $U\subset\reals^2$ be a nonempty open set.
For $j\in\{1,2,3\}$, 
let $X_j$ be a $C^\omega$ nowhere vanishing vector field in $U$.
Suppose that for any distinct indices $j\ne k\in\{1,2,3\}$,
all integral curves of $X_j,X_k$ intersect transversely
at every point of $U$.

The weak convergence theorem of \cite{JMR}
is concerned with functions that satisfy $g_j\in \lt(U)$
and $X_j g_j\in\lt(U)$, whereas the results stated above in \S\ref{section:mainresults}
are concerned with the special case in which $X_jg_j\equiv 0$. 
Here we extend those results to this more general situation.

\begin{theorem} \label{thm:nonosc2}
Let $(X_j: j\in\{1,2,3\})$ be as above. 
Suppose that 
the curvature of the $3$-web associated to $(X_j: j\in\{1,2,3\})$
does not vanish at any point of $U$.
Then for any exponent $p>\tfrac32$
and any auxiliary function $\eta\in C^\infty_0(U)$,
there exist $C<\infty$ and $s<0$ such that
\begin{equation} 
\big| \int_{\reals^2} \eta\,\prod_{j=1}^3 g_j \big|
\le C\prod_j \big( \norm{g_j}_{W^{s,p}} + \norm{X_j g_j}_{W^{s,p}}\big)
\ \text{ for all $g_j\in C^1(\reals^2)$.}
\end{equation}
\end{theorem}

\begin{corollary} \label{cor:weakconverge2}
Let $(X_j: j\in\{1,2,3\})$ be as above. 
Let $p>\tfrac32$.
Let $g_j^{\nu},X_j g_j^{\nu} \in L^p(\reals^2)$ 
be uniformly bounded,
and suppose that $g_j^{\nu}\rightharpoonup g_j$ weakly
as $\nu\to\infty$ for $j=1,2,3$.
Then
\begin{equation}
\prod_{j=1}^3 g_j^{\nu}  
\rightharpoonup
\prod_{j=1}^3 g_j
\ \text{ weakly as $\nu\to\infty$}
\end{equation}
in every relatively compact open subset of $U$.
\end{corollary}

To deduce Theorem~\ref{thm:nonosc2} from the results proved above,
introduce $C^\omega$ diffeomorphisms $\phi_j = (\varphi_j^1,\varphi_j^2)$
from  $U$ to open subsets of $\reals^2$,
such that the curves $\{\bx: \varphi_j^1(\bx)=t\}$ are the integral
curves of $X_j$. Write $g_j = F_j\circ\phi_j$.
Then the $W^{s,p}$ norms of $g_j$ and of $X_j g_j$
together control the $W^{s,p}$ norms of $F_j$ and of $\frac{\partial F_j}{\partial y_2}$.
By simple decomposition and interpolation, 
it suffices to bound the integral under the assumption that
for each $j$,
\[ \norm{F_j}_{W^{s,p}} + \norm{\frac{\partial^M F_j}{\partial y_2^M}}_{W^{s,p}} \le 1\] 
for some $M<\infty$;
we may choose $M$ as large as may be desired.

Expand $F_j$ in Fourier series with respect to the second variable:
\[ F_j(y,t) = \sum_{n\in\integers} f_{j,n}(y)e^{int}.\]
Then
\[ \norm{f_{j,n}}_{W^{s,p}} = O(1+|n|)^{-N}\]
with $N$ as large as may be desired. 
Thus we are led to
\[ \sum_{\bn\in\integers^3} \int \eta(\bx) e_\bn(\bx) \prod_j (f_{j,n}\circ\varphi_j^1)\]
with
\[ e_\bn(\bx) = \prod_{k=1}^3 e^{i n_k \varphi_k^2(\bx)}.\]
Set $\eta_\bn = \eta e_{\bn}$.
These functions satisfy
\[ \norm{\eta_{\bn}}_{C^K} = O(1+|n|)^{K-N}\]
for any $K<\infty$. 

Thus it suffices to invoke a small improvement on Theorem~\ref{thm:nonosc}:
under the hypotheses of that theorem, there exists $K<\infty$ such that
\begin{equation}
\big| \int \eta \prod_j (f_j\circ\varphi_j)\big|
\le C\norm{\eta}_{C^K}\prod_j \norm{f_j}_{W^{s,p}},
\end{equation}
uniformly for all $C^K$ functions
$\eta$ supported in a fixed compact region
in which the hypotheses hold.
This can be deduced from the formally more restrictive result already proved,
by introducing a $C^\infty$ partition of unity $\{\zeta_\alpha^2\}$ to reduce to the case
in which $\varphi_j(\bx)\equiv x_j$ for $j=1,2$ for each $\alpha$,
then expanding $\zeta_\alpha\cdot\eta$ in Fourier series
and incorporating factors $e^{in_jx_j}$ into $f_j$. 
\qed

\section{Remarks on the nondegeneracy hypotheses} \label{section:hypotheses}

\noindent (1)\ 
Phases $\phi$ that satisfy the hypotheses of Theorem~\ref{mainthm}
exist in profusion.
Given a point $\bar x\in[0,1]^3$,
for generic tuples $(a_{j,k},b_{i,j,k})$
of real numbers satisfying the natural symmetry conditions,
any phase satisfying
\begin{equation} 
 \frac{\partial^2\phi}{\partial x_j \partial x_k}(\bar x)=a_{j,k}
\ \text{ and } \ 
 \frac{\partial^3\phi}{\partial x_i \partial x_j \partial x_k}(\bar x)=b_{i,j,k}
\end{equation}
is rank one nondegenerate in some neighborhood of $\bar x$.
In other words, we claim that if $\phi$ is rank one degenerate in every neighborhood of $\bar\bx$,
then its second and third order partial derivatives at $\bar x$ must satisfy
certain algebraic relations \eqref{ugly}.

Restrict attention to phases whose mixed second order partial derivatives
$\frac{\partial^2 \phi}{\partial x_j\partial x_k}$, $j\ne k$, are all nonzero at $\bar\bx$.
Suppose that $H$ is a small $C^\omega$ hypersurface containing $\bar\bx$,
on which $\nabla\tilde\phi$ vanishes identically,
with $\tilde\phi(\bx) = \phi(\bx)-\sum_j h_j(x_j)$.
Suppose that $H$ is represented
by an equation $x_3 = \kappa(x_1,x_2)$ in a neighborhood of $(\bar x_1,\bar x_2)$,
with $\kappa$ smooth.
Thus $\kappa(\bar x_1,\bar x_2) = \bar x_3$.

Write $\phi_j$ for $\frac{\partial\phi}{\partial x_j}$,
$\phi_{j,k}$ for the corresponding second partial derivatives,
and $\phi_{i,j,k}$ for third order derivatives.
Denote partial derivatives of $\kappa$ by $\kappa_j$, for $j=1,2$.

The vanishing of $\nabla_j\tilde\phi$ at $(x_1,x_2,\kappa(x_1,x_2))$ for $j=1,2$ implies that
$\phi_1(x_1,x_2,\kappa(x_1,x_2))$ is independent of $x_2$ in a neighborhood of $(\bar x_1,\bar x_2)$.
Therefore
\[\phi_{1,2}(x_1,x_2,\kappa(x_1,x_2)) + \kappa_2(x_1,x_2) \phi_{1,3}(x_1,x_2,\kappa(x_1,x_2))=0\]
in a neighborhood of $(\bar x_1,\bar x_2)$.
We write this relation as 
$\phi_{1,2} + \kappa_2\phi_{1,3}=0$,
leaving it understood that $\phi$ and its partial derivatives
are evaluated at $(x_1,x_2,\kappa(x_1,x_2))$ while $\kappa$ is evaluated
at $(x_1,x_2)$, and that $(x_1,x_2)$ varies within a small neighborhood of $(\bar x_1,\bar x_2)$.
Likewise, $\phi_{2,1} + \kappa_1\phi_{2,3}=0$.
Thus
\begin{equation}\label{kappas} \kappa_2 = - \phi_{1,2}\,\phi_{1,3}^{-1}
\ \text{ and } \ 
\kappa_1 = - \phi_{2,1}\,\phi_{2,3}^{-1}.\end{equation}
Differentiating the first of these relations with respect to $x_1$
and the second with respect to $x_2$, and invoking the relation
$\kappa_{2,1} = \kappa_{1,2}$, we find that 
\begin{equation} \label{ugly}
(\phi_{2,3}^2) \big( \phi_{1,2,1} \phi_{1,3} - \phi_{1,2}\phi_{1,3,1}\big)
\equiv
(\phi_{1,3}^2) \big( \phi_{1,2,2} \phi_{2,3} - \phi_{1,2}\phi_{2,3,2}\big)
\end{equation}
at $(x_1,x_2,\kappa(x_1,x_2))$,
for all $(x_1,x_2)$ in a neighborhood of $(\bar x_1,\bar x_2)$.
In particular, \eqref{ugly} holds at $\bar x$.

\eqref{ugly} was derived under the assumption that the third coordinate
vector does not belong to the tangent space to $H$
at $\bar x$. Thus without that assumption, we conclude that if $\phi$ is rank one degenerate
in every neighborhood of $\bar x$, then at least one of three variants
of \eqref{ugly}, obtained from \eqref{ugly} by permuting the three coordinate variables,
must hold for the partial derivatives of $\phi$ at $\bar x$.
Rank one nondegeneracy therefore holds in all sufficiently small
neighborhoods of $\bar x$, for generic values of second and third 
partial derivatives of $\phi$ at $\bar x$.

\medskip
\noindent (2)\ 
The hypotheses of Theorem~\ref{mainthm}, taken as a whole
rather than individually, are stable with respect to small perturbations of $\phi$.
Indeed, the hypothesis that all three mixed second partial derivatives
are nowhere vanishing is manifestly stable.
A phase $\phi$ satisfying this auxiliary hypothesis
is rank one degenerate if and only if there exist $C^\omega$ functions $h_j(x_j)$ such that
$\nabla_j\phi(\bx)=h_j(x_j)$ for every $\bx\in H$,
for some piece of $C^\omega$ hypersurface $H\subset (0,1)^3$.

An exhaustive class of candidate hypersurfaces $H$
can be constructed, in terms of $\phi$, as follows.
Fix a base point $\bar x$ and consider hypersurfaces $H\owns \bar x$
such that the third coordinate vector does not lie in the
tangent space to $H$ at $\bar x$.
Express $H$ locally as a graph $x_3 = \kappa(x_1,x_2)$.
Determine $\kappa(x_1,\bar x_2)$ by solving the differential equation
\[\frac{\partial\kappa}{\partial x_1}(x_1,\bar x_2) = - \phi_{2,1}\phi_{2,3}^{-1}(x_1,\bar x_2)\]
derived above, with initial condition $\kappa(\bar x_1,\bar x_2)=\bar x_3$.
Recall that the mixed second partial derivative $\phi_{2,3}$ vanishes nowhere,
by hypothesis. 

For each $x_1$ in a small neighborhood of $\bar x_1$,
determine $\kappa(x_1,x_2)$
by solving
\[\frac{\partial\kappa}{\partial x_2}(x_1,x_2) = - \phi_{1,2}\phi_{1,3}^{-1}(x_1,x_2)\]
with the initial condition $\kappa(x_1,\bar x_2)$ determined in the preceding step.
This defines a $C^\omega$ hypersurface $H$ containing $\bar x$,
and this is locally the only such hypersurface passing through $\bar x$ whose tangent space
does not contain the third coordinate vector and that could potentially satisfy
the condition in the definition of rank one degeneracy of $\phi$.
Repeating this construction twice more with suitable permutations
of the coordinate indices yields three (or fewer) candidate hypersurfaces
for each point $\bar x$. Plainly this construction is continuous with respect
to $\phi,\bar x$.

Once a hypersurface $H$ is specified, the vanishing of the gradient
of $\tilde\phi(\bx) = \phi(\bx)-\sum_{j=1}^3 h_j(x_j)$
at each point of $H$ determines the derivative $h'_j$ at each point
of $\reals$ sufficiently close to $x_j$. Thus the functions $h_j$ are
completely determined in a neighborhood of $\bar x$, up to additive constants.
Again, these depend continuously on $\phi,\bar x$.

If $\phi$ satisfies the hypotheses of Theorem~\ref{mainthm},
if $\bar x\in [0,1]^3$, and if $H$ and associated functions $h_j$
are as above, then $\nabla\tilde\phi$ fails to vanish identically
on $H$, so by real analyticity, some partial derivative along $H$
of $\nabla\phi$ fails to vanish at $\bar x$. This nonvanishing
is stable under small perturbations of $\phi,\bar x$.

\medskip
\noindent (3)\ 
The examples \ref{example:instability}
also demonstrate that the optimal exponent $1+\delta$
in Lemma~\ref{lemma:sublevel2} is not stable with respect to
perturbations of $\phi$, if the auxiliary hypothesis on the nonvanishing
of all three mixed partial derivatives is relaxed.

\section{More on integrals with oscillatory factors} \label{section:moreon}

Li, Tao, Thiele, and the present author \cite{CLTT}
investigated multilinear functionals
\[ S_\lambda(\bff) = \int_{\reals^D} e^{i\lambda\psi} \eta \prod_{j\in J} (f_j\circ\varphi_j)\]
with $\varphi_j:\reals^D\to\reals^{d_j}$ linear, 
and established
bounds of the type $O(|\lambda|^{-\gamma}\prod_j\norm{f_j}_\infty)$,
for certain tuples $(\psi,(\varphi_j: j\in J))$,
under two different sets of hypotheses. Both sets of hypotheses
were rather restrictive.
In one set, it was required that $d_j=D-1$
for every $j\in J$. In the other, $d_j=1$ for every $j$, and $|J|<2D$.
The latter result was invoked in the discussion above.

The method developed above yields an alternative proof of these results,
and thus our discussion can be modified to be self-contained,
with no invocation of results from \cite{CLTT}. 
More significantly, the method 
makes it possible to remove the assumption that $d_j=1$, 
as we now show.

Let $\scriptd>d\in\naturals$.
Let $\{\varphi_j: j\in J\}$ be a family of surjective linear mappings
from $\reals^\scriptd$ to $\reals^d$.
Such a family is said to be in
general position if for any subset $\tilde J\subset J$
satisfying $0<|\tilde J|\le  \scriptd/d$, 
the linear mapping 
\begin{equation} \label{Dtodinvertible} 
\reals^\scriptd\owns \bx \mapsto (\varphi_j(\bx): j\in\tilde J) \in (\reals^d)^{\tilde J}
\end{equation}
is injective. 

\begin{theorem} \label{thm:cltt_altproof}
Let $d,\scriptd\in\naturals$ with $\scriptd/d\in\naturals$.
Let $\eta\in C^\infty_0(\reals^{\scriptd})$.
Let $\psi$ be a real-valued $C^\omega$ function defined in 
a neighborhood $U$ of the support of $\eta$.
Let $J$ be a finite index set of cardinality $|J|$
satisfying $1 \le |J|<2\scriptd/d$.

Let $\{\varphi_j: j\in J\}$ be a family of surjective linear mappings
$\varphi_j:\reals^\scriptd\to\reals^d$ in general position.
Suppose that $\psi$ cannot be expressed in the form
$\psi = \sum_{j\in J} h_j\circ\varphi_j$
in any nonempty open set, with $h_j\in C^\omega$.

Then there exist $\delta>0$ and $C<\infty$ such that for all
$\lambda\in\reals$ and all functions $f_j\in L^\infty(\reals^d)$,
the form
\begin{equation}
S_\lambda(\bff) = \int_{\reals^\scriptd} e^{i\lambda\psi} 
\prod_{j\in J} (f_j\circ\varphi_j) \eta
\end{equation}
satisfies
\begin{equation} \label{extraconclusion}
|S_\lambda(\bff)| \le C|\lambda|^{-\delta} \prod_{j\in J}\norm{f_j}_{L^\infty}.
\end{equation}
\end{theorem}

This extends Theorem~{2.1} of
\cite{CLTT}, in which it is assumed that $d=1$, and that $\psi$ is a polynomial.
The polynomial hypothesis is not essential to the proof given 
in \cite{CLTT}, but the restriction $d=1$ is.

The simplest instance of Theorem~\ref{thm:cltt_altproof} with $d>1$ is as follows.
Let $B\subset\reals^d$ be a ball centered at $0$.
Let $Q:\reals^d\times\reals^d\to\reals$ 
be a homogeneous quadratic real-valued polynomial.
To $Q$, associate its antisymmetric part
$Q^*(x,y) = \tfrac12(Q(x,y)-Q(y,x))$.
Denote by $\norm{\cdot}$ any norm on the vector
space of all antisymmetric quadratic real-valued polymomials.

\begin{corollary} \label{cor:last}
Let $d\ge 2$.
There exist $C<\infty$ and $\gamma>0$ such that for
all functions $f_j\in L^2$,
\begin{equation}
\big|
\iint_{B\times B} e^{iQ(x,y)} f_1(x)f_2(y)f_3(x+y)\,dx\,dy
\big|
\le C\norm{Q^*}^{-\gamma} \prod_j \norm{f_j}_{L^2}.
\end{equation}
\end{corollary}

\begin{example}
Let $d=2$ and
$Q((x_1,x_2),\,(y_1,y_2)) = x_1y_2$.
Then
\begin{equation}
\big|
\iint_{[0,1]^2\times[0,1]^2 } e^{i\lambda x_1y_2} f_1(x)f_2(y)f_3(x+y)\,dx\,dy
\big|
\le C|\lambda|^{-\gamma} \prod_j \norm{f_j}_{L^2}.
\end{equation}
\end{example}

\begin{proof}[Proof of Theorem~\ref{thm:cltt_altproof}]
The proof of Theorem~\ref{thm:cltt_altproof} has the same overarching structure as the
analysis developed above for Theorem~\ref{thm:3}. 
However, one step of the proof of Theorem~\ref{thm:3}
broke down when the mappings $\varphi_j$ were linear, and
Theorem~{2.1} of \cite{CLTT} was invoked, in a black box spirit, to treat the linear case. 
Much of the proof of Theorem~\ref{thm:cltt_altproof} closely follows
arguments above and hence will be merely sketched, 
but we will show in more detail how the problematic step, 
which arises near the end of the analysis,
can be modified to handle linear mappings.

Let $\rho>0$ be a small exponent, which will ultimately
depend on another exponent $\sigma$ introduced below, which in turn will
depend on an exponent $\tau$ in a sublevel set bound
\eqref{anothersublevel}. Assume without loss of
generality that $\lambda\ge 1$ and that
$\norm{f_j}_\infty \le 1$.
Decompose
\begin{equation}
f_j(y) = \sum_{m} \eta_{m}(y) \sum_{k\in \integers^d} 
a_{j,m,k} e^{i\pi \lambda^{1/2} k\cdot y} 
\end{equation}
with each $\eta_{m}$ supported on the double of a cube of sidelength 
$\lambda^{-1/2}$
and $|\eta_m| + |\lambda^{-1/2}\nabla \eta_m| = O(1)$,
and \[ \sum_k |a_{j,m,k}|^2 = O(1)\] 
uniformly in $j,m,\lambda$.
Decompose $f_j = g_j+h_j+F_j$ where $\norm{F_j}_\infty = O(\lambda^{-N})$
for all $N<\infty$, 
$h_j$ is the sum over $m,k$ of those terms satisfying
$|a_{j,m,k}| \le \lambda^{-\sigma}$,
and $g_j$ has an expansion of the same type 
with $a_{j,m,k}=0$ for all but at most $O(\lambda^{2\sigma})$ 
indices $k$ for each $j,m$.  The contributions of all $F_j$ are negligible,
and we may therefore henceforth replace $f_j$ by $g_j+h_j$ for each index $j$.

Write $\bk = (k_j: j\in J)\in (\integers^d)^J$.
Define the linear mapping $L: (\integers^d)^J\to \reals^\scriptd$ 
to be the transpose of $\bx\mapsto (\varphi_j(\bx): j\in J)$; thus
\begin{equation} L(\bk) = \sum_{j\in J} \varphi_j^*(k_j)\end{equation}
where $\varphi_j^*$ denotes the transpose of the linear mapping $\varphi_j$.
Writing $\bm = (m_j: j\in J)$ and $\bx \in\reals^\scriptd$,
our functional can be expanded as
\[ S_\lambda(\bff)
= \sum_\bm \sum_\bk \prod_{j\in J} a_{j,m_j,k_j}
I(\bm,\bk)
\]
with
\begin{align*}
I(\bm,\bk) &= 
\int 
e^{i\lambda\Phi_\bk(\bx)}
\zeta_\bm(\bx)\,dx
\\
\Phi_\bk(\bx) 
&= \psi(\bx) +  \pi\lambda^{-1/2} L(\bk)\cdot\bx
\\
\zeta_\bm(\bx) &= \prod_{j\in J} \eta_{m_j}(\varphi_j(\bx)).
\end{align*}

While the number of indices $\bm$ in play is comparable to $(\lambda^{d/2})^{|J|}$,
there are only $O(\lambda^{\scriptd/2})$ indices $\bm$ for which $\zeta_\bm$
does not vanish identically.
We claim that there exists $\theta\in(0,1)$,
which depends only on the ratio $\scriptd/(d|J|)$, such that for any $\bm$
and any sequences of scalars $b_j(\cdot)$,
\begin{equation} \label{beatl2}
\sum_\bk \prod_{j\in J} |b_{j}(k_j)|
\cdot |I(\bm,\bk)|
\le C \lambda^{C\rho} \lambda^{-\scriptd/2} \prod_{j\in J} 
\norm{b_j}_{l^2}^{1-\theta}\norm{b_j}_{l^\infty}^\theta
\end{equation}
uniformly in $\bm,\lambda$.
Indeed, 
\begin{equation} \label{Ifirstbound} 
|I(\bm,\bk)|= O(\lambda^{-\scriptd/2}) \end{equation}
uniformly in $\bm,\lambda$.
For each $\bm$ for which $\zeta_\bm$ does not vanish identically, choose 
$z_\bm$ in the support of $\zeta_\bm$.
Integrating by parts sufficiently many times gives
\begin{equation} \label{Isecondbound} 
|I(\bm,\bk)| \le  C_N\lambda^{-N}(1+|\nabla\psi(z_\bm) + L(\bk)|)^{-N}
\ \text{ for every $N<\infty$} \end{equation}
unless
\begin{equation}
|\nabla\Phi_\bk(z_\bm)| \le \lambda^\rho.
\end{equation}

Recalling that $|J| \ge \scriptd/d$,
consider any subset $S\subset J$ of cardinality equal to $\scriptd/d$.
If $\sum_{j\in S} \varphi_j^*(k_j)=0$
then $k_j=0$ for every $j\in S$,
since the mapping $\bx \mapsto (\varphi_j(\bx): j\in S)$
is bijective by the general position hypothesis.
Therefore if $N$ is chosen to be sufficiently large
then the summation over all vectors $(k_j: j\in S)$ of
$\min\big(\lambda^{-\scriptd/2}, \lambda^{-N}(1+|\nabla\psi(z_\bm) + L(\bk)|)^{-N}\big)$
is $O(\lambda^{-\scriptd/2}\lambda^{C\rho})$, uniformly
for all vectors $(k_j: j\in J\setminus S)$.
It follows that
\begin{equation*} 
\sum_\bk \prod_{j\in J} |b_{j}(k_j)|
\cdot |I(\bm,\bk)|
\le C \lambda^{C\rho} \lambda^{-\scriptd/2} \prod_{j\in J\setminus S} 
\norm{b_j}_{l^1}
\prod_{j\in S} \norm{b_j}_{l^\infty}
\end{equation*}
by
\eqref{Ifirstbound}, \eqref{Isecondbound},
and the general position assumption \eqref{Dtodinvertible}. 
Since $|J|<2d^{-1} \scriptd$, it follows by interpolation that
\begin{equation*} 
\sum_\bk \prod_{j\in J} |b_{j}(k_j)|
\cdot |I(\bm,\bk)|
\le C \lambda^{C\rho} \lambda^{-\scriptd/2} \prod_{j\in J} 
\norm{b_j}_{l^q}
\end{equation*}
for some exponent $q>2$. 
Then $\norm{b_j}_{l^q} 
\le \norm{b_j}_{l^2}^{1-\theta} \norm{b_j}_{l^\infty}^\theta$,
for some $\theta=\theta(q)>0$, yielding \eqref{beatl2}.

From \eqref{Ifirstbound} and \eqref{beatl2},
for $f_j=g_j+h_j$ with the properties indicated above, there follows 
\begin{equation}
|S_\lambda(\bff)|
\le |S_\lambda(\bg)| + O(\lambda^{-\sigma+C\rho}).
\end{equation}
Therefore, choosing $\rho$ to be sufficiently
small relative to $\sigma$, it suffices to analyze $S_\lambda(\bg)$.

The quantity $S_\lambda(\bg)$
can in turn be expressed as a sum of $O(\lambda^{C\sigma})$
terms, in each of which each function $g_j$ takes the simple form
\begin{equation}
g_j(x) = \sum_{m_j} a_{j,m_j} e^{i\pi\lambda^{1/2} k_{j,m_j}}
\zeta_{m_j}(x)
\end{equation}
with $|a_{j,m_j}|=O(1)$. 
We assume this form henceforth,
at the expense of a factor $O(\lambda^{C\sigma})$.
This factor can be absorbed at the end of the proof,
by choosing $\sigma$ sufficiently small
relative to the exponent $\tau$ that appears below,
just as was done in other proofs earlier in the paper.  Thus
\begin{equation}
|S_\lambda(\bg)| \le C \sum_\bm |I(\bm,\bk_\bm)|
\end{equation}
with $\bk_\bm = (k_{j,m_j}: j\in J)$.

Define 
\begin{equation}
\Phi(\bx) = \psi(\bx) + \pi \lambda^{-1/2} L(\bk) \cdot \bx.
\end{equation}
Consider those $\bm$ that
are stationary in the sense that $|\nabla\Phi(z_\bm)|
\ge\lambda^\rho$. 
By \eqref{Isecondbound}, the sum of the contributions 
of all such $\bm$ is $O(\lambda^{-N})$ for every $N<\infty$.
Therefore in order to complete the analysis,
it suffices to show that the 
number of $\bm$ for which $\Phi$ is nonstationary,
is $O(\lambda^{-\tau}\lambda^{\scriptd/2})$ for some exponent $\tau>0$.

Let $B\subset\reals^\scriptd$ be any ball of finite radius.
Let $h_j:\reals^d\to\reals^d$ be arbitrary Lebesgue measurable functions.
Define
\begin{equation}
\scripte = \{\bx\in B: 
|\nabla \psi(\bx)-\sum_{j\in J} (h_j\circ\varphi_j)(\bx) \cdot D\varphi_j| <\eps\}.
\end{equation}
Here, $h_j$ takes values in $\reals^d$,
and $(h_j\circ\varphi_j)\cdot D\varphi_j$ takes values in $\reals^\scriptd$.
To complete the proof of Theorem~\ref{thm:cltt_altproof},
it now suffices to show that there exist $C<\infty$ and $\tau>0$ such that
\begin{equation} \label{anothersublevel}
|\scripte|
\le C \eps^\tau
\end{equation}
uniformly for all $\eps\in(0,1]$ and all functions $h_j$.

Assume temporarily that $|J| \ge \scriptd/d$.
Let $\tilde J\subset J$ be any subset of cardinality $|\tilde J|=\scriptd/d$.
The general position hypothesis ensures that
there exists a linear subspace $V\subset\reals^\scriptd$ of dimension $d$
such that $\kernel(\varphi_j)\subset V$ for each $j\in J\setminus\tilde J$,
but for each $j\in\tilde J$,
$\varphi_j|_V$ is an invertible linear mapping from $V$ to $\reals^d$.

If the system of equations
$\nabla \psi(\bx)-\sum_{j\in J} (h_j\circ\varphi_j)(\bx) \,D\varphi_j| <\eps\} = 0$
is restricted to any translate $V+\by$ of $V$,
those terms $h_j\circ\varphi_j$ with $j\in J\setminus\tilde J$
become constant functions of $\bx\in V$.
For any $\bx\in V$, what results is an invertible linear system of $d$
equations for $d$ unknowns $h_j(\varphi_j(\bx)\cdot D\varphi_j$, 
with the index $j$ running over $\tilde J$.

By the same reasoning as developed in the analyses of upper bounds for measures
of sublevel sets above,
we may conclude that there exist functions of the form $H_j+r_j$,
where $H_j:\reals^d\to\reals^d$ are drawn from a compact family of $C^\omega$
functions specified in terms of $\psi,\{\varphi_j: j\in J\}$ alone
and $r_j\in\reals^d$ are constant vectors,
and a set $\tilde\scripte$ satisfying $|\tilde\scripte| \ge c|\scripte|^C$,
such that  
\begin{equation}  \label{darnconstants}
|\nabla \psi(\bx)-\sum_{j\in J} \big[(H_j\circ\varphi_j)(\bx)+r_j\big] \cdot D\varphi_j| <C\eps
\ \text{ for all $\bx\in \tilde\scripte$.}
\end{equation}

We have reached the point at which the proof of
Theorem~\ref{thm:3} must be augmented in order to 
treat Theorem~\ref{thm:cltt_altproof}.
Let $C_0$ be some finite constant.
If 
\begin{equation} \label{wishbound} |r_j|\le C_0 \ \text{ for all $j\in J$} \end{equation}
then the functions $\tilde H_j = H_j+r_j$ are drawn from a compact family of $C^\omega$
functions, and the same reasoning as in the proof of
Theorem~\ref{thm:3} can be applied to conclude that
$|\tilde\scripte|\le C\eps^\tau$,
and hence that the same holds for $|\scripte|$ with modified constants $C,\tau$.

However, it is not true that there exists $C_0$ such that \eqref{wishbound} holds.
Indeed, if $G_j:\reals^d\to\reals$ are linear functions satisfying
$\sum_{j\in J} G_j\circ \varphi_j\equiv 0$,
then for any $t\in(0,\infty)$,
replacement of $h_j$ by $h_j + t\nabla G_j$ does 
not change the quantity $\nabla\psi - \sum_j (h_j\circ\varphi_j)\cdot D\varphi_j$,
and consequently does not change $\scripte$.
If $|J|>\scriptd/d$ then there exists a linear solution of
$\sum_{j\in J} G_j\circ \varphi_j\equiv 0$,
with at least one $G_j$ not identically zero. 
By taking $t$ arbitrarily large, 
one finds that no uniform {\em a priori} bound \eqref{wishbound}
is available for the functions $h_j$ in terms of $\psi$, $(\varphi_j: j\in J)$,
and $\eps$ alone.

If $\eps\le 1$, as we may assume, and if $\scripte$ is nonempty,
then while the individual quantities $r_j$ may be large,
\begin{equation} |\sum_{j\in J} r_j\cdot D\varphi_j|=O(1). \end{equation}
This follows from the condition
\[ |\nabla\psi(\bx) - \sum_j (H_j\circ\varphi_j)\cdot D\varphi_j 
- \big(\sum_j r_j \cdot D\varphi_j\big) \big| < \eps\le 1\]
by the triangle inequality,
since $\nabla\psi$ and $H_j$ are uniformly bounded. 
There exist $\tilde r_j\in\reals^\scriptd$ satisfying
\[ \sum_{j\in J} \tilde r_j\cdot D\varphi_j =\sum_{j\in J} r_j\cdot D\varphi_j\]
and $|\tilde r_j| = O(1)$ for every $j\in J$.
Define $\tilde h_j = H_j + \tilde r_j$.
These modified functions define the same sublevel set as do the original $h_j$, since
\[\sum_{j\in J} (\tilde h_j\circ\varphi_j)\cdot D\varphi_j
\equiv \sum_{j\in J} (h_j\circ\varphi_j)\cdot D\varphi_j
\ \text{ on $\tilde\scripte$.} \]
Thus we may replace $h_j$ by $\tilde h_j$ for all indices $j$.
The functions $\tilde h_j$ are now drawn from a compact family
of $C^\omega$ functions determined by $\psi,\{\varphi_j: j\in J\}$.
The same reasoning as in the above analyses of sublevel sets
completes the proof of Theorem~\ref{thm:cltt_altproof}.
\end{proof}

The less singular case in which $|J|<\scriptd/d$ 
can be treated by a simplified form of this reasoning.
Details are omitted.

\medskip
The intermediate conclusion that $h_j = H_j+r_j$
on a large set, with $H_j$ uniformly bounded and $r_j$ constant
though not necessarily uniformly bounded,
breaks down without the restriction $|J|< 2\scriptd/d$.
For an example, consider $d=1$, $\scriptd=2$, and $|J|=4$ with mappings 
$\varphi_1(\bx) = \varphi_1(x_1,x_2) = x_1$, 
$\varphi_2(\bx) = x_2$, 
$\varphi_3(\bx) = x_1+x_2$, 
$\varphi_4(\bx) = x_1-x_2$. 
Set $G_j(x) = 2x^2$ for $j=1,2$,
and $=-x^2$ for $j=3,4$.
Then $\sum_{j=1}^4 G_j\circ\varphi_j\equiv 0$.
Therefore $g_j = G'_j$ satisfy
$\sum_{j=1}^4 (g_j\circ\varphi_j)\cdot D\varphi_j \equiv 0$.
Therefore any tuple of functions $h_j$
could be replaced by $h_j + tg_j$ for any parameter $t\in\reals$,
without changing the associated sublevel set $\scripte$.
Thus no upper bound at all holds for $(h_j: j\in J)$ modulo 
constants $r_j$, as in the above argument.

Conversely, in the context of the preceding paragraph, 
if $\sum_j G_j\circ\varphi_j\equiv 0$ then each $G_j$ must be
a polynomial of degree at most $2$. Thus each $g_j$ is a polynomial 
of degree at most $1$, though not necessarily constant.
This suggests that when $|J|\ge 2\scriptd/d$,
the reasoning should be modified by applying difference operators to
$\nabla\psi -\sum_j (h_j\circ\varphi_j)\cdot D\varphi_j$, 
and that difference operators of higher degrees 
should be required as $\scriptd/d$ increases.

\section{A scalar-valued sublevel set inequality}
\label{section:FEqns}

Let $B\subset\reals^2$ be a ball of positive
radius, and let $\varphi_j:B\to\reals^1$ be real analytic
for $j\in\{1,2,3\}$. Suppose that $\nabla\varphi_j$
are pairwise linearly independent at each point in $B$.
Let $0\le\eta\in C^\infty(B)$.

The functional equation
$f(x)+g(y)+h(x+y)=0$, has been widely studied.
Its solutions are the ordered triples $(f(x),g(y),h(x+y)) = 
(ax+c_1,ay+c_2,a(x+y)-c_1-c_2)$ with $a,c_1,c_2$
all constant, and no others.
Approximate solutions, in a certain sense, have been studied 
in \cite{christyoung}.
We consider here the more general functional equation
\begin{equation} \label{fneqn}
\sum_{j=1}^3 (f_j\circ\varphi_j) = 0 \text{ almost everywhere}
\end{equation}
where the mappings $\varphi_j$ need not be linear,
and the functions $f_j$ are real-valued.
We discuss related sublevel sets
\begin{equation} \label{fneqn:ineq}
S(\bff,r) =\{\bx\in B: |\sum_{j=1}^3 (f_j\circ\varphi)(\bx) | \le r\}
\end{equation}
associated to ordered triples $\bff$
of scalar-valued functions.
The inequality \eqref{fneqn:ineq} differs from corresponding
inequalities studied and exploited in various
proofs above in two ways: it is homogeneous rather than inhomogeneous,
and it is a single scalar inequality, rather than a system of
two scalar inequalities.

Theorem~\ref{thm:nonosc} has the following implication concerning
the nonexistence of nontrivial solutions of \eqref{fneqn}.

\begin{corollary} \label{cor:fneqn1}
Let $B\subset\reals^2$ be a closed ball of positive, finite radius.
For $j\in\{1,2,3\}$ let $\varphi_j\in C^\omega$ map a neighborhood of $B$ 
to $\reals$, and suppose that $\nabla \varphi_j$ are pairwise linearly
independent at each point of $B$.
Suppose that the curvature of the web defined by $(\varphi_j: j\in\{1,2,3\})$
does not vanish identically on $B$.
Let $\bff$ be an ordered triple of Lebesgue measurable real-valued functions.
Suppose that for each index $j$
and each $t\in\reals$,
\begin{equation} \label{nullset}
|\{x: f_j(x)=t\}|=0.
\end{equation}
If $\bff$ is a solution of the functional equation \eqref{fneqn}
then each function $f_j$ is constant.
\end{corollary}

In particular, all $C^\omega$ solutions $\bff$ 
of \eqref{fneqn} are constants.
Indeed, one of the three component functions $f_j$
must fail to satisfy the hypothesis \eqref{nullset},
and hence must be constant.
It follows immediately from the functional equation \eqref{fneqn} that 
the other two component functions are also constant. 
\qed

A more quantitative statement is as follows.

\begin{corollary} \label{cor:fneqn2}
Let $B\subset\reals^2$ be a closed ball of positive, finite radius.
For $j\in\{1,2,3\}$ let $\varphi_j\in C^\omega$ map a neighborhood of $B$ 
to $\reals$, and suppose that $\nabla \varphi_j$ are pairwise linearly
independent at each point of $B$.
Suppose that the curvature of the web defined by $(\varphi_j: j\in\{1,2,3\})$
does not vanish identically on $B$.
There exist $\delta>0$ and $C<\infty$ such that 
for any ordered triple $\bff$ of Lebesgue measurable real-valued functions
and any $r\in(0,\infty)$,
the sublevel set $S(\bff,r)$ satisfies
\begin{equation} |S(\bff,r)| \le C
\sup_{t\in\reals} \big|\{x\in \varphi_j(B): |f_j(x)-t|\le r\}\big|^\delta
\end{equation}
for each $j\in\{1,2,3\}$.
\end{corollary}

In \S\ref{section:variablecoeffs} we discuss a related 
inequality for sublevel sets associated to expressions
$\sum_{j=1}^3 a_j(x) (f_j\circ\varphi_j)(x)$
with nonconstant coefficients $a_j$,
in the special case in which the mappings $\varphi_j$ are all linear.

Returning to the two corollaries formulated above,
we will first prove Corollary~\ref{cor:fneqn2},
then will indicate how a modification of the proof gives
Corollary~\ref{cor:fneqn1}. The following lemma will be used.

\begin{lemma} \label{lemma:nonlinearsobolev}
Let $\sigma<0$. 
Let $I\subset\reals$ be a bounded interval.
Then there exists $C<\infty$ such that for any real-valued function $f\in\lt(\reals)$
supported in a fixed bounded set, for any $A\in(0,\infty)$,
\begin{equation}
\int_{\lambda\le A}  \norm{\one_I e^{i\lambda f}}_{H^\sigma}^2\,d\lambda
\le CA \sup_{t\in\reals} \big|\{x\in I: |f(x)-t| \le A^{-1}\}|^{|\sigma|}.
\end{equation}
\end{lemma}

\begin{proof}
It suffices to treat the case $A=1$, since
the substitution $\lambda = A\tau$ reduces the general case to this one.

Let $h$ be a nonnegative Schwartz function 
satisfying $h(y)\ge 1$ for all $y\in[-1,1]$,
with $\widehat{h}$ supported in $[-1,1]$.
\begin{align*} 
\int_{\lambda\le 1}  \norm{\one_I e^{i\lambda f}}_{H^\sigma}^2\,d\lambda
&\le
\int h(\lambda) \norm{\one_I e^{i\lambda f}}_{H^\sigma}^2\,d\lambda
\\&= \int h(\lambda) \int_\reals
\big|\int e^{i\lambda f(x)}e^{-ix\xi}\,\one_I(x)\,dx\big|^2
(1+\xi^2)^\sigma\,d\xi \,d\lambda
\\ &= \int_\reals h(\lambda) \int_\reals
\iint_{I\times I} e^{i\lambda [f(x)-f(y)]}e^{-i(x-y)\xi}
\,dx\,dy
(1+\xi^2)^\sigma\,d\xi \,d\lambda
\\&=
\iint_{I\times I}
\big(\int_\reals
e^{-i(x-y)\xi}
(1+\xi^2)^\sigma\,d\xi \big)
A \widehat{h}(A(f(y)-f(x)))
\,dx\,dy
\\& \le C A
\iint_{I\times I}
|x-y|^{-1-\sigma}\,
\big|\widehat{h}(A(f(y)-f(x)))\big|
\,dx\,dy.
\end{align*}
Since $\sigma<0$, this is majorized by
\begin{align*}
C A \iint_{I^2} 
|x-y|^{-1+|\sigma|}
&\one_{|f(x)-f(y)|\le A^{-1}}(x,y)
\,dx\,dy
\\& \le CA
\sup_{y\in I} \int_I
|x-y|^{-1+|\sigma|}
\one_{|f(x)-f(y)|\le A^{-1}}(x)\,dx
\\& \le CA \sup_t 
\big|\{x\in I: |f(x)-t| \le A^{-1}\}|^{|\sigma|}.
\end{align*} 
\end{proof}

\begin{proof}[Proof of Corollary~\ref{cor:fneqn2}]
It suffices to establish the conclusion in the special case in which $r=1$,
since replacing $f_j$ by $r^{-1}f_j$ reduces the general case to this one.

Fix a nonnegative $C^\infty_0$ cutoff function $\zeta$.
We aim for an upper bound for
$\iint_{\reals^2} \one_{S(\bff),1)}\cdot\zeta\,dx\,dy$.
Let $h:\reals\to[0,\infty)$ be $C^\infty$ and compactly
supported, and be $\equiv 1$ on $[-1,1]$.
Consider instead the majorant
\begin{equation}  \label{themajorant}
\iint h\big((\sum_j (f_j\circ\varphi_j)\big)\cdot\zeta.
\end{equation} 
By implementing a partition of unity,
we may introduce $C^\infty_0$ cutoff functions 
satisfying $\prod_{j=1}^3 \eta_j(\varphi_j(x,y))\equiv 1$
on the support of $\zeta$, with $\eta_j$
supported on an interval $I_j$.
Then \eqref{themajorant} is equal to
\begin{equation} \label{hhatrepresentation}
c\int_\reals \widehat{h}(\lambda)
\Big(\int_{\reals^2} \prod_j 
(\eta_j\circ\varphi_j)e^{i\lambda f_j\circ\varphi_j}\,\zeta(x,y)\,dx\,dy\Big)
\,d\lambda.
\end{equation} 
By Theorem~\ref{thm:nonosc}, there exists $\sigma<0$
for which \eqref{hhatrepresentation} is majorized by
\begin{align} 
Cr\int_\reals (1+\lambda)^{-2} 
\prod_{j=1}^3 \norm{\eta_j e^{i\lambda f_j}}_{H^\sigma}\,d\lambda
& \le C\prod_{j=1}^3
\Big(\int_\reals (1+\lambda)^{-2} 
\norm{\eta_j e^{i\lambda f_j}}_{H^\sigma}^3\,d\lambda\Big)^{1/3}
\notag
\\
&\le C\prod_{j=1}^3
\Big(\int_\reals (1+\lambda)^{-2} 
\norm{\eta_j e^{i\lambda f_j}}_{H^\sigma}^2\,d\lambda\Big)^{1/3}
\label{majorizedby}
\end{align} 
since
$\norm{\eta_j e^{i\lambda f_j}}_{H^\sigma} 
\le \norm{\eta_j e^{i\lambda f_j}}_{L^2} = O(1)$
uniformly in all parameters because each $f_j$ is real-valued
and $\eta_j$ has bounded support.

For any index $j\in\{1,2,3\}$,
\begin{equation*}
\int_\reals (1+\lambda)^{-2} 
\norm{\eta_j e^{i\lambda f_j}}_{H^\sigma}^2\,d\lambda
\le C\sum_{k=0}^\infty  2^{-2k}
\int_{|\lambda| \le 2^k}
\norm{\eta_j e^{i\lambda f_j}}_{H^\sigma}^2\,d\lambda.
\end{equation*}
To each term in this sum, apply Lemma~\ref{lemma:nonlinearsobolev} with $A = 2^k$ to obtain
a majorization by
\begin{equation*}
 C\sum_{k=0}^\infty  2^{-2k}
\cdot 2^k 
\sup_t \big|\{x\in I_j: |f_j(x)-t|\le 2^{-k}\}\big|^{|\sigma|}
\le C \sup_t \big|\{x\in I_j: |f_j(x)-t|\le 1\}\big|^{|\sigma|}.
\end{equation*}
Inserting this bound into \eqref{majorizedby} gives
\[ |S(\bff,1)| \le C\prod_{j=1}^3 \sup_{t_j\in\reals} 
\big|\{x\in\varphi_j(B): |f_j(x)-t_j|\le 1\}\big|^{|\sigma|/3}.\]
\end{proof}

We have implicitly proved a lemma that may be useful in future work:
\begin{lemma}
Let $\sigma<0$. 
Let $\eta\in C^\infty(\reals)$ be supported in a closed bounded interval $I\subset\reals$.
There exists $C<\infty$, depending on $\sigma,\eta,|I|$,
such that for any measurable function $f:\reals\to\reals$, 
\begin{equation}
\int_\reals (1+\lambda^2)^{-1} \norm{\eta e^{i\lambda f}}_{H^\sigma}^2\,d\lambda
\le C\sup_t |\{x\in I: |f(x)-t|\le 1\}|^{|\sigma|}.
\end{equation}
\end{lemma}

\begin{proof}[Proof of Corollary~\ref{cor:fneqn1}]
Defining a measure $\mu$ on $I_j^2$ by $\,d\mu(x,y) = |x-y|^{-1+\gamma}
\,dx\,dy$, we have shown that
\begin{equation} 
\int_{\lambda\le 2^k r^{-1}} \norm{\eta_j e^{i\lambda f_j}}_{H^\sigma}^2\,d\lambda
\le C 2^{k} r^{-1} \mu(\{(x,y): |f_j(x)-f_j(y)| \le 2^{-k}r\}).
\end{equation} 
By summing over all nonnegative integers $k$ we deduce that
\begin{equation} 
\int_{\reals} r(1+r\lambda)^{-2} 
\norm{\eta_j e^{i\lambda f_j}}_{H^\sigma}^2\,d\lambda
\le C\mu(\{(x,y)\in I_j^2: |f_j(x)-f_j(y)| \le r\}).
\end{equation} 
If $f_j$ satisfies the hypothesis \eqref{nullset},
then 
$\mu(\{(x,y)\in I_j^2: |f_j(x)-f_j(y)| \le r\})\to 0$ as $r\to 0^+$.
Therefore $|S(\bff,r)|\to 0$ as $r\to 0^+$.
Therefore the set of points at which the equation \eqref{fneqn} holds 
is a Lebesgue null set.
\end{proof}

\section{A scalar sublevel set inequality with variable coefficients}
\label{section:variablecoeffs}

Throughout this section, $\varphi_j:\reals^2\to\reals^1$ are 
assumed to be linear and surjective.
Let $\Omega\subset\reals^2$ be a nonempty bounded open ball or parallelepiped. 
For $j\in\three$ let
$a_j:\overline{\Omega}\to\reals$ be $C^\omega$ functions.
By this we mean that $a_j$ extends to a real analytic
function defined in some neighborhood of $\overline{\Omega}$.
To any three-tuple $\bff = (f_j: j\in\three)$
of Lebesgue measurable functions $f_j:\Omega\to\reals$,
and to any $\eps>0$, associate the sublevel set
\begin{equation}
S(\bff,\eps) = \{x\in\Omega: 
\big|\sum_{j=1}^3 a_j(x) (f_j\circ\varphi_j)(x)\big| <\eps\}.
\end{equation}

\begin{theorem} \label{thm:sublevel}
Let $\varphi_j:\reals^2\to\reals^1$ be pairwise linearly independent linear mappings.
Let $\Omega,a_j$ be as above.
Suppose that for each $j\in\three$, $a_j(x)\ne 0$ for every $x\in\overline{\Omega}$.
Finally, suppose that for any nonempty open set $U\subset\Omega$
and any $C^\omega$ functions $F_j:U\to\reals$ satisfying
$\sum_{j=1}^3 a_j(x)(F_j\circ\varphi_j)(x)=0$
for every $x\in U$, all three functions $F_j$ vanish identically on $U$.
Then there exist $\gamma>0$ and $C<\infty$
such that 
for every $\eps>0$ and every three-tuple $\bff$ of Lebesgue measurable
functions 
satisfying 
\begin{equation} \label{lowerbound}
|f_1(y)|\ge 1\ \forall\, y\in\varphi_1(\Omega),
\end{equation}
the sublevel set $S(\bff,\eps)$ satisfies
\begin{equation} \label{mainineq}
|S(\bff,\eps)| \le C\eps^\gamma.
\end{equation}
\end{theorem}

The conclusion seems likely to remain valid
if the hypothesis that $a_j$ vanish nowhere,
is relaxed to $a_j$ not vanishing identically.
We emphasize that the mappings $\varphi_j$ are assumed in Theorem~\ref{thm:sublevel} to be linear.

Several results related to Theorem~\ref{thm:sublevel} are known,
besides those in \S\ref{section:FEqns}.
If each $a_j$ is constant and the mappings $\varphi_j$ are linear, 
then whenever $\sum_j f_j\circ\varphi_j$ vanishes Lebesgue almost everywhere,
each $f_j$ must agree almost everywhere with an affine function.
If $|\sum_j f_j\circ\varphi_j(x)|\le\eps$
for all $x\in\Omega\setminus E$, and if $|E|$ is sufficiently small,
then there exist affine functions $L_j$ satisfying
$|f_j(y)-L_j(y)| \le C\eps$
for all $y\in\varphi_j(\Omega)\setminus E_j$
with $|E_j| \le C|E|$.
However, no inequality of the form \eqref{mainineq}, with power law dependence
on $\eps$, is known for this linear constant coefficient situation.

In the proof, it suffices to treat the special case in which
$|f_j(y)|\le 2$ for every $y\in\varphi_j(\Omega)$ and each $j\in\three$,
and $|f_1(y)|\in[1,2]$ for every $y\in\varphi_1(\Omega)$.
Indeed, for $k\ge 0$ define $E_k$ to be the set of all $x\in\scripte$
that satisfy $2^k\le \max_j |f_j\circ \varphi_j(x)| < 2^{k+1}$.
Then $E_k = S(2^{-k}\bff,2^{-k}\eps)$. 
Therefore the conclusion of the special case gives 
$|E_k| \lesssim 2^{-\gamma k}\eps^\gamma$.
Summing over $k$ yields the desired bound for $|\scripte|$.

We may assume that $\overline{\Omega} = [0,1]^2$,
by partitioning a small neighborhood of $\overline{\Omega}$
into finitely many cubes, making an affine change of coordinates
in each, and treating each cube separately.
In part of the proof we use coordinates $(x,y)\in[0,1]\times[0,1]$,
and write $D_1 = \frac{\partial}{\partial x}$
and $D_2 = \frac{\partial}{\partial y}$.
By making a linear change of variables in $\reals^2$,
We may also assume without loss of generality that $\varphi_1(x,y)\equiv x$,
$\varphi_2(x,y) \equiv y$, and $\varphi_3(x,y)=x+y$. 

It suffices to show that there exists $\eps_0$ such that
the conclusion holds for all $\eps\in(0,\eps_0]$.
It is no loss of generality to assume, as we will, that
\begin{equation} \label{bycontradiction}
|\scripte| \ge\eps^{\delta_0} \end{equation}
for a sufficiently small exponent $\delta_0>0$.
Indeed, if this assumption fails to hold then we have the stated conclusion,
with $\gamma = \delta_0$ and $C=1$.

Rewrite the inequality characterizing $\scripte = S(\bff,\eps)$ as 
\begin{equation} \label{approx1}
f_3(x+y) + a(x,y)f_1(x) = b(x,y)f_2(y) + O(\eps)
\qquad  \forall\, (x,y)\in\scripte
\end{equation}
with $a = a_1/a_3$ and $b = -a_2/a_3$.
Let $c_0>0$ be small and define
\begin{equation}
\tilde\scripte_{1} = \big\{
y\in[0,1]: |\{x\in[0,1]: (x,y)\in\scripte\}| \ge c_0|\scripte|
\big\}.
\end{equation}
If $c_0$ is sufficiently small then
by Fubini's theorem and the Cauchy-Schwarz inequality,
\begin{equation} 
|\tilde \scripte_{1}| \gtrsim |\scripte|^2 \gtrsim \eps^{2\delta_0}.
\end{equation}
Henceforth we replace $\scripte$ by its subset 
$\scripte_1 =  \{(x,y)\in\scripte: y\in \tilde\scripte_1\}$.

Let $\scripte_2\subset\reals^3$ be the set of all ordered triples
$(x,y,s)\in\reals^3$ such that $(x-s,y+s)\in \scripte_1$
and $(x,y)\in\scripte_1$. This set satisfies
$|\scripte_2| \gtrsim |\scripte_{1}|^2\gtrsim |\scripte|^4$
by the Cauchy-Schwarz inequality.
Indeed, 
\[ |\scripte_1| = C\int_I |\{(x,y)\in\scripte_1: x+y=t\}|\,dt\]
where $I$ is a bounded subinterval of $\reals$
and with $|\cdot|$ denoting one-dimensional Lebesgue measure in the integral.
Therefore
\[ |\scripte_1|^2\le C\int_I 
|\{(x,y)\in\scripte_1: x+y=t\}|^2\,dt
= C|\{((x,y),(x',y'))\in \scripte_1\times\scripte_1: x+y=x'+y'\}|,  \]
with the last $|\cdot|$ denoting the natural three-dimensional Lebesgue
measure on the hyperplane in $\reals^4$ defined by this equation,
with the constant $C$ permitted to change from one occurrence to the next.
The set of all pairs
$((x,y),(x',y'))$ that satisfy $x+y=x'+y=$
is in measure-preserving one-to-one correspondence with
$\scripte_2$ via the relation $(x',y') = (x-s,y+s)$.

For any $(x,y,s)\in\scripte_2$,
\begin{equation} \label{approx2}
f_3(x+y) 
+ a(x-s,y+s)f_1(x-s) 
= b(x-s,y+s)f_2(y+s) + O(\eps) 
\ \forall\,(x,y,s)\in\scripte_2.
\end{equation}
For any $(x,y,s)\in\scripte_2$ we have the two approximate relations
\eqref{approx1},\eqref{approx2}. 
The contributions of $f_3$ cancel when these two relations are subtracted, leaving 
\begin{multline} \label{approx3}
a(x-s,y+s)f_1(x-s)  - a(x,y)f_1(x)
\\
= b(x-s,y+s)f_2(y+s) - b(x,y)f_2(y) + O(\eps) 
\qquad\forall\,(x,y,s)\in\scripte_2.
\end{multline}


The set $\scripte_3$ of all $(x,x',s,y)\in\reals^4$ such that
both $(x,y,s)$ and $(x',y,s)$ belong to $\scripte_2$ satisfies
\begin{equation}
|\scripte_3| \gtrsim |\scripte_2|^2 \gtrsim |\scripte|^8 \gtrsim\eps^{8\delta_0}.
\end{equation}
Consider any such $(x,x',s,y)$.  
Consider the conjunction of \eqref{approx3} with
the corresponding relation with $(x,y,s)$ replaced by $(x',y,s)$. 
Express this pair of relations as the approximate matrix equation  
\begin{equation} \label{approxmatrixeqn1}
B(x,x',s,y) 
\begin{pmatrix} f_2(y) \\ f_2(y+s) \end{pmatrix}
= A(x,x',s,y) + O(\eps)
\end{equation}
in which the coefficient matrices $A,B$ are the square matrix
\begin{equation} B(x,x',s,y) =
\begin{pmatrix}
b(x-s,y+s) & -b(x,y) \\ b(x'-s,y+s) & -b(x',y)
\end{pmatrix}
\end{equation}
and the column matrix
\begin{equation} A(x,x',s,y) = 
\begin{pmatrix} a(x-s,y+s)f_1(x-s)  - a(x,y)f_1(x)
\\ a(x'-s,y+s)f_1(x'-s)  - a(x',y)f_1(x') \end{pmatrix},
\end{equation}
respectively.

\begin{lemma}
As a function of $(x,x',s,y)$, the determinant $\det(B)$ does not vanish identically. 
\end{lemma}

\begin{proof}
Assume to the contrary that $\det(B)\equiv 0$.
Then the ratio $b(x-s,y+s)\,/\,b(x'-s,y+s)$ is independent of $s$,
whence 
\[ \frac{\partial^2}{\partial s\,\partial x} \ln|b(x-s,y+s)| \equiv 0.\]
Therefore $b$ takes the form 
\[ b(x,y) \equiv h(x+y)\cdot k(y) \]
for some smooth nowhere vanishing functions $h,k$.

Choosing $f_1(x)\equiv 0$, $f_2(y) = k(y)^{-1}$,
and $f_3(z) = h(z)$ , we have
\begin{equation}
f_3(x+y)  + a(x,y)f_1(x)  \equiv b(x,y)f_2(y)
\end{equation}
on a nonempty open set. This contradicts the hypothesis of Theorem~\ref{thm:sublevel} that
the functional equation has no solution except the trivial
solution $f_1\equiv f_2 \equiv f_3\equiv 0$. 
\end{proof}

For any $(x,x',s,y)$, multiply both sides of 
the approximate matrix equation \eqref{approxmatrixeqn1}
by the cofactor matrix of $B(x,x',s,y)$ to conclude that
\begin{equation} \label{approxmatrixeqn2}
\det(B)(x,x',s,y) \cdot g(y) = \scripta(x,x',s,y) + O(\eps),
\end{equation}
where $\scripta(x,x',s,y)$ is one of the two components of the product of the
cofactor matrix of $B(x,x',s,y)$ with $A(x,x',s,y)$.
Thus $\scripta$ 
is a linear combination of products of the given coefficients
$a,b$, evaluated at points that are functions
of $(x,x',s,y)$, with coefficients in $[-2,2]^4$.
Those coefficients are the quantities
$f_1(x-s),f_1(x),f_1(x'-s),f_1(x')$, whose dependence on
$(x,x',s)$ is merely Lebesgue measurable and is unknown.
However, $\scripta$ depends linearly, hence real analytically,
on those coefficients.

By partitioning $[0,1]^2$ into finitely many smaller cubes,
and identifying each subcube again with $[0,1]^2$
via an affine change of variables, we may assume that each coefficient
$a_j$ is defined and analytic in a large fixed ball that contains $[0,1]^2$.
Define $K\subset\reals^7$ to be the set of all tuples
$\theta = (x,x',s,r) = (x,x',s,r_1,r_2,r_3,r_4)$
such that $r\in[-2,2]^4$, $(x,x')\in[0,1]^2$, and $s\in[-2,2]$.
$K$ is compact and connected. 
Let $(y,\theta)$ vary over $[0,1]\times K$.
\eqref{approxmatrixeqn2} can thus be written as
\begin{equation} \label{approxmatrixeqn3}
\det(B)(y,\theta) \cdot f_2(y) = \scriptastar(y,\theta) + O(\eps)
\end{equation}
for all $(y,\theta)\in K$ for which $(x,x',s,y)\in\scripte_3$, 
with $\scriptastar$ a real analytic function of $(y,\theta)$
in a neighborhood of $[0,1]\times K$.

The set of all $(x,x',s,y)\in\scripte_3$ 
has Lebesgue measure $\gtrsim |\scripte|^{C_0}\ge\eps^{C_0\delta_0}$.
On the other hand, since the $C^\omega$ function $(x,x',s,y)\mapsto \det(B)
(x,x',s,y)$ does not vanish identically, there exists $\eta>0$ such that
\begin{equation}
\big|\{(x,x',s,y):  |\det(B)(x,x',s,y)|\le r\} \big| \lesssim r^\eta
\ \forall\,r\in(0,1].
\end{equation}
Choose a constant $C_1\in\reals^+$ that satisfies $\eta\cdot C_1 > C_0$.
Applying the preceding inequality with $r = \eps^{C_1\delta_0}$,
$r^\eta$ is small relative to $\eps^{C_0\delta_0}$, and thus
we may conclude that there exists $(x,x',s)$ satisfying
\begin{equation}
\big|\{
y\in[0,1]: (x,x',s,y)\in\scripte_3
\ \text{ and } \ |\det(B)(x,x',s,y)|\ge\eps^{C_1\delta_0} 
\}\big|
\gtrsim \eps^{\delta_0}.
\end{equation}
The conclusion is that there exists 
$\bar\theta=\bar\theta(\bff,\eps)\in K$ satisfying
\eqref{approxmatrixeqn3} for the indicated set of pairs $(y,\theta)$, with 
\begin{equation} |\det(B)(x,x',s,y)|\ge\eps^{C_1\delta_0}. \end{equation}
For such $\bar\theta$,
\begin{equation} \label{gconclusion}
|f_2(y) -\det(B)(y,\bar\theta)^{-1} \scriptastar(y,\bar\theta)| = O(\eps)
\end{equation}
for all $y$ in a set of measure $\gtrsim \eps^{\delta_0}$.

Revert to the initial notation, with mappings $\varphi_j$ and coefficients $a_j$.
The conclusion proved thus far can be summarized as follows.
Let $a_j,\varphi_j$ satisfy the hypotheses of Theorem~\ref{thm:sublevel}.
Let $\delta_0,\eps_0>0$ be sufficiently small.
There exist a compact connected set $K\subset \reals^7$, 
and a function $F_2:[0,1]\times K\to\reals$
that extends meromorphically to a neighborhood of $[0,1]\times K$,
with the following property.
Let $\bff$ and $\eps\in(0,\eps_0]$ 
satisfy the hypotheses of the theorem,
as well as the auxiliary condition $|S(\bff,\eps)|\ge \eps^{\delta_0}$.
Then there exist $\scripte'\subset S(\bff,\eps)$  
satisfying $|\scripte'|\gtrsim|S(\bff,\eps)|^C$,
and $\bar\theta = \bar\theta(\bff,\eps)\in K$, 
such that the triple $(f_1,\tilde f_2,f_3)$
defined by $\tilde f_2(y) = F_2(y,\bar\theta)$
satisfies
\begin{equation}
\big| a_2(x) \tilde f_2(\varphi_2(x))
+ \sum_{j\ne 2} a_j(x) f_j(\varphi_j(x))\big|
= O(\eps)\ \forall\,x\in\scripte'
\end{equation}
and 
\begin{equation}
|f_2(y)-\tilde f_2(y)| = O(\eps) \ \forall\,y\in\varphi_2(\scripte').
\end{equation}
Moreover, the function $F_2$ factors as
$F_2(y,\theta) = \alpha(y,\theta)/\beta(y,\theta)$
with $\alpha,\beta$ both analytic in a neighborhood
of $[0,1]\times K$ and satisfying 
\begin{equation}
|\beta(y,\bar\theta)| \ge \eps^{C_1\delta_0} \ \forall\, y\in\varphi_2(\scripte').
\end{equation}

This reasoning can be applied twice more in succession, with the roles of the
indices $j\in\three$ permuted, to approximate each of $f_1,f_3$
by $C^\omega$ functions in the same way as has been done for $f_2$.
With each iteration, $\scripte$ is replaced by a subset,
and one of the functions $f_k$ is replaced by an approximating meromorphic
function $\tilde f_k$; these replacements are retained through subsequent iterations.
The conclusion may be summarized as follows,
incorporating a change in the meaning of the auxiliary space $K$.

Let $a_j,\varphi_j$ be as in the statement of Theorem~\ref{thm:sublevel}.
Let $\delta_0,\eps_0>0$ be sufficiently small.
There exist a compact connected set $K\subset \reals^{21} = (\reals^7)^3$
and three $C^\omega$ functions $F_j:[0,1]\times K\to\reals$,
such that for any $\bff$ and any $\eps\in(0,\eps_0]$ 
satisfying the hypotheses of the theorem with associated sublevel
set $S(\bff,\eps)$ satisfying 
$|S(\bff,\eps)|\ge\eps^{\delta_0}$,
there exist a subset $\scripte'\subset S(\bff,\eps) \subset[0,1]^2$  
satisfying $|\scripte'|\gtrsim|S(\bff,\eps)|^C$
and an associated parameter $\bar\theta = \bar\theta(\bff,\eps)\in K$, 
such that the ordered triple of approximating functions 
$(\tilde f_j: j\in\three)$
defined by $\tilde f_j(y) = F_j(y,\bar\theta)$ satisfies
\begin{equation}
\big| \sum_{j=1}^3 a_j(x) \tilde f_j(\varphi_j(x))\big|
= O(\eps)\ \forall\,x\in\scripte'
\end{equation}
and 
\begin{equation}
|f_j(y)-\tilde f_j(y)| = O(\eps) \ \forall\,y\in\varphi_j(\scripte').
\end{equation}
Moreover, for each $j\in\three$, the function $F_j$ factors almost everywhere 
in its domain $[0,1]\times K$ as
\[ F_j(y,\theta) = \alpha_j(y,\theta)/\beta_j(y,\theta)\]
with $\alpha_j,\beta_j$ analytic in a neighborhood of $[0,1]\times K$. 
The denominators $\beta_j$ satisfy
\begin{equation}
|\beta_j(y,\bar\theta(\bff,\eps))| \ge \eps^{C_1\delta_0}\ 
\ \forall\, y\in\varphi_j(\scripte'). 
\end{equation}
The exponents $C,C_1$ depend only on the data $a_j,\varphi_j$
and the choice of $\eps_0,\delta_0$.

Consider the function
of $(x,\theta)\in[0,1]^2\times K$ defined by
\begin{equation}
H(x,\theta) = 
\sum_{j=1}^3 a_j(x) \cdot \alpha_j(\varphi_j(x),\theta)
\cdot \prod_{i\ne j} \beta_i(\varphi_i(x),\theta)
\end{equation}
along with the partial derivatives
$\frac{\partial^\alpha}{\partial x^\alpha} H(x,\theta)$ with respect to $x$
of $H$, indexed by $\alpha\in\{0,1,2,\dots\}^2$.
This function $H$ is arrived at by multiplying
$\sum_{j=1}^3 a_j(x) F_j(\varphi_j(x),\theta)$
by
$\prod_{i=1}^3 \beta_i(\varphi_i(x),\theta)$
in order to arrive at a function that is holomorphic,
rather than merely meromorphic.

If $x\in\scripte'$ then 
\[ |H(x,\theta)| \lesssim
\big| \sum_j a_j(x) F_j(\varphi_j(x,\theta))\big|
= O(\eps)\]
since the functions $\beta_i$ are bounded.
Thus in order to majorize the Lebesgue measure
of the sublevel set $S(\bff,\eps)$, it will suffice to produce
a satisfactory majorization of the measure of a sublevel
set of $x\mapsto H(x,\bar\theta(\bff,\eps))$. 

To analyze sublevel sets associated to $H$
requires information concerning $H$,
and information concerning $\bar\theta(\bff,\eps)$.
But first, we review a happy general property \eqref{happy} of 
real analytic functions that depend real analytically
on auxiliary parameters. 
See Bourgain \cite{bourgainold}, and Stein and Street \cite{steinstreet}. 
There exist $N,C<\infty$ such that
for any multi-index satisfying $|\alpha| = N+1$,
for every $(x,\theta)\in[0,1]^2\times K$,
\begin{equation} \label{happy}
\big|\frac{\partial^\alpha}{\partial x^\alpha} H(x,\theta)\big|
\le C \sum_{|\beta|\le N} 
\big|\frac{\partial^\beta}{\partial x^\beta} H(x,\theta)\big|.
\end{equation}

Introducing the nonnegative $C^\omega$ function
\begin{equation}
\tilde H(x,\theta) = \sum_{|\beta|\le N} 
\big|\frac{\partial^\alpha H(x,\theta)}{\partial x^\alpha} \big|^2,
\end{equation}
it follows from the Cauchy-Schwarz inequality that
$\tilde H$ satisfies the differential inequality
\begin{equation}
|\nabla_x \tilde H(x,\theta)|\le C'|\tilde H(x,\theta)|
\end{equation}
uniformly for all $(x,\theta)\in [0,1]^2 \times K$.
This differential inequality allows us to replace $\tilde H(x,\theta)$
by a function of $\theta$ alone; it implies that
there exists $C\in(0,\infty)$ such that the function 
$G(\theta) = \tilde H((0,0),\theta)$
satisfies
\begin{equation}
C^{-1} G(\theta) \le \tilde H(x,\theta) \le CG(\theta)
 \text{ uniformly for all $(x,\theta)\in [0,1]^2 \times K$.}
\end{equation}
$G\in C^\omega$ in a neighborhood of $K$,
and $H(x,\theta)=0$ for every $x\in[0,1]^2$
if and only if $G(\theta)=0$.

The following result, a variant of a lemma often
attributed to van der Corput, is essentially well known.
\begin{lemma}
Let $N<\infty$.
Let $C_1,C_2\in(0,\infty)$.
There exist $C<\infty$  and $\rho>0$ with the following property.
Let $\psi\in C^{N+1}([0,1]^2)$
satisfy $\norm{\psi}_{C^{N+1}}\le C_2$
and
\[ \sum_{0\le|\alpha|\le N} |\partial^\alpha\psi(x)| \ge C_1
\ \forall\,x\in[0,1]^2.\]
Then
for any $\eps>0$,
\begin{equation}
\big| \{x\in[0,1]^2:
|\psi(x)|\le\eps\} \big|
\le C\eps^\rho.
\end{equation}
\end{lemma}

The upper bound on the $C^{N+1}$ norm cannot be dispensed with entirely
in this formulation. 
Consider for instance the example
$\psi(x) = \eps \sin(\eps^{-1} x_1)$, with $N=2$. 

A consequence of the lemma is for any $\theta$ for which $G(\theta)\ne 0$,
for any $\eta\in(0,\infty)$,
\begin{equation} \label{vdc:consequence}
\big|\{ x\in[0,1]^2: |H(x,\theta)| \le \eta G(\theta) \}\big|
\le C\eta^\rho.  \end{equation}

To complete the proof of the theorem, it would be desirable to know
that $G$ does not vanish
identically on $K$. We will not actually prove that this is the case.
Instead, note that 
if $|S(\bff,\eps)|\le \eps^{\delta_0}$
for every datum $(\bff,\eps)$
satisfying the hypotheses of the theorem,
then the desired conclusion holds with $\gamma = \delta_0$.
Thus it suffices to treat the case in which there exists at least one datum
$(\bff,\eps)$ that satisfies the reverse inequality
$|S(\bff,\eps)| > \eps^{\delta_0}$,
along with the hypotheses of the theorem.
We will prove that 
$G(\bar\theta(\bff,\eps))\ne 0$
for any such datum, and hence may assume 
in the remainder of the proof that $G$ does not vanish identically on $K$.

To prove that $G(\bar\theta)\ne 0$ in this situation, with 
$\bar\theta = \bar\theta(\bff,\eps)$, observe first that
none of the factors $\beta_j(y,\bar\theta)$ vanishes identically
as a function of $y$. Indeed, each such factor is $\gtrsim \eps^{C_1\delta_0}$ on a 
set whose Lebesgue measure is minorized by a positive quantity.
By dividing by $\prod_i \beta_i(\varphi_i(x),\theta)$
in the definition of $H$, 
we conclude that if $G(\bar\theta)=0$ then
$\sum_{j=1}^3 a_j(x) F_j(\varphi_j(x),\bar\theta)=0$
almost everywhere as a function of $x\in[0,1]^2$. 
By the main hypothesis of Theorem~\ref{thm:sublevel}, 
$\sum_{j=1}^3 a_j(x) F_j(\varphi_j(x),\bar\theta)$
vanishes on an open set of values of $x$ only if each function
$x\mapsto F_j(\varphi_j(x),\bar\theta)$ vanishes identically.
However, the construction has $|f_1(y)-F_1(y,\bar\theta)|=O(\eps)$ for $y$
in a subset of positive measure, and by hypothesis,
$|f_1(y)|\in[1,2]$ for almost every $y$. Therefore $F_1(y,\bar\theta)\ne 0$. 

Define the zero variety
\begin{equation} Z = \{\theta\in K: G(\theta)=0\}. \end{equation}
$G$ is  $C\omega$ and nonnegative
in a neighborhood of $K$, $G$ does not vanish identically on $K$,
and $K$ is connected. 
Therefore by a theorem of \L{}ojasiewicz \cite{Loja1959},
there exist $c,\tau>0$ such that
\begin{equation}
G(\theta) \ge c\distance(\theta,Z)^\tau \ \forall\,\theta\in K.
\end{equation}

If $\bff,\eps,S(\bff,\eps)$ satisfy the hypotheses, then
$\bar\theta = \bar\theta(\bff,\eps)$ satisfies
$\distance(\bar\theta,Z) \gtrsim \eps^{C\delta_0}$.
Indeed, consider any $x\in\scripte'$.
Then for $y = \varphi_1(x)$, $|f_1(y)-F_1(y,\bar\theta)|= O(\eps)$
and $|f_1(y)|\in[1,2]$,
so $|F_1(y,\bar\theta)|\ge 1-O(\eps)\ge\tfrac12$. 
Since $F_1 = \alpha_1/\beta_1$, it follows that
\begin{equation} 
|\alpha_1(y,\bar\theta)| \ge \tfrac12 |\beta_1(y,\bar\theta)| 
\gtrsim \eps^{C_1\delta_0}.
\end{equation}
The function $\alpha_1$ is real analytic with respect to both variables,
hence is Lipschitz, and vanishes identically on $Z$. 
Therefore
$\distance(\bar\theta,Z)\gtrsim\eps^{C\delta_0}$, 
and consequently
$G(\bar\theta) \gtrsim\eps^{C\delta_0}$.
Applying \eqref{vdc:consequence} gives
\begin{equation}
\big|\{ x\in[0,1]^2: |H(x,\bar\theta)| = O(\eps)  \}\big|
=O((\eps^{1-C\delta_0})^\rho). 
\end{equation}
If $\delta_0$ is chosen to be sufficiently small then $1-C\delta_0>0$, 
so this inequality becomes
\begin{equation}
\big|\{ x\in[0,1]^2: |H(x,\bar\theta)| = O(\eps)  \}\big|
=O(\eps^\gamma), 
\end{equation}
for a certain exponent $\gamma>0$ that depends only on the coefficients $a_j$
and the mappings $\varphi_j$.
This completes the proof of Theorem~\ref{thm:sublevel}.
\qed

\section{A remark and a question}

Continuing to assume linearity of the mappings $\varphi_j$,
more can be deduced from the analysis in \S\ref{section:variablecoeffs}. 
Drop the assumption that no nontrivial
solution exists, and ask whether for any $\bff=(f_1,f_2,f_3)$ and any $\eps$, 
$\bff$ can be approximated within $O(\eps)$ 
on some subset $S'\subset S(\bff,\eps)$
satisfying $|S'|\gtrsim |S(\bff,\eps)|^C$,
by an $\reals^3$--valued function $\bg$ 
drawn from a finite-dimensional family of $C^\omega$ functions
that depends only on the data $(\varphi_i,a_i: i\in\{1,2,3\})$. 
More generously, in light of that analysis, 
we allow meromorphic approximants by asking whether there exist $g_j$ and $\beta_j$,
drawn from such a family, such that $\beta_j$ does not vanish
identically and $\beta_j f_j-g_j = O(\eps^{1-\rho})$ on $\varphi_j(S')$.
We refer to this as the approximability property.

It suffices to approximate $f_k$ by a component $g_k$ of such a $\bg$
for a single index $k$, for then a rather simple analysis
can be applied to the relation 
$\sum_{j\ne k} a_j (f_j\circ\varphi_j) = -a_k (g_k\circ\varphi_k) +O(\eps)$;
restrict this equation to level curves of $\varphi_i$ for
each of the two indices $i\ne k$ in turn and exploit the transversality hypothesis.

The analysis in \S\ref{section:variablecoeffs} shows that  $f_2$ can be so approximated,
except possibly in the special case in which $a_2(x,y)/a_3(x,y)$
can be factored in the form $h(x+y)/k(y)$,
that is, $(h\circ\varphi_3)\,/\,(k\circ\varphi_2)$.
This reasoning can be repeated for any permutation of the
indices $1,2,3$. The conclusion, in invariant form with the mappings 
$\varphi_j$ assumed to be linear, 
is that the approximability property holds, and follows from the analysis sketched,
for all but a small family of exceptional cases. 
Each of those exceptional cases can be transformed, by
application of symmetries of the problem, to one of the two examples
\begin{align} 
\label{classical}
&f_1(x) + f_2(y) + f_3(x+y)=0. 
\\
\label{canonical!} 
& f_1(x) + f_2(y) + e^x f_3(x+y)=0. 
\end{align}
These symmetries are linear changes of variables in $\reals^2$ and
in the domains $\reals^1$ of the three mappings $\varphi_j$,
multiplication of the equation by an arbitrary nowhere vanishing
$C^\omega$ function $b(x,y)$, 
and incorporation of coefficients into functions $f_j$
via multiplicative substitutions $\tilde f_j(x) = f_j(x) u_j(x)$,
with $u_j\in C^\omega$ vanishing nowhere in the relevant domain.
The equation \eqref{canonical!} has a two-dimensional space of solutions $\bff$, 
with $f_3(x) = c_1e^{-x} + c_2$ for arbitrary coefficients $c_1,c_2\in\reals$.
The approximability property does not hold for either \eqref{classical}
or \eqref{canonical!};
counterexamples can be constructed by exploiting multiprogressions of arbitrarily high rank.
\qed

\medskip
\begin{question} \label{question:19.1}
Let $\eps>0$, and let $\bff$ be measurable.
Let $\varphi_j(x,y) = x$, $=y$, and $=x+y$ for $j=1,2,3$, respectively.
Let $S(\bff,\eps)$ be the set of all $(x,y)\in B$ satisfying $|f_1(x)+f_2(y)+e^x f_3(x+y)|<\eps$.

Do there exist an exact $C^\omega$ solution $\bff^*$ of \eqref{canonical!}
and a subset $S'\subset S(\bff,\eps)$ satisfying $|S'|\ge c|S(\bff,\eps)|^C$
such that $|f_j\circ\varphi_j(x,y)-f_j^*\circ\varphi_j(x,y)|\le C\eps$ for every $(x,y)\in S'$?
The constants $c,C$ are to be independent of $\bff,\eps$.  \end{question}

The answer is negative for the equation \eqref{classical}.

\medskip
\begin{question} \label{question:19.2}
Does Theorem~\ref{thm:sublevel} remain valid
if the mappings $\varphi_j$ are assumed to be merely real analytic
with pairwise transverse gradients, rather than linear?
\end{question}

A manuscript answering Question~\ref{question:19.2} in the affirmative, under certain
auxiliary hypotheses, is in progress \cite{sublevel4}.
That result is used to establish
a quadrilinear analogue of Theorem~\ref{thm:nonosc}  --- again, under
auxiliary hypotheses --- in \cite{quadrilinear}.
It would be desirable to go farther, dropping the hypothesis that no exact $C^\omega$
solutions of the underlying equation exist, and weakening the conclusion
to approximability by exact solutions, as in Question~\ref{question:19.1}.

\section{Large sublevel sets: An example} 
\label{section:sublevelexample}

Consider the ordered triple of submersions $[0,1]^2\to\reals$
defined by $(x,y)\mapsto x$, $\mapsto y$, and $\mapsto x+y$.
To any ordered triple $(f,g,h)$ of Lebesgue measurable functions associate
the sublevel set 
\begin{equation}
\scripte = \{(x,y)\in[0,1]^2: 
|g(x)-h(x+y)|<\eps
\ \text{ and } \ 
|y-f(x)-h(x+y)|<\eps \}
\end{equation}
defined by the indicated inhomogeneous system of two inequalities for $(f,g,h)$.
The reasoning developed above, for instance in \S\ref{section:sublevel1}, 
demonstrates that
\begin{equation} \label{crudebound} |\scripte|  = O(\eps^{1/2}).  \end{equation}
That reasoning may appear to have been wasteful,
and indeed,  $|\scripte| = O(\eps)$ uniformly for all affine functions $f,g,h$. 
Here we show, via a construction based on multiprogressions of rank $2$,
that the exponent $1/2$ in \eqref{crudebound} cannot be improved. 

Let $\eps>0$ be small, with $\eps^{-1/2}\in\naturals$.
Set $N = \eps^{-1/2}$.
For each $k\in\integers$, define
\begin{equation}
f(x) =  k\eps^{1/2}-x
\ \text{ whenever $|x - k\eps^{1/2}|<\tfrac12 \eps^{1/2}$.}
\end{equation}
Define
\begin{equation} \label{notypo1}
g(y) =  k\eps^{1/2} + k\eps
\ \text{ whenever $|y-k\eps^{1/2}|<\tfrac12\eps^{1/2}$.}
\end{equation}
For each $t\in\reals$ there exist unique $k,n\in\integers$
with $0\le n < N$ such that
$|t - (k\eps^{1/2}+n\eps)|<\tfrac12 \eps$. 
Define
\begin{equation} \label{notypo2}
h(t) =  n\eps^{1/2}+n\eps
\ \text{ whenever } 
|t - (k\eps^{1/2}+n\eps)|<\tfrac12 \eps. 
\end{equation}

For $m,n\in\integers$ satisfying $0\le n< N$,
define $\scripte(m,n)$ to be the set of all $(x,y)\in \reals^2$ that satisfy the three inequalities
\begin{equation} \left\{ \begin{aligned}
& |y-n\eps^{1/2}|< \tfrac12 \eps^{1/2}, 
\\& |x-(m-n)\eps^{1/2}|<\tfrac12\eps^{1/2},
\\& |x+y-(m\eps^{1/2}+n\eps)|< \tfrac12 \eps.
\end{aligned} \right. \end{equation}
The sets $\scripte(m,n)$ are pairwise disjoint and satisfy
\begin{equation} |\scripte(m,n)| = \eps^{3/2} + O(\eps^2). \end{equation}
The number of indices $(m,n)\in \integers \times \{0,1,2,\dots,N-1\}$ 
for which $\scripte(m,n)\subset[0,1]^2$ is $\ge c\eps^{-1}$.

If $\scripte(m,n)\subset [0,1]^2$, then $\scripte(m,n)\subset \scripte$. 
Indeed, let $(x,y)\in\scripte(m,n)$.
Firstly,
\begin{equation} g(y)-h(x+y) = 0  \end{equation}
since both $g(y)$ and $h(x+y)$ are defined to be $n\eps^{1/2}+n\eps$
in this region.
Secondly,
\begin{align*}
f(x)+h(x+y)-y 
&=  ((m-n)\eps^{1/2} -x)  + h(x+y) -y
\\ & =  - \big( x+y-m\eps^{1/2}-n\eps\big) + \big(h(x+y) -n\eps^{1/2}-n\eps\big).
\end{align*}
Since $x+y$ lies in the strip indicated in the definition
of $\scripte(m,n)$,
\[ |x+y-m\eps^{1/2}-n\eps|<\tfrac12\eps
\text{ and } h(x+y) = n\eps^{1/2} +n\eps.\]
Consequently
\begin{equation} |y-f(x)-h(x+y)| < \tfrac12\eps.  \end{equation}

Thus $\scripte(m,n)\subset\scripte$ whenever $\scripte(m,n)\subset[0,1]^2$. 
There are $\ge c\eps^{-1}$ such sets, pairwise disjoint
and satisfying $|\scripte(m,n)| \ge \eps^{3/2}-O(\eps^2)$. Therefore
\begin{equation} |\scripte| 
\ge c' \eps^{1/2} \end{equation}
for a certain constant $c'>0$.
\qed


\section{Remarks on sublevel sets} \label{section:sublevelquestion}

Implicit in the discussion is a variant of the usual notion of a sublevel set bound.
Let $d\ge 1$ be an arbitrary dimension.
Let $\eps,\delta>0$  and $N\in\naturals$ be parameters.

Let $\scripts$ be the collection of all sets 
$S\subset\delta\integers=\{\delta n: n\in\integers\}$
of cardinality exactly $|S|=N$. 
Let there be given $d$ functions $h_j$,
each with domain $[0,1]$ and with range in $\scripts$.
Codomains consisting of sets of cardinality $N$, rather than of $N$-tuples,
are natural in the variant that we seek to formulate.
Set $\bh = (h_j: j\in\{1,2,\dots,d\})$.

Let $\phi:[0,1]^d\to\reals$ be $C^1$.
Define $E_N(\phi,\bh) \subset[0,1]^d$ 
to be the set of all $\bx\in[0,1]^d$
for which there exists  
$(s_1,\dots,s_d)\in (\delta\integers)^d$, 
with each $s_j\in h_j(x_j)$,
satisfying 
\begin{equation} |\nabla_j\phi(\bx)-s_j|\le\eps. \end{equation}
Define
\begin{equation}
\Lambda_N(\phi) = \sup_{\bh} |E_N(\phi,\bh)|.
\end{equation}

\begin{question}
For $\phi$ or for a class of functions $\phi$,
what upper bounds are valid for 
$\Lambda_N(\phi)$? 
\end{question}

In the special case $N=1$,
in which $h_j(x_j)$ can be regarded as a scalar rather than a set,
we are asking for an upper bound for $\big|\{\bx: |\nabla(\tilde\phi)<\eps|\}\big|$,
with $\tilde\phi(\bx) = \phi(\bx)-\sum_j H_j(x_j)$ and
$H'_j = h_j$.
There is a trivial  majorization
\begin{equation}
\Lambda_N(\phi) \le N^d \Lambda_1(\phi),
\end{equation}
obtained by regarding each $h_j$
as a collection of $N$ real-valued functions $h_{j,i}$,
leading to an inclusion 
\[ E_N(\phi,\bh) \subset \bigcup_{i_1,\dots,i_N} E_1(\phi,(h_{i_1},\dots,h_{i_N})). \]
Thus
\[\Lambda_N(\phi) \le N^d \Lambda_1(\phi).\]
We hope that for large $N$, for natural classes of $\phi$ 
such as compact families of $C^\omega$ functions,
stronger bounds hold for $\Lambda_N(\phi)$.

This is a simplification of the issue that arose,
with $N$ comparable to $\lambda^t$ for a certain positive exponent $t$,
in the proof of Theorem~\ref{mainthm}. 
Let $L_j:[0,1]^2\to\reals$ be submersions, for $j\in\{1,2,3\}$,
with no two of these having linearly dependent differentials at any $\bx\in[0,1]^2$.
Let $\eps,\delta,N,\scripts$ be as above.
Let $M\in\naturals$ be another parameter.

Let $\phi:[0,1]^2\to\reals$ be $C^1$.
Let $\bh$ be as above. 
Define $M(\bx)$ to be the number of tuples $(s_1,s_2,s_3)$
with each $s_j\in h_j(L_j(\bx))$
that satisfy 
$|\nabla\phi(\bx)-s_j|<\eps$.
Let 
\[ E(\phi,\bh)=\{\bx: M(\bx)\ge M\}. \]
Let
\[ \Lambda(\phi) = \sup_\bh |E(\phi,\bh)|.\]

\begin{question} \label{question:sublevel2}
For $\phi$ and $\{L_j\}$ or for a class of such functions,
what upper bounds does $\Lambda(\phi)$ 
satisfy in terms of $\eps,N,M$?
\end{question}

\medskip
A multitude of variants and generalizations of sublevel set inequalities
are likely to be relevant to future investigations of oscillatory integral inequalities. 
With 
\[S(\bff,\eps) = \Big\{x\in S: \big|\sum_{j\in J} a_j(x)(f_j\circ\varphi_j)(x)\big|<\eps\Big\},\] 
the cardinality $|J|$ of the index set $J$,
the dimension $D$ of the ambient set $S$,
the dimension $d$ of the codomain of the mappings $\varphi_j$,
the nature of the coefficient functions $a_j$
(which may be scalar- or matrix-valued, $C^\omega$ or $C^\infty$, and so on),
the dimension of the codomain of vector-valued functions $f_j$
can all be varied. Mappings $\varphi_j$ that are homogeneous of degree
one with respect to a subset of the coordinates for $S$ arise naturally, 
as Cauchy-Schwarz/$TT^*$ reasoning
leads naturally to factors $f_k(\varphi_k(x'))\,\overline{f_k}(\varphi_k(x))$,
and the substitution $x' = x+t$ and Taylor expansion then lead to
$F_k(\psi_k(x,t)) = f_k(\varphi_k(x)+tD\varphi_k(x))\overline{f_k}(\varphi_k(x))$.
Sublevel sets of the more general type
\[S(\bff,\eps) = \Big\{x\in S: \big|\sum_{j\in J}\sum_{\alpha\in A} 
a_{j,\alpha}(x)(f_{j,\alpha}\circ\varphi_j)(x)\big|<\eps\Big\},\]
with $A$ another finite index set and with the mappings $\varphi_j$
independent of the index $\alpha\in A$, arise upon consideration
of the formal gradient of $\sum_{j\in J} a_j\cdot(f_j\circ\varphi_j)$.


\end{document}